\documentclass[11pt]{amsart}
\usepackage{graphicx}
\usepackage{setspace}
\usepackage{amssymb,amsmath,amsthm,amsfonts,amscd}
\usepackage{hyperref}

\usepackage{color}
\usepackage{booktabs}
\usepackage{tabularx}
\usepackage{enumitem}
\usepackage[retainorgcmds]{IEEEtrantools}
\usepackage[notcite,notref,final]{showkeys}
\usepackage[final]{pdfpages}
\usepackage{fancyhdr}

\usepackage{tikz}
\usetikzlibrary{cd, positioning, automata, arrows, shapes, matrix, arrows, decorations.pathmorphing}
\tikzset{main node/.style={circle,minimum size=1em,inner
sep=0.5pt}}
\tikzset{state node/.style={circle,draw,minimum size=2em,fill=blue!20,inner
sep=0pt}}
\tikzset{small node/.style={circle,draw,minimum size=0.5em,inner
sep=2pt,font=\sffamily\bfseries}}

\setlist[enumerate,1]{leftmargin=1.8em,label=\textup{(}\alph*\textup{)}}

\newtheorem{thm}{Theorem}[section]
\newtheorem*{thm*}{Theorem}
\newtheorem{lemma}[thm]{Lemma}
\newtheorem{prop}[thm]{Proposition}

\newtheorem{corollary}[thm]{Corollary}

\newtheorem{Definition}[thm]{Definition}
\newenvironment{definition}
{\begin{Definition}\rm}{\end{Definition}}
\newtheorem{Example}[thm]{Example}
\newenvironment{example}
{\begin{Example}\rm}{\end{Example}}

\theoremstyle{definition}
\newtheorem{remark}[thm]{\textbf{Remark}}

\newcommand{\C}{\mathbb{C}}

\newcommand{\R}{\mathbb{R}}
\newcommand{\Z}{\mathbb{Z}}

\newcommand{\cala}{\mathcal{A}}

\newcommand{\cali}{\mathcal{I}}

\newcommand{\calp}{\mathcal{P}}
\newcommand{\calr}{\mathcal{R}}

\newcommand{\vea}{\varepsilon_{ab}}
\newcommand{\ve}[1]{\varepsilon_{#1}}

\DeclareMathOperator{\image}{\mathrm{im}}
\DeclareMathOperator{\id}{\mathrm{id}}

\DeclareMathOperator{\source}{\mathrm{source}}
\DeclareMathOperator{\target}{\mathrm{target}}
\DeclareMathOperator{\length}{\mathrm{length}}

\newcommand{\ra}{\rightarrow}
\newcommand{\se}{\subseteq}
\newcommand{\ip}[1]{\langle#1\rangle}

\newcommand{\inverse}{^{-1}}

\newcommand{\abs}[1]{\lvert #1 \rvert}

\newcommand\bm[1]{\begin{bmatrix}#1\end{bmatrix}}

\newcommand{\papb}{M_\alpha M_\beta}

\newcommand{\pbpa}{M_\beta M_\alpha}

\newcommand{\pa}{M_\alpha} 
\newcommand{\pb}{M_\beta}

\newcommand{\rep}{\mathrm{rep}}
\newcommand{\repq}{\mathrm{rep}_K Q} 
\newcommand{\repqr}{\mathrm{rep}_K (Q,\mathcal{I})} 
\newcommand{\repkqr}{\mathrm{rep}_K (Q,\mathcal{I})} 
\newcommand{\repcqr}{\mathrm{rep}_\C (Q,\mathcal{I})} 
 
\newcommand{\repqrf}{\mathrm{rep}_K (Q,\mathcal{I}_f)} 
\newcommand{\brepqrf}{\mathrm{rep}_K (\bar Q,\bar\cali_f)} 

\newcommand{\jcfil}[1]{J_C^{(#1)}}

\newcommand{\unbr}{\mathrm{Unbr}}

\begin{document}

\title[Subregular $J$-rings via quivers]{Subregular $J$-rings of Coxeter systems via quiver path algebras}

\author[I. Dimitrov]{Ivan Dimitrov}
\address{Department of Mathematics and Statistics, Queen's University}
\email{dimitrov@queensu.ca}

\author[C. Paquette]{Charles Paquette}
\address{Department of Mathematics and Computer Science, Royal Military
College of Canada}
\email{charles.paquette.math@gmail.com}

\author[D. Wehlau]{David Wehlau}
\address{Department of Mathematics and Computer Science, Royal Military
College of Canada}
\email{wehlau@rmc.ca}

\author[T. Xu]{Tianyuan Xu}
\address{Department of Mathematics, University of Colorado Boulder}
\email{tianyuan.xu@colorado.edu}



\vspace{-1em}
\begin{abstract}
We study the subregular $J$-ring $J_C$ of a Coxeter system $(W,S)$, a subring of Lusztig's $J$-ring. We prove that
$J_C$ is isomorphic to a quotient of the path algebra of the double quiver of $(W,S)$ by a suitable ideal that we
associate to a family of Chebyshev polynomials. As applications, we use quiver representations to study the
category mod-$A_K$ of finite dimensional right modules of the algebra $A_K=K\otimes_\Z J_C$ over an algebraically
closed field $K$ of characteristic zero. Our results include classifications of Coxeter systems for which mod-$A_K$
is semisimple, has finitely many simple modules up to isomorphism, or has a bound on the dimensions of simple
modules. Incidentally, we show that every group algebra of a free product of finite cyclic groups is Morita
equivalent to the algebra $A_K$ for a suitable Coxeter system; this allows us to specialize the classifications to
the module categories of such group algebras.
\end{abstract}

\maketitle


\section{Introduction}
\label{sec:introduction}

We study a subring of the $J$-ring of an arbitrary Coxeter system $(W,S)$.  The $J$-ring was first introduced by
Lusztig in \cite{LusztigCell2} in the case where $W$ is a Weyl or affine Weyl group to help study the
Kazhdan--Lusztig cells in $W$.  Later, Lusztig showed in \cite{LusztigHecke} that the same construction of the
$J$-ring is valid for arbitrary Coxeter systems, at least in the so-called ``equal-parameter'' case. In the
``unequal-parameter'' case, the validity of the construction relies on what has come to be known as Lusztig's
conjectures P1-P15; see \cite[Section 14.2]{Bonnafe}. We only deal with the equal-parameter case in this paper.

By definition, the $J$-ring equals the free abelian group $J=\oplus_{w\in W}\Z t_w$ as a group, and products in $J$
are given by the formula \begin{equation} \label{eq:introJprod} t_xt_y=\sum_{z\in W}\gamma_{x,y,z\inverse} t_z
\end{equation} where each coefficient $\gamma_{x,y,z\inverse}$ is a certain nonnegative integer obtained via the
Kazhdan--Lusztig basis of the Hecke algebra of $(W, S)$.  The formula endows $J$ with the structure of an
associative (but not necessarily unital) ring.  Moreover, for each \emph{two-sided Kazhdan--Lusztig cell} $E$ of
$W$, the subgroup $J_E:=\oplus_{w\in E} \Z t_w$ of $J$ is a subring of $J$. In this paper, we focus on the ring
$J_C$ where $C$ is a particular two-sided cell of $W$ called the \emph{subregular cell}. This cell consists of all
the non-identity elements in $W$ that are \emph{rigid} in the sense that they each have a unique reduced
expression; see \cite{LusztigSubregular} and \cite{Xu}. We call $J_C$ the \emph{subregular $J$-ring} and study the
structure and representations of $J_C$.

Also called the \emph{asymptotic Hecke algebra}, the $J$-ring may be viewed as a limit of the Hecke algebra of $W$
in the sense of \cite{LusztigLimit}.  As such, the $J$-ring has been an important tool for studying Hecke algebras
and reductive groups; see, for example, \cite{LusztigCellIV}, \cite{GeckDecomposition}, \cite{GeckCellular} and
\cite{LusztigSpecialRep}.  Besides its applications, the structure of the $J$-ring itself has also been studied
extensively. Notable results include the following: Bezrukavnikov, Finkelberg, and Ostrik studied a categorical
version of the $J$-ring  in \cite{BFOIII} and used it to compute explicitly the structure of the ring $J_E$ for
each two-sided cell $E$ in $W$; Braverman and Kazhdan showed in \cite{BravermanKazhdan} that $J$ is isomorphic to a
certain subalgebra of the Harish-Chandra Schwartz algebra of a reductive group; by using a generalization of the
Robinson--Schensted algorithm called the \emph{affine matrix-ball construction}, Kim and Pylyavskyy gave a
canonical presentation for the $J$-ring in the special case where $W$ is an (extended) affine symmetric group in
\cite{MatrixBall}, extending the work of Xi in \cite{Xi} for the same case.  

It is worth noting that the results on the structure of the $J$-ring mentioned above are all restricted to Weyl or
(extended) affine Weyl groups. On the other hand, the $J$-ring makes sense for an arbitrary Coxeter system, so it
is natural to wonder what the structure of the $J$-ring is for more general Coxeter systems.  Indeed, in
Kazhdan--Lusztig theory it can often be interesting to study Coxeter systems in the full generality. One such
indication comes from the proof of the famous ``positivity conjecture'' of Kazhdan and Lusztig, which states that
all coefficients of so-called \emph{Kazhdan--Lusztig polynomials} are nonnegative integers.  After its first
appearance in \cite{KL} in 1979, the conjecture was proved along with other related deep results for Weyl and
affine Weyl groups in the next two years by Kazhdan--Lusztig \cite{KL2}, Beilinson--Bernstein
\cite{BeilinsonBernstein} and Brylinski--Kashiwara \cite{BrylinskiKashiwara}, via geometric methods involving local
intersection cohomology of Schubert varieties, $D$-modules and perverse sheaves. The proof for the general case
came much later: building upon the work of Soergel in \cite{Soergel1} \cite{Soergel2} \cite{Soergel3}, Elias and
Williamson proved the positivity conjecture for arbitrary Coxeter systems in \cite{EliasWilliamson} in 2014.  In
their work, they introduced a graphical calculus and a type of Hodge theory for the Soergel category, each of which
is interesting in its own right; see \cite{SoergelCalculus} and \cite{WilliamsonHodge}.

As was the case for the positivity conjecture, a disparity exists between what is known about the $J$-rings of Weyl
or affine Weyl groups and the $J$-rings of other Coxeter systems. With the exception of Alvis' work on the Coxeter
group of type $H_4$ in \cite{AlvisH4}, the structures of the $J$-rings of non-Weyl Coxeter groups remain largely
unexplored. One obstacle to understanding $J$-rings in general is the diffculty in computing the structure
constants of Hecke algebras with respect to their Kazhdan--Lusztig bases,  which are necessary for obtaining the
coefficients $\gamma_{x,y,z\inverse}$ in Equation \eqref{eq:introJprod}.  As we will show, however, it is possible
to circumvent this obstacle if we restrict from the $J$-ring to the subregular $J$-ring. In \cite{Xu}, the last
named author gave a description of products of the form $t_xt_y$ in $J_C$ that does not involve Kazhdan--Lusztig
theory. In the present paper, we use this description to show that $J_C$ is isomorphic to certain quotients of the
path algebra of a quiver, then use quiver representations to study representations of $J_C$.  Roughly speaking, the
reason why we can understand $J_C$ in full generality, in contrast to the entire $J$-ring, is that the rigidity of
the elements of $C$ makes $C$ and $J_C$ more amenable to combinatorial analysis. It seems interesting that a
similar contrast is also visible in the book \cite{Bonnafe} by Bonnaf\'e, where he singles out the subregular cell
in Chapters 12 and 13 (he calls the cell the \emph{submaximal cell}) and exploits its rich combinatorics in his
investigation of various Kazhdan--Lusztig objects attached to the cell, including the so-called \emph{cell module}
of $C$ and its connection to the reflection representation of the Hecke algebra of $(W,S)$.  

Let us elaborate on how $J_C$ relates to quivers. Recall that every Coxeter system $(W,S)$ corresponds to a unique
Coxeter diagram $G$ with vertex set $S$ and edge set $\{a-b:a,b\in S, m(a,b)\ge 3\}$.  We define the \emph{double
quiver} of $(W,S)$ to be the directed graph $Q=(Q_0,Q_1)$ with vertex set $Q_0=S$ and edge set $Q_1=\{a\ra b:
a,b\in S, m(a,b)\ge 3\}$, where we have a pair of arrows $a\ra b$ and $b\ra a$ arising from each edge $a-b$ in $G$.
Next, we consider the path algebra $\Z Q$ of $Q$ over $\Z$ and associate to each suitable family $\{f_n:n\in
\Z_{\ge 2}\}$ of polynomials an ideal $\cali_f^\Z$ in $\Z Q$ called an \emph{evaluation ideal} of $\{f_n: n\in
\Z_{\ge 2}\}$ (Definition \ref{def:ArrowRelations}).  Our first main result, Theorem \ref{thm:JcIso}, establishes
an algebra isomorphism between $J_C$ and the quotient $\Z Q/\cali_u^\Z$ where $\cali_u^\Z$ is the evaluation ideal
of a family $\{u_n\in \Z[x]:n\ge 2\}$ of ``Chebyshev polynomials''.

Fixing an algebraically closed field $K$ of characteristic zero, we extend the result that $J_C\cong \Z
Q/\cali_u^\Z$ in two ways (see Remark \ref{rmk:fieldAssumptions} for a discussion about assumptions on the field
$K$). First, in Theorem \ref{thm:IsoQuotients}, we show that upon an extension of scalars we may alter the family
$\{u_n\}$ without changing the isomorphism type of the quotient of the path algebra by the evaluation ideal.  More
precisely, we show that for any two \emph{uniform} families of polynomials $\{f_n\}, \{g_n\}$ over $K$ (Definition
\ref{def:DistPoly}), we have $K Q/\cali_f\cong K Q/\cali_g$ where $K Q$ is the path algebra of $Q$ and $\cali_f,
\cali_g$ are evaluation ideals of $KQ$ constructed from $\{f_n\}, \{g_n\}$.  The Chebyshev polynomials $\{u_n\}$
form a uniform family, and the result enables us to realize the algebra $A_K:=K \otimes_\Z J_C$ as a quotient $K
Q/\cali_f$ where the ideal $\cali_f$ is generated by elements which can take very simple forms; see
\autoref{sec:contractionExamples}.  Together, Theorems \ref{thm:JcIso} and \ref{thm:IsoQuotients} generalize
Example 6.10 of Diaz-Lopez's paper \cite{DiazLopez}.  That paper cites the example as its main motivation and
remarks that the example suggests a stronger connection between path algebras and asymptotic Hecke algebras. We
hope our result can be viewed as further support for such a connection. 

In our second extension of Theorem \ref{thm:JcIso}, we develop a procedure to modify a quiver $Q$ to a new quiver
$\bar Q$ such that the algebras $K Q/\cali_f$ and $K \bar Q/\bar\cali_f$  are Morita equivalent for any uniform
family of polynomials $\{f_n\}$, where $\bar\cali_f$ is the evaluation ideal of $K\bar Q$ associated to $\{f_n\}$;
see Theorem \ref{thm:MoritaEquivalence}. We call the procedure a \emph{quiver contraction} and will often apply it
iteratively, starting from the double quiver of a Coxeter diagram. Quiver contractions reveal certain interesting
algebras that are Morita equivalent to algebras $A_{K}$ associated to Coxeter systems, such as the Laurent
polynomial ring $K[t,t\inverse]$ or group algebras of free products of finite cyclic groups; see Examples
\ref{eg:LaurentPolynomials} and \ref{eg:freeProduct}.  In addition, we use quiver contractions to justify certain
assumptions on Coxeter systems in the study of representations of $A_K$. For example, for any Coxeter system whose
Coxeter diagram $G$ is a tree, quiver contractions allow us to assume that $G$ contains no simple edges when
studying representations of $A_K$; see Example \ref{eg:treeContraction}.

Theorem \ref{thm:JcIso} and its extensions allow us to study representations of the subregular $J$-ring via
quivers.  More precisely, we use the double quiver $Q$ to study the category mod-$A_K$ of finite dimensional right
modules of the algebra $A_K$.  Representations of the $J$-ring and of rings of the form $J_E$ (where $E$ is a
two-sided cell) are not only interesting on their own but also intimately related to representations of $W$ and its
Hecke algebra; see \cite{LusztigHecke}, \cite{LusztigJInvolutions}, \cite{LusztigSpecialRep}, \cite{GeckCellular}
and \cite{Pietraho}.  On the other hand, quivers arise naturally in many areas of mathematics and have close
connections to the representation theory of finite dimensional algebras, Kac--Moody algebras, quantum groups, and
so on; see \cite{SavageQuiver} and \cite{Schiffler}. 

To study mod-$A_K$ via $Q$, we use the well-known fact that for each ideal $\cali$ in $K Q$, the category of
modules of the quotient $K Q/\cali$ is equivalent to the category of \emph{representations of $Q$} that
\emph{satisfy} the relations in $\cali$ (see \autoref{sec:QuiverReps}).
Our main results are Theorems \ref{thm:repThm1} and \ref{thm:repThm2}, which characterize in terms of the Coxeter
diagram $G$ when the category mod-$A_K$ is semisimple, contains finite many simple modules, or has a bound on the
dimensions of simple modules.  In a sense, the chacterizations are similar to those of the representation types of
quivers given by the celebrated Gabriel's Theorem (see \cite{DDPW}). Since we can use quiver contractions to show
that every group algebra of a free product of finite cyclic groups is Morita equivalent to the algebra $A_K$ for a
suitable Coxeter system (Example \ref{eg:freeProduct}), Theorems \ref{thm:repThm1} and \ref{thm:repThm2} lead to
similar characterizations for the module categories of such group algebras, which may be of independent interest as
they are stated without mention of Coxeter systems or Kazhdan--Lusztig theory; see Remark \ref{rmk:freeProductRep}
and Proposition \ref{prop:freeProductRep}.

The rest of the paper is organized as follows. In Section \ref{sec:background}, we recall the relevant background
on Coxeter systems, subregular $J$-rings, path algebras, and quiver representations. In Section
\ref{sec:QuiverRealization}, we define uniform families of polynomials $\{f_n\}$ and their associated evaluation
ideals, then we realize $J_C$ and the algebra $A_K$ as quotients of path algebras by suitable evaluation ideals via
Theorems \ref{thm:JcIso} and \ref{thm:IsoQuotients}. Section \ref{sec:A-modTools} deals with quiver contractions
and its main result is Theorem \ref{thm:MoritaEquivalence}, which asserts that $K Q/\cali_f$ is Morita equivalent
to $K \bar Q/\bar\cali_f$ if the quiver $\bar Q$ is obtained from $Q$ via a sequence of contractions. We define
contractions in \autoref{sec:contractionDefinitions}, give detailed examples of contractions in
\autoref{sec:contractionExamples} and prove Theorem \ref{thm:MoritaEquivalence} in \autoref{sec:MoritaEquivalence},
then we analyze and give examples of representations of contracted quivers in \autoref{sec:repAnalysis}. Finally,
we state and prove the results on mod-$A_K$ in Section \ref{sec:mod-AResults}. Most of the examples from
\autoref{sec:contractionExamples} and \autoref{sec:repAnalysis} will be used in the proofs.

\subsection*{Acknowledgements}
The first three named authors are supported by the National Sciences and Engineering Research Council of Canada.
The second and third named authors are also supported by the Canadian Defence Academy Research Programme. We thank
R. M. Green for reading a draft of the paper and for his helpful comments.

\section{Background}
\label{sec:background}

\subsection{Coxeter Systems}
\label{sec:CoxeterSystems}

A \emph{Coxeter system} is a pair $(W,S)$ where $S$ is a finite set and $W$ is the group given by the presentation
\[
    W=\ip{S\,\vert\, (ab)^{m(a,b)}=1 \; \text{for all $a,b\in S$ with
    $m(a,b)<\infty$}},
\]
where $m$ denotes a map $m:S\times S\ra \Z_{\ge 1}\cup \{\infty\}$ such that for all $a,b\in S$, we have
$m(a,b)=m(b,a)$, and  $m(a,b)=1$ if and only if $a=b$. These conditions imply that $a^2=1$ for all $a\in S$ and
that
\begin{equation}
    \label{eq:BraidRelation}
    aba\dots=bab\dots, 
\end{equation}
where both sides contain $m(a,b)$ factors, for every two distinct generators $a,b\in S$ with $m(a,b)<\infty$. We
call each side of Equation \eqref{eq:BraidRelation} an  $\{a,b\}$-\emph{braid} and call the equation a \emph{braid
relation}.

Each Coxeter system $(W,S)$ can be encoded via its \emph{Coxeter diagram}, the weighted, undirected graph $G$ whose
vertex set is $S$, whose edge set is $\{\{a,b\}: m(a,b)\ge  3\}$, and where the weight of an edge $\{a,b\}$ is
$m(a,b)$. An edge with weight $m$ in $G$ is \emph{simple} if $m=3$ and is \emph{heavy} otherwise. When drawing $G$,
we label each edge with its weight except for simple edges.  A Coxeter system $(W,S)$ is said to be
\emph{irreducible} if its Coxeter diagram $G$ is connected and \emph{reducible} otherwise. 

For the rest of the paper, we let $(W,S)$ be an irreducible Coxeter system and let $G$ be its Coxeter diagram. The
irreduciblity assumption is made to simplify our statements, as the reducible case can be easily derived from the
irreducible case for all the relevant results; see Remark \ref{rmk:IrreducibleAssumption}.

\subsection{The Subregular J-ring}
\label{sec:SubregularJ}

Let $S^*$ be the free monoid generated by $S$. For each element $w\in W$, the words in $S^*$ that express $w$ and
have minimal length are called the \emph{reduced words} of $w$. The common length of these words, denoted $l(w)$,
is called the \emph{length} of $w$. By the well-known Matsumoto--Tits theorem, every two reduced words of $w$ can
be obtained from each other via a finite sequence of braid relations.

An element in $W$ is called \emph{rigid} if it has a unique reduced word. In this paper we are particularly 
interested in the set
\[
    C=\{w\in W: w\neq 1, w \text{\; is rigid}\}. 
\]
 The set $C$ is known to be a \emph{two-sided
 Kazhdan--Lusztig cell} of $W$, and is called the \emph{subregular cell} or
 \emph{submaximal cell} of $W$ (see \cite{Xu} and \cite[Chapter 12]{Bonnafe}). 

\begin{remark}
(a) By the Matsumoto--Tits theorem, a word $w\in S^*$ expresses an element in $C$ if and only if $w$ is nonempty
and does not contain as a contiguous subword a word of the form $aa$ for any $a\in S$ or an $\{a,b\}$-braid for any
distinct elements $a,b\in S$. 

(b) Henceforth we will identify each element $w\in C$ with its unique reduced word. In particular, we will also use
$w$ to denote the reduced word of the element (as in Propositions \ref{prop:JcProducts} and \ref{prop:lower}, for
example).  \label{rmk:RigidityCriterion} 
\end{remark}

To define the subregular $J$-ring, we first recall the construction of the \emph{$J$-ring}, or the \emph{asymptotic
Hecke algebra}, of $(W,S)$. The construction is due to Lusztig, who defined the $J$-ring as the free abelian group
$J:=\oplus_{w\in W} \Z t_w$ and defined multiplication in $J$ by the formula 
\[ 
    t_x t_y =\sum_{z\in W}\gamma_{x,y,z\inverse} t_z 
\] 
where each coefficient $\gamma_{x,y,z\inverse}$ is a certain nonnegative integer extracted from the structure
constants for the \emph{Kazhdan--Lusztig basis} of the \emph{Iwahori--Hecke algebra} of $(W,S)$; see
\cite{LusztigCell2} and \cite[Section 18.3]{LusztigHecke}.  Lusztig showed that for each two-sided cell $E$ of $W$,
the subgroup $J_E:=\oplus_{w\in E} \Z t_w$ is in fact a subring of $J$. We define the \emph{subregular $J$-ring} to
be the subring $J_C$ of $J$ arising from the subregular cell $C$ of $W$.

While the definition of $J$ relies heavily on Kazhdan--Lusztig theory, it is shown in \cite{Xu} that we can
describe products in the subregular $J$-ring via simple manipulations of reduced words. To do so, for each pair of
distinct generators $a,b\in S$, let us call an element $w\in C$ an \emph{$\{a,b\}$-element} if $w$ lies in the
subgroup of $W$ generated by $a$ and $b$. For two words $x=\dots a_2a_1, y=b_1b_2\dots\in S^*$ with $a_1=b_1$, let
$x*y$ be the word $\dots a_2b_1b_2\dots$, the result of concatenating $x$ and $y$ and deleting one duplicate copy
of the letter $a_1=b_1$. Then products in $J_C$ behave as follows:

\begin{prop}[{\cite[Corollary 4.2, Propositions 4.4 \& 4.5]{Xu}}]
    Let $x,y$ be elements of $C$ with reduced words $x=\dots a_2a_1$ and
    $y=b_1b_2\dots$, where we take $a_2$ and $b_2$ to be nonexistant when
    $l(x)=1$ and $l(y)=1$, respectively. Then the following holds.
    \begin{enumerate}
        \item If $a_1\neq b_1$, then $t_xt_y=0$.
        \item If $a_1= b_1$ and $a_2\neq b_2$ \textup{(}including the vacuous
            cases where $a_2$ or $b_2$ do not exist\textup{)}, then $t_{x}t_y=t_{x*y}$.
        \item If $a_1=b_1$ and $x,y$ are both $\{a,b\}$-elements for
            some $a,b\in S$, then $t_xt_y$ is a linear combination of
            the form $\sum_{z\in Z} t_z$ where $Z$ is a certain set of
            $\{a,b\}$-elements.
    \end{enumerate}
    \label{prop:JcProducts}
\end{prop}
\noindent Note that the first two parts of the proposition imply that $J_C$ has a unit, namely, the element
$\sum_{a\in S}t_a$. In the last part, the set $Z$ can be obtained via a \emph{truncated Clebsch--Gordan rule}, but
the exact description of $Z$ is not essential to this paper, so we omit it. Instead, we describe below the product
$t_xt_y$ from Proposition \ref{prop:JcProducts}.(c) in the special case where $l(x)=2$. The special case is in fact
equivalent to the general case because one can deduce the latter from the former by induction.

\begin{prop}
    [{\cite[Corollary 4.2]{Xu}}]
        \label{prop:DihedralJcProducts}
        Let $a,b\in S$ and let $m=m(a,b)$. Suppose that $m\ge 3$. For all
        $1<i<m$, let $w_{a,i}$ be the $\{a,b\}$ element $aba\dots$ of
        length $i$ and let $t_{a,i}=t_{w_{a,i}}$, then define $w_{b,i}=bab\dots$ and $t_{b,i}$ similarly. Then for all $1<i<m$, we have
    \begin{equation}
        \label{eq:tRecursion}
    t_{ab}t_{b,i}= 
    \begin{cases}
        t_{a,i-1}+t_{a,i+1} & \text{if $i< m-1$};\\
        t_{a,i-1} & \text{if $i=m-1$}.
    \end{cases}
    \end{equation}
\end{prop} 

The following example illustrates how Proposition \ref{prop:JcProducts} can be used to compute the product $t_xt_y$
for all $x,y\in C$: Suppose $(W,S)$ is a Coxeter system where $S=\{a,b,c\}$ and $m(a,b)=3, m(a,c)=4, m(b,c)=5$. Let
$x=abcb, y=bcbcac$. Then $x,y\in C$ by Remark \ref{rmk:RigidityCriterion}.(a).  The first two parts of Proposition
\ref{prop:JcProducts} imply that $t_y t_x=0$ and 
\[
    t_x t_y = (t_{ab}t_{bcb})(t_{bcbc}t_{cac})=t_{ab}(t_{bcb}t_{bcbc})t_{cac}.
\]
The product $t_{bcb}t_{bcbc}$ can be computed using Part (c) and turns out to
equal $t_{bc}$. Applying Part (b) again completes the computation:
\[
t_x t_y = t_{ab}t_{bc}t_{cac}=t_{abcac}.
\]
Intuitively, as the example shows, the reductions allowed by the first two parts of Proposition
\ref{prop:JcProducts} mean that the most interesting multiplication in $J_C$ happen ``locally'', for elements
within subgroups of $W$ generated by two elements. This fact is a key reason why Theorem \ref{thm:JcIso} holds.

\subsection{Path Algebras}
\label{sec:PathAlgebras}
In this and the next subsection, we recall the background on quivers, path algebras and quiver representations that
is relevant to the paper. Our main reference is \cite{Schiffler}. 

A \emph{quiver} is a directed graph $Q=(Q_0,Q_1)$ where $Q_0$ is the set of vertices and $Q_1$ is the set of
directed edges, or \emph{arrows}. The sets $Q_0$ and $Q_1$ will be finite for all quivers in this paper. For each
arrow $\alpha: a\ra b$, we call $a$ and $b$ the \emph{source} and the \emph{target} of $\alpha$ and denote them by
$\source(\alpha)$ and $\target(\alpha)$, respectively. An arrow $\alpha$ is called a \emph{loop at $a$} if
$\source(\alpha)=\target(\alpha)=a$. 

A \emph{path} on $Q$ is an element of the form $p=\alpha_1\alpha_2\dots\alpha_n$ where
$\target(\alpha_i)=\source(\alpha_{i+1})$ for all $1\le i\le n-1$; we define the \emph{source} of $p$ to be
$\source(p):=\source(\alpha_1)$ and the \emph{target} of $p$ to be $\target(p):=\target(\alpha_n)$.  To each vertex
$a\in Q_0$, we associate a special path $e_a$ called the \emph{stationary path at $a$}; we consider it as a path
that ``stays at $a$'', so in particular we have $\source(e_a)=\target(e_a)=a$. The \emph{length} of the path $p$,
denoted by $\mathrm{length}(p)$, is defined to be the number of arrows it traverses. In other words, each arrow has
length 1, each stationary path has length 0, and we have $\length(p)=\sum_{i=1}^n \length(\alpha_i)$ for each path
$p=\alpha_1\dots\alpha_n$.

Let $\calp$ be the set of all paths on $Q$, and let $R$ be a commutative ring.  The \emph{path algebra of $Q$ over
$R$}, denoted by $RQ$, is the $R$-algebra with $\calp$ as an $R$-basis and with multiplication induced by path
concatenation: for paths $p=\alpha_1\dots\alpha_m,q=\beta_1\dots\beta_n\in \calp$, we define $pq$ to be the path
$\alpha_1\dots\alpha_m\beta_1\dots\beta_n$ if $\target(p)=\source(q)$ and to be 0 otherwise. In particular, for any
path $p$ with source $a$ and target $b$, we have $e_ap=p=pe_b$ in $RQ$.  Consequently, $RQ$ contains the unit
$1=\sum_{a\in Q_0}e_a$, and we can describe $RQ$ as the algebra generated by the arrows and stationary paths in $Q$
subject only to the relations $e_ae_b=\delta_{a,b}e_a$ for all $a,b\in Q_0$ and $e_a\alpha=\alpha=\alpha e_b$ for
each arrow $\alpha: a\ra b$ in $Q_1$.

Among the elements of $RQ$, we will be especially interested in elements of the form $r=\sum_{p} c_p p\in RQ$ where
the sum is taken over a finite set of paths on $Q$ which share the same source and the same target. Following
\cite[Definition 3.1]{Schiffler}, we call such an element $r$ a \emph{uniform relation} or simply a \emph{relation}
on $Q$. We define the source and target of $r$ to be the common source and common target of the paths appearings in
it, respectively. Our first main theorem, Theorem \ref{thm:JcIso}, asserts that $J_C\cong \Z Q/\cali_f$ for a
suitable quiver $Q$ and a suitable ideal $\cali_u^\Z$ generated by a set of relations of the form
$\calr=\{r_u(\alpha):\alpha\in Q_1\}$, where each relation corresponds to an arrow in $Q$. 

\subsection{Quiver Representations}
\label{sec:QuiverReps}

Let $Q$ be a quiver and let $K$ be an arbitrary field.  We recall below some basic facts about the representation
theory of the path algebra $KQ$ and its quotients. All representations and modules we mention in this paper will be
finite dimensional.  

Let mod-$KQ$ be the category of finite dimensional right $KQ$-modules. It is well-known that mod-$KQ$ is naturally
equivalent to the category $\repq$ of finite dimensional representations of $Q$ over $K$. Here, a
\emph{representation} of a quiver $Q$ over $K$ is an assignment 
\[ M=(M_a,M_\alpha)_{a\in Q_0, \alpha\in Q_1} \]
of a $K$-vector space $M_a$ to each vertex $a$ of $Q$ and a linear map $M_\alpha: M_a\ra M_b$ for each arrow
$\alpha: a\ra b$ in $Q$; the \emph{dimension} of $M$ is defined by $\dim(M):=\sum_{a\in Q_0} \dim (M_a) $.  A
morphism $\varphi: M\ra N$ between two representations $M, N$ of $Q$ consists of the data
$\varphi=(\varphi_a)_{a\in Q_0}$ of linear maps $\varphi_a: M_a \ra N_a$ for $a\in Q_0$ such that $\varphi_b\circ
M_\alpha= N_\alpha\circ \varphi_a$ for every arrow $\alpha: a\ra b$ in $Q$. The equivalence between the two
categories can be established by two naturally defined quasi-inverse functors $\mathcal{F}: \textrm{mod-}KQ \ra
\repq$ and $\mathcal{G}: \repq \ra \textrm{mod-}KQ$; see \cite[Chapter 5]{Schiffler}. 

We can modify the equivalence between mod-$KQ$ and $\repq$ to account for relations on $Q$. To do so, for each
representation $M$ of $Q$, we set $M_{e_a}=\id_{M_a}$ for all $a\in Q_0$ and associate to each path
$p=\alpha_1\dots\alpha_n$ on $Q$ the map 
\[
    M_p:=  M_{\alpha_n}\circ\dots\circ M_{\alpha_1},
\]
and we say that $M$ \emph{satisfies} a relation $r=\sum c_p p$ if $\sum c_pM_p=0$. For an ideal $\cali$ of $KQ$
generated by a set of relations $\calr$, define a representation of $Q$ to be a \emph{representation of
$(Q,\cali)$} if it satisfies all relations in $\calr$. Finally, let $\repqr$ be the full subcategory of $\repq$
whose objects are the representations of $(Q,\cali)$.  Then it is well-known that $\repqr$ is equivalent to
mod-$KQ/\cali$, the category of finite dimensional right modules of the quotient $KQ/\cali$.

\begin{remark}
    \label{rmk:shorthandNotation}
We introduce two types of shorthand notation to be used for the rest of the paper.  First, for a category
$\mathbf{C}$, we will write $M\in \mathbf{C}$ to mean that $M$ is an object in $\mathbf{C}$. Second, given a
two-sided ideal $I$ in a ring $R$ and an element $r\in R$, we will denote the coset $r+I$ simply by $r$.  
\end{remark}

Familiar notions from mod-$KQ$ have obvious counterparts in $\repq$: The \emph{zero representation} in $\repq$ is
the representation $M$ with $M_a=0$ for all $a\in Q_0$. A \emph{subrepresentation} of a representation $M$ is an
assignment $N=(N_a,N_\alpha)_{a\in Q_0,\alpha\in Q_1}$ such that for every arrow $\alpha: a\ra b$ in $Q$, we have
$N_a\se M_a$, $M_\alpha(N_a)\se N_b$, and $N_\alpha$ equals the restriction of $M_\alpha$ to $N_a$. A
representation is \emph{simple} if it does not contain any proper, nonzero subrepresentation.  The \emph{direct
sum} of two representations $M,N$ is the reprentation $M\oplus N$ where $(M\oplus N)_a=M_a\oplus N_a$ and $(M\oplus
N)_\alpha((m,n))=(M_\alpha(m),N_\alpha(n))$ for every arrow $\alpha: a\ra b$ and every element $(m,n)\in M_a\oplus
N_a$.  Finally, a representation is \emph{semisimple} if it is a direct sum of simple representations, and each of
$\repq$ and $\repqr$ is \emph{semisimple} if all representations in it are semisimple.  These notions agree with
their counterparts in mod-$KQ$ under the equivalences $\mathcal{F}$ and $\mathcal{G}$. For example, a
representation $M\in \repq$ is simple if and only if the module $\mathcal{G}(M)\in \text{mod-}KQ$ is simple, and
$\repqr$ is semisimple if and only if mod-$KQ/\cali$ is semisimple. Indeed, the agreement of the notions can be
attributed to the facts that mod-$KQ$ and $\repq$ are abelian categories, that the definitions in $\repq$ and
mod-$KQ$ are specializations of the corresponding categorical notions, and that $\mathcal{F}, \mathcal{G}$ are
equivalences of abelian categories.

\section{{Quiver Realizations}}
\label{sec:QuiverRealization}
Henceforth, let $K$ be an algebraically closed field of characteristic zero, let $(W,S)$ be a Coxeter system, and
let $G, C$ and $J_C$ be the Coxeter diagram, subregular cell and subregular $J$-ring of $(W,S)$, respectively. Let
$A=A_K:=K\otimes_\Z J_C$. In this section, we associate a quiver $Q$ to $(W,S)$ and then show that $J_C\cong \Z
Q/\cali_u^\Z$ and $A\cong KQ/\cali_f$ for suitable ideals $\cali_u^\Z\se \Z Q$ and $\cali_f\se KQ$. By
\autoref{sec:QuiverReps}, the latter isomorphism will allow us to study the category $\mathrm{rep}_K A$ via the
equivalent category $\mathrm{rep}_K(Q,\cali_f)$.

\subsection{Statement of Results}
\label{sec:IsoStatements}
Let $Q=(Q_0,Q_1)$ be the quiver with $Q_0=S$ and $Q_1=\{(a,b):a,b\in S, m({a,b})\ge 3\}$. Each edge $a-b$ in the
Coxeter diagram $G$ gives rise to a pair of arrows $a\ra b$ and $b\ra a$ in $Q$, and all arrows of $Q$ arise this
way.  For an arrow $\alpha: a\ra b$ in $Q$, we call the arrow $b\ra a$ arising from the same edge in $G$ the
\emph{dual arrow of $\alpha$} and denote it by $\bar \alpha$; we define the \emph{weight} of $\alpha$ to be
$m(a,b)$ and denote it by $m_\alpha$. We call the quiver $Q$ the \emph{double quiver of $(W,S)$} or the
\emph{double quiver of $G$}.

The ideal $\cali_u^\Z$ for which $J_C\cong \Z Q/\cali_u^\Z$ is generated by a set of (uniform) relations obtained
via \emph{arrow evaluations} of polynomials from suitable polynomial families.  We first define arrow evaluations:

\begin{definition}
    For each arrow $\alpha$ in $Q$, let $\mathrm{Eval}_\alpha: K[x]
    \ra KQ$ be the unique $K$-linear map such that
    $\mathrm{Eval}_\alpha(1)=e_a$ where $a=\source(\alpha)$ and
    \[\mathrm{Eval}_\alpha(x^n)=\alpha\bar\alpha\alpha\dots,\]
    the product with $n$ factors that start with $\alpha$ and alternate in
    $\alpha$ and $\bar\alpha$, for all $n>0$. For each polynomial $f\in
    K[x]$, we write 
    \[
        f(\alpha,\bar\alpha):=\mathrm{Eval}_\alpha(f)
    \]
    and call $f(\alpha,\bar\alpha)$ the \emph{$\alpha$-evaluation of $f$}.
\label{def:ArrowEval} 
\end{definition}

By a ``polynomial family'' we mean a countable collection $\{f_n:n\in\Z_{\ge 2}\}$ of polynomials in $K[x]$.  Note
that for $f(\alpha,\bar\alpha)$ to yield a uniform relation on $Q$, the polynomial $f$ needs to be either even or
odd, therefore we will consider only polynomial families $\{f_n\}$ where each $f_n$ is either an even or an odd
polynomial.  To describe further conditions we would like to impose on $\{f_n\}$, we need more notation:

\begin{definition}
    \label{def:fTilde}
    For each even polynomial $f=\sum c_i x^{2i} \in K[x]$, let 
    \[
        \tilde f=\sum {c_i} x^i;
    \]
    for each odd polynomial $f=\sum c_i x^{2i+1}\in K[x]$, let 
    \[
       \tilde f=\sum {c_i} x^i.
    \]
\end{definition}

Note that when $f$ is an even or odd polynomial of degree $n$, the polynomial $\tilde
f$ has degree $\lfloor{n/2}\rfloor$ where $\lfloor{-}\rfloor$ denotes the floor
function; moreover, we have
\begin{equation}
    f(\alpha,\bar\alpha)=
    \begin{cases}
       \tilde f(\alpha\bar\alpha)  & \text{if $f$ is even};\\
       \tilde f(\alpha\bar\alpha)\cdot \alpha=\alpha\cdot \tilde f(\bar\alpha\alpha) & \text{if
       $f$ is odd},
        \end{cases}
    \label{eq:TildeEval}
\end{equation}
\noindent where we evaluate a constant term $c$ in $\tilde f$ to $ce_a$ for $a=\source(\alpha)$. For example, if
$f=x^3-2x$ then $\tilde f=x-2$ and $f(\alpha,\bar\alpha)= \alpha\bar\alpha\alpha-2\alpha$, and if $f=x^4-1$ then
$\tilde f=x^2-1$ and $f(\alpha,\bar\alpha)=\alpha\bar\alpha\alpha\bar\alpha-e_a$ where $a=\source(\alpha)$. 

We are ready to define the polynomial families we need.

\begin{definition} 
\label{def:DistPoly} 
A \emph{uniform} family of polynomials (over $K$) is a set \[
\{f_n\in K[x]:n\in \Z_{\ge 2}\}\]
such that for all $n\in \Z_{\ge 2}$, we have
\begin{enumerate} 
\item $f_n$ has degree $n$, is even when $n$
    is even, and is odd when $n$ is odd. 
\item zero is
not a root of $\tilde f_n$, and no root of $\tilde f_n$ is repeated.  
\end{enumerate}
\end{definition}

Given a uniform  polynomial family, we assign one relation to each arrow and define an ideal $\cali_f$ of $KQ$ as
follows:

\begin{definition}
    \label{def:ArrowRelations} 
    Let $\{f_n: n\ge 2\}$ be a uniform  family of polynomials.
    \begin{enumerate}
        \item For each arrow $\alpha$ in $Q$, we set $m=m_\alpha$ and define
            \begin{equation}
        r_f(\alpha)=\begin{cases}
        0 & \text{if $m=\infty$};\\
        f_{m-1}(\alpha, \bar\alpha) & \text{if $m<\infty$}.
        \end{cases}
        \label{eq:r_fDef}
    \end{equation}
\item
We define the \emph{evaluation ideal of $\{f_n\}$} to be the two-sided ideal 
\[
    \cali_f :=\ip{r_f(\alpha):\alpha\in Q_1}
\]
of $KQ$ generated by the relations of the form $r_f(\alpha)$.    More generally, if $f_n\in R[x]$ for all $n\ge 2$
for some subring $R$ of $K$, we define $\cali_f^R$ to be the two-sided ideal of $RQ$ given by 
\[
    \cali_f^R:= \ip{r_f(\alpha): \alpha\in Q_1}\se RQ.  
\] 
    \end{enumerate}
   \end{definition}

\begin{example}
    \label{eg:Cheb}
    Suppose $K=\C$, and 
    consider the polynomials $u_n$ for $n\ge 0$ where 
    \begin{equation}
        \label{eq:Cheb}
        u_0=1, \quad u_1=x, \quad\text{and\; }  u_n=xu_{n-1}-u_{n-2}\;
        \text{\;for all
        $n\ge 2$}.
    \end{equation}
    These polynomials are normalizations of the \emph{Chebyshev polynomials of
    the second kind}. It is easy to see by induction that for each $n\ge 2$,
    the polynomial $u_n$ has degree $n$, is even when $n$ is even, and is odd
    when $n$ is odd.  Moreover, it is known that $u_n$ has $n$ distinct nonzero real roots $z_1,\dots,z_n$ where 
    \[
        z_i=2\cos\left(\dfrac{i\pi}{n+1}\right)
    \]
    for each $i$. The definition of the polynomial $\tilde u_n$ implies that
    $\tilde u_n$ has $\lfloor \frac{n}{2}\rfloor$ distinct nonzero roots,
    namely, the numbers $z_i^2$ where $1\le i\le\lfloor\frac{n}{2}\rfloor$. It
    follows that $\{u_n:n\ge 2\}$ forms a uniform  family of polynomials over $\C$.
    Note that $u_n\in \Z[x]$ for all $n\ge 2$, so $\cali_u^{\Z}$ makes sense as
    an ideal of $\Z Q$.  
\end{example}

We state our first two results below.

\begin{thm}
    \label{thm:JcIso}
Let $\{u_n:n\in \Z_{\ge 2}\}$ be as in Example \ref{eg:Cheb}. Then $J_C\cong \Z Q/\cali_{u}^\Z$ as unital rings.  
\end{thm}

\begin{thm}
    \label{thm:IsoQuotients}
    Let $K$ be an algebraically closed field of characteristic zero and let $\{f_n\}, \{g_n\}$ be two
    uniform  families of polynomials over $K$. Then $KQ/\cali_{f}\cong
    KQ/\cali_{g}$ as $K$-algebras.
\end{thm}

\begin{remark}
We can now explain why it suffices to deal with only irreducible Coxeter systems in this paper. Recall that if
$(W,S)$ is reducible, then the connected components of its Coxeter diagram are the diagrams of Coxeter systems
$(W_i,S_i)$ for $1\le i\le k$ for some $k\ge 2$, and we have $S=\sqcup_i S_i, W=\Pi_{i}W_i$, where the symbols
$\sqcup$ and $\Pi$ denote disjoint union and direct product, respectively. Now let $C(i), Q(i)$ be the subregular
cell and the double quiver of $(W_i, S_i)$ for each $i$.  Then $C= \sqcup_i C(i)$ by definition and $J_C=\Pi_i
J_{C(i)}$ by Part (a) of Proposition \ref{prop:JcProducts}. On the other hand, for any uniform polynomial family
$\{f_n\}$ over $K$ where $f_n\in \Z[x]$ for all $n$ (such as $\{u_n\}$) it is easy to see that $\Z
Q/\cali_f^\Z=\Pi_i \Z Q(i)/\cali_f^\Z(i)$ and $KQ/\cali_f\cong \Pi_i KQ(i)/\cali_f(i)$, where for each $i$ the
ideals $\cali_f^\Z(i)$ and $\cali_f(i)$ are the evaluation ideals of $\{f_n\}$ in $\Z Q(i)$ and $K Q(i)$,
respectively. It follows that we can deduce Theorem \ref{thm:JcIso} and Theorem \ref{thm:IsoQuotients} for
reducible Coxeter systems from the irreducible cases by taking suitable direct products.
\label{rmk:IrreducibleAssumption}
\end{remark}

\subsection{Proof of Theorem \ref{thm:JcIso}}
\label{sec:JcIsoProof}

In this section we prove Theorem \ref{thm:JcIso} by constructing an explicit isomorphism $\bar\varphi: \Z
Q/\cali_u^\Z \ra J_c$. To connect the two sides of the isomorphism, first observe that given any element
$w=s_1s_2\dots s_k\in C$, we must have $m(s_i,s_{i+1})\ge 3$ for all $1\le i\le k-1$: otherwise we can exchange
$s_i$ and $s_{i+1}$ to obtain another reduced word of $w$, contradicting the fact that $w$ is rigid. It follows
that the quiver $Q$ contains an arrow $\alpha_i: s_i \ra s_{i+1}$ for all $1\le i\le k$ as well as the path 
\[
    p_w:=\alpha_1\alpha_2\cdots \alpha_{k-1}.
\]
Recall the notation $\calp$ for the set of all paths on $Q$, and consider the map $\iota: C\ra \calp$ which sends
$w$ to $p_w$ for all $w\in C$. For each arrow $\alpha: a\ra b$ in $Q$ with $m_\alpha<\infty$, let $
p_\alpha:=\alpha\bar\alpha\alpha\dots $ be the path of length $m_\alpha-1$ obtained by concatenating $\alpha$ and
$\bar\alpha$ repeatedly. Define a path $p\in \calp$ to be \emph{unbraided} if it does not contain $p_\alpha$ as a
subpath, i.e., if we cannot write $p=p_1p_\alpha p_2$ for some paths $p_1,p_2\in \calp$, for all $\alpha\in Q_1$
with $m_\alpha<\infty$. Let $\mathrm{Unbr}(Q)$ be the set of unbraided paths in $\calp$. Then by Remark
\ref{rmk:RigidityCriterion}.(a), the image of $\iota$ is exactly $\mathrm{Unbr}(Q)$. Since $\iota$ is clearly
injective, it gives a bijection from $C$ to $\unbr(Q)$. We will henceforth use $\iota$ exclusively to denote this
bijection. The definitions and notation of this paragraph are inspired by those from \cite[Chapter 12]{Bonnafe},
where the bijection $\pi: \unbr(Q)\ra C$ is essentially the inverse of $\iota$.

Having connected $C$ to $\Z Q$, let us next consider the effect of quotienting $\Z Q$ by the ideal $\cali_u^\Z$.
Let $\alpha\in Q_1$ and let $m=m_{\alpha}$.  Since the polynomial $u_{m-1}$ has degree $(m-1)$, the relation
$r_f(\alpha)=r_{m-1}(\alpha,\bar\alpha)\in \cali_u^\Z$ must be a linear combination of the alternating path
$q:=\alpha\bar\alpha\alpha\cdots$ of length $m-1$ and strictly shorter, unbraided paths in $\Z Q$. Since $\alpha$
is arbitrary, it follows that  modulo $\cali_u^\Z$ we can rewrite every path as a linear combination of unbraided
paths.  In other words, every element in the quotient $\Z Q/\cali_u^\Z$ can be represented in the form $\sum_{p\in
\unbr(Q)}c_pp$ where the coefficients $c_p\in \Z$ are zero for all but finitely many paths. 

The final tool we need concerns a natural filtration of $J_C$.  For each $i\in \Z_{\ge 0}$, let $ C^{(i)}=\{w\in C:
l(w)\le i+1\} $ and let $J_C^{(i)}=\oplus_{w\in C^{(i)}} \Z t_w $. As the example at the end of
\autoref{sec:SubregularJ} illustrates, Propositions \ref{prop:JcProducts} and \ref{prop:DihedralJcProducts} imply
that given elements $x,y\in C$ with length $l(x)=p+1$ and $l(y)=q+1$ for some $p,q\ge 0$, the product $t_{x}t_y$ is
always a linear combination of terms of the form $t_{z}$ where $l(z)\le p+q+1$. It follows that the filtration 
\begin{equation}
    \label{eq:jcfil}
    0\se \jcfil{0} \se \jcfil{1}\se \dots.
\end{equation}
\noindent
equips $J_C$ with the structure of a filtered algebra. The same propositions also imply the following result. 

\begin{prop}
    \label{prop:lower}
    Let $w=s_1s_2\dots s_k\in C$. Then in $J_C$, we have
    \[
        t_{s_1s_2}t_{s_2s_3}\dots t_{s_{k-1}s_k}\in t_w + J_C^{(k-2)}
    \]
    where $t_w+J_C^{(k-2)}=\{t_w+ z: z\in J_C^{(k-2)}\}$. In other words, the
    product is the sum of $t_w$ and a linear combination of terms $t_y$ for
    which $l(y)<k$. 
\end{prop}
\noindent
This proposition will be useful for proving that the map $\bar\varphi: \Z Q/\cali_u^\Z \ra J_C$ is an isomorphism,
because we will examine several outputs of the map $\bar\varphi$ which have the form $t_{s_1s_2}t_{s_2s_3}\dots
t_{s_{k-1}s_k}$.  Rather than giving a formal proof of it, however, let us only sketch the main ideas needed with
an example. The proposition follows from repeated application of Proposition \ref{prop:DihedralJcProducts} in the
special case that $w$ is an $\{a,b\}$-element for some $a,b\in S$. The general case then reduces to the special
case in the way illustrated by the following example: suppose $w=abacacb\in C$ for some Coxeter system and let $T=
t_{ab}t_{ba}t_{ac}t_{ca}t_{ac}t_{cb}$.  Then $k=7$, and by the special case we have
\[
  T=  (t_{ab}t_{ba})(t_{ac}t_{ca}t_{ac})(t_{cb})\in \left(t_{aba}+
    J_C^{(1)}\right)\left(t_{acac}+J_C^{(2)}\right)\left(t_{cb}+J_C^{(0)}\right)
\] 
where the factors in parentheses correspond to the longest ``dihedral'' subwords $aba, acac, cb$ of $w$. The
filtration \eqref{eq:jcfil} implies that all terms $t_w$ with $l(w)=k$ which appear in $T$ must come from the
product $t_{aba}t_{acac}t_{cb}$, where each factor is the ``highest degree part'' in a pair of parentheses. This
product is nothing but $t_w$ by Proposition \ref{prop:JcProducts}.(b), therefore 
\[
   T \in t_{aba}t_{acac}t_{cb}+J_C^{(5)}=t_w +
   J_C^{(k-2)},
\]
as desired.

We are ready to prove Theorem \ref{thm:JcIso}.  Roughly speaking, the isomorphism holds for two main reasons:
first, as we mentioned in \autoref{sec:SubregularJ}, all interesting multiplications in $J_C$ happen ``locally''
along individual edges of the Coxeter diagram, just as the relations generating $\cali_u^\R$ are defined in the
same fashion; second, via arrow evaluations, the recursion from Equation \eqref{eq:tRecursion} which controls the
local multiplication in $J_c$ ``agrees with'' the recursive definition of $\{u_n\}$ which controls the generators
of $\cali_u^\Z$. We make these remarks more precise in the following proposition, where Theorem \ref{thm:JcIso}
appears as its last assertion.

\begin{prop} 
    Let $(W,S)$ be an irreducible Coxeter system and let $Q$ be its double
    quiver. 
    \begin{enumerate}
    \item 
There exists a unique algebra homomorphism
$\varphi: \Z Q \ra J_C$ such that for every pair of dual arrows $\alpha: a\ra
b$ and $\beta: b\ra a$ in $Q$, we have
\begin{equation}
        \varphi(e_a)= t_{a}, \quad \varphi(e_b)=t_b, \quad
        \varphi(\alpha)=t_{ab},\quad \varphi(\beta)=t_{ba}. 
    \label{eq:varphi}
\end{equation}
Moreover, for all $1\le i\le m:=m(a,b)$, we have 
    \begin{equation}
        \varphi(u_{i-1}(\alpha,\beta))=
        \begin{cases}
            t_{a,i} & \text{if $i<m$};\\
            0 & \text{if $i=m<\infty$}
        \end{cases}
        \label{eq:u1}
    \end{equation}
    and similarly
    \begin{equation}
        \varphi(u_{i-1}(\beta,\alpha))=
        \begin{cases}
            t_{b,i} & \text{if $i<m$};\\
            0 & \text{if $i=m<\infty$},
        \end{cases}
        \label{eq:u2}
    \end{equation}
    where $w_{a,i}, w_{b,i}, t_{a,i}, t_{b,i}$ are as in Proposition
    \ref{prop:DihedralJcProducts}.
\item The map $\varphi$ factors through the ideal $\cali_u^\Z$ and induces a
    homomorphism $\bar\varphi: \Z Q/\cali_u^\Z\ra J_C$ given by 
    $
    \bar\varphi(p) = \varphi(p)
$ 
for all $p\in
    \unbr(Q)$.  
\item The map $\bar\varphi$ is a unital algebra isomorphism,
    therefore $J_C\cong \Z Q/\cali_u^\Z$.  
    \end{enumerate} 
\end{prop}

\begin{proof}
(a)  Recall from Section \ref{sec:PathAlgebras} that $\Z Q$ is generated by the arrows and stationary paths of $Q$
subject only to the relations $e_ue_v=\delta_{u,v}e_u$ for all $u,v\in Q_0$ and $e_a\alpha=\alpha=\alpha e_b$ for
every arrow $\alpha: a\ra b$ in $Q_1$.  On the other hand, in $J_c$ we have $t_ut_v=\delta_{u,v}t_u$ for all
$u,v\in Q_0$ and $t_at_{ab}= t_{ab}=t_{ab}t_b$ for every arrow $\alpha: a\ra b$ in $Q_1$ by Proposition
\ref{prop:JcProducts}.  Thus, the relations satisfied by the generators of $\Z Q$ are respected in the assignment
$e_a\mapsto t_a, \alpha\mapsto t_{ab}$ for all arrows $\alpha:a \ra b$ in $Q_1$.  It follows that this assignment
extends to a unique algebra homomorphism $\varphi: \Z Q \ra J_C$ which satisfies Equation \eqref{eq:varphi}. The
homomorphism is unital since
\[
    \varphi(1)=\varphi\left(\sum_{a\in Q_0} e_a\right)=\sum_{a\in
Q_0}\varphi(e_a)=\sum_{a\in Q_0} t_a=1.
\]
Note that for each $w=s_1s_2\dots s_k\in C$, Equation \eqref{eq:varphi} and the fact that $\varphi$ is a
homomorphism imply that $\varphi(p_w)$ is exactly the element $t_{s_1s_2}t_{s_2s_3} \dots t_{s_{k-1}s_k}$. It
follows from Proposition \ref{prop:lower} that
\begin{equation}
    \varphi(p_w)\in t_w+ J_C^{(l(w)-2)}. \quad
    \label{eq:top}
\end{equation}

To prove Equations \eqref{eq:u1} and \eqref{eq:u2}, we induct on $i$. For $i\le 2$, the equations follow from the
definition of $\varphi$ and the fact that $u_0=1, u_1=x$. For $i>2$, the recursion $u_{i}=xu_{i-1}-u_{i-2}$ implies
that 
\[
u_i(\alpha,\beta)= \alpha u_{i-1}(\beta,\alpha)-u_{i-2}(\alpha,\beta), 
\]
therefore
\begin{eqnarray*} 
    \varphi(u_i(\alpha,\beta)) &=& 
    \varphi(\alpha)
    \varphi(u_{i-1}(\beta,\alpha))-\varphi(u_{i-2}(\alpha,\beta))\\
    &=&t_{ab}t_{b,i}-t_{a,i-1}\\
    &=& 
    \begin{cases}
        t_{a,i+1} & \text{if $i<m$};\\
        0 & \text{if $i=m<\infty$},
    \end{cases}
\end{eqnarray*}
where the second equality holds by induction and the third equality follows from Proposition
\ref{prop:DihedralJcProducts}. This proves Equation \eqref{eq:u1}; the proof of Equation \eqref{eq:u2} is similar. 

(b) By construction, the ideal $\cali_u^\Z$ is generated by elements of the form
$r_u(\alpha)=u_{m_\alpha-1}(\alpha,\bar\alpha)$ where $\alpha$ is an arrow in $Q$ with finite weight. Such elements
vanish via $\varphi$ by Equation \eqref{eq:u1}, therefore $\varphi$ factors through $\cali_u^\Z$ and descends to
the map $\bar\varphi$ as claimed.

(c) The map $\bar\varphi$ is unital since $\varphi$ is unital. To show that $\bar\varphi$ is surjective, we prove
that $t_w\in \image\bar\varphi$  for all $w\in C$ by induction on $l(w)$. In the base case where $l(w)=1$, we must
have $w=a$ for some $a\in S$ and hence $t_w= \bar\varphi(e_a)\in \image\bar\varphi$.  When $l(w)>1$,  we have 
\[
    \bar\varphi(p_w)=\varphi(p_w)\in t_w + J_C^{(l(w)-2)}
\]
by \eqref{eq:top} and $J_C^{(l(w)-2)}\se \image\bar\varphi$ by induction, therefore $t_w\in \image\bar\varphi$.

It remains to prove that $\bar\varphi$ is injective. Let $ x=\sum_{p\in \unbr(Q)}c_pp $ be a nonzero element in $\Z
Q/\cali_u^\Z$. We need to show that $\bar\varphi(x)\neq 0$. To do so, let $k$ be the maximal number such that
$c_p\neq 0$ for some unbraided path $p$ of length $k$, and let $\{p_1,\dots,p_n\}$ be the set of paths of length
$k$ appearing with nonzero coefficients in $x$. Let $w_i=\iota^{-1}(p_i)$ and write $c_i:=c_{p_i}$ for all $1\le
i\le n$. Then $l(w_{i})= k+1$ for all $i$ and we have \[ \bar \varphi(x)= \sum_{i=1}^{n} c_i \varphi(p_{i}) \in
\sum_{i=1}^n c_it_{w_i}+J_C^{(k-1)} \] by \eqref{eq:top}.  It follows that $\bar\varphi(x)\neq 0$, and the proof is
complete.  
\end{proof}

\subsection{Proof of Theorem \ref{thm:IsoQuotients}: Dihedral
Case}
\label{sec:IsoQuotients}
Let $\{f_n\}$ be a uniform  family of polynomials over $K$.  As the generators of the ideal $\cali_f$ correspond to
individual pairs of dual arrows in $Q$, we first prove Theorem \ref{thm:IsoQuotients} in the \emph{dihedral} case,
the case where $\abs{S}=2$. Let $S=\{a,b\}$, let $m=m(a,b)\ge 3$, and denote the arrows $a\ra b$ and $b\ra a$ in
$Q$ by $\alpha$ and $\beta$, respectively. If $m=\infty$, then $\cali_f=0$ and the theorem clearly holds, so until
Corollary \ref{coro:DihedralMatch} we assume that $m$ is finite. Under this assumption, we show that $KQ/\cali_f$
is semisimple and find its Artin--Wedderburn decomposition. We start with the category $\repqrf$ in light of the
equivalence between mod-$KQ/\cali_f$ and $\repqrf$ (see \autoref{sec:QuiverReps}). The simple modules of $\repqrf$
turn out to have the following forms:

\begin{lemma} \label{lemm:DihedralSimples}
\begin{enumerate}
\item For each root $\lambda$ of the polynomial $\tilde f_{m-1}$, the assignment
\[
    M(\lambda):=(M_a, M_b,\pa, \pb)=(K,K,\id,\lambda\cdot\id)
\] 
defines a simple representation in $\mathrm{rep}(Q,\cali_f)$. Moreover, if $\lambda$ and $\lambda'$ are distinct
roots of $\tilde f_{m-1}$, then $M(\lambda)\not\cong M({\lambda'})$.

\item If $m$ is even, then the assignments 
\[
S(a):=(K,0,0,0), \quad S(b):=(0,K,0,0)
\]
define two non-isomorphic simple representations in $\repqrf$.
\end{enumerate} 
\end{lemma}

\begin{proof}
    (a) Recall that $\cali_f$ is generated by the relations
    \begin{equation}\label{eq:rf_alpha}
        r_f(\alpha)=f_{m-1}(\alpha,\beta)=
    \begin{cases}
   \tilde f_{m-1}(\alpha\beta)  & \text{if $m$ is odd};\\
       \tilde f_{m-1}(\alpha\beta)\alpha& \text{if $m$ is  even},
    \end{cases} 
\end{equation}
and
    \begin{equation}\label{eq:rf_beta}
    r_f(\beta)=f_{m-1}(\beta,\alpha)=
    \begin{cases}
   \tilde f_{m-1}(\beta\alpha)  & \text{if $m$ is odd};\\
   \tilde f_{m-1}(\beta\alpha)\beta& \text{if $m$ is  even}. 
    \end{cases} 
\end{equation}
The maps $\papb$ and $\pbpa$ both equal $\lambda\cdot \id$ as maps from $K$ to $K$.  Since $\lambda$ is a root of
$\tilde f_{m-1}$, it follows that $\tilde f_{m-1}(\papb)=\tilde f_{m-1}(\pbpa)=0$, hence $M(\lambda)$ satisfies the
relations $r_f(\alpha)$ and $r_f(\beta)$ and forms a representation in $\mathrm{rep}(Q,\cali_f)$. Note that
$M(\lambda)$ is simple by basic linear algebra. 

To check that $M(\lambda)\not\cong M({\lambda'})$ for distinct roots $\lambda, \lambda'$ of $\tilde f_{m-1}$, let
$M({\lambda'})=(M'_a,M'_b,M'_\alpha,M'_\beta)$. Then an isomorphism $\phi: M_\lambda\ra M_\mu$ must consist of two
linear isomorphisms $\phi_a: M_a\ra M_a, \phi_b:M_b\ra M_b$ such that \[ \phi_b M_\alpha=M'_\alpha\phi_a,\quad
\phi_a M_\beta=M'_\beta \phi_b.  \] The isomorphisms $\phi_a, \phi_b$ must be multiplication by nonzero scalars
$x,y$, respectively, whence the above equations become $y=x$ and $\lambda y=\lambda' x$. This cannot happen,
therefore $M(\lambda) \not\cong M(\lambda')$.

(b) When $m$ is even the assignments $M_\alpha=M_\beta=0$ clearly satisfy the relations $r_f(\alpha)$ and
$r_f(\beta)$, so $S(a)$ and $S(b)$ define representations in $\repqrf$. Moreover, the representations are simple
and non-isomorphic by dimension considerations.  
\end{proof} 

To prove $\repqrf$ is semisimple, we will use the following linear algebra facts to decompose every representation
in $\repqrf$ into a direct sum of simple modules.
\begin{lemma} \label{lemm:decompose} 
    Let $h\in K[X]$ be a polynomial with degree $k\ge 1$ and with 
 $k$ distinct nonzero roots $z_1, z_2, \ldots, z_k$ in $K$. 
 Let $U$ and $V$ be finite dimensional vector spaces, and let $A: U \to V$ and
 $B: V \to U$ be linear maps  such that 
\begin{equation} \label{eq:odd}
h(BA) = 0_U \quad \quad \quad {\text{and}} \quad \quad \quad h(AB) = 0_V
\end{equation}
or
\begin{equation} \label{eq:even}
h(AB)A = 0_U \quad \quad \quad {\text{and}} \quad \quad \quad h(BA)B = 0_V.
\end{equation}
Then the following results hold.
\begin{enumerate}
\item Both $AB$ and $BA$ are diagonalizable; their eigenvalues lie in the set $\{z_1, z_2, \ldots, z_k\}$ if
    \eqref{eq:odd} holds and in the set  $\{0,z_1, z_2, \ldots, z_k\}$ if  \eqref{eq:even} holds. In particular, we
    have eigenspace decompositions 
    \[ U = U_{z_1} \oplus U_{z_2} \oplus \ldots \oplus U_{z_k}, \quad  V = V_{z_1}
    \oplus V_{z_2} \oplus \ldots \oplus V_{z_k}\]
    if \eqref{eq:odd} holds and 
    \[U = U_0\oplus U_{z_1} \oplus
    U_{z_2} \oplus \ldots \oplus U_{z_k}, \quad V = V_0\oplus V_{z_1} \oplus V_{z_2} \oplus \ldots \oplus V_{z_k}\]
    if \eqref{eq:even} holds, where $U_\lambda$ and $V_\lambda$ denotes the $\lambda$-eigenspace of $BA$ and $AB$
    for each scalar $\lambda$, respectively.

\item For all $1\le i\le k$, the restrictions of $A$ to $U_{z_i}$ and of ${z_i}\inverse\cdot B$ to $V_{z_i}$ form
    mutually inverse isomorphisms. When \eqref{eq:even} holds, the restrictions of $A$ to $U_0$ and of $B$ to $V_0$
    are both zero maps.
\end{enumerate}
\end{lemma}

\begin{proof} 
    (a)  The equations in \eqref{eq:odd} and in \eqref{eq:even} imply that the minimal polynomials of both $AB$ and
    $BA$ divide  $h$ and the polynomial $g:=x\cdot h\in K[x]$, respectively. The result follows since the
    polynomials $h$ and $g$ have distinct roots in the sets $\{z_1,\dots, z_k\}$ and $\{0,z_1,\dots,z_k\}$,
    respectively.

    (b) Let $1\le i\le k$. Set $U_i=U_{z_i}, V_i=V_{z_i}$ and $B'=z_i\inverse\cdot B$. Use $\vert$ to denote
    restriction of maps so that, for example, $A\vert_{U_i}$ stands for the restriction of $A$ to $U_i$. By direct
    computation, we have $Au\in V_{i}$ for all $u \in U_{i}$, $B'v\in U_{i}$ for all $v\in V_{i}$, and
    $B'A\vert_{U_i}=\id_{U_{i}}, AB'\vert_{V_i}=\id_{V_{i}}$. This proves the first claim. To prove the second
    claim, assume the equations in \eqref{eq:even} hold and write $h= x \cdot \bar{h} + c$ where $c$ is the
    constant term of $h$. Then $c \neq 0$ since $0$ is not a root of $h$, and we have
    \[
        h(AB)A=Ah(BA)=A(\bar h(BA)\cdot BA+c)=A\bar h(BA)\circ BA+cA.
    \]
    Let $u\in U_0$. Then $BA(u)=0$, therefore 
    \[
        0 =[h(AB)A] (u) =[A\bar(BA)](BA(u))+cA(u)=cA(u).
\]
where the first equality holds since $h(AB)A=0_U$. It follows that $A(u)=0$, so $A\vert_{U_0}$ is the
zero map.  The proof that $B\vert_{V_0}$ is the zero map is similar.
 \end{proof}

\begin{thm} \label{thm:ArtinWedderburn} 
    Let $(W,S)$ be an irreducible Coxeter system where $S=\{a,b\}$ and
    $3\le m:=m(a,b)<\infty$. 
\begin{enumerate} 
    \item Suppose $m$ is odd. Then the category $\mathrm{rep}_K(Q,\cali_f)$ is semisimple and has exactly $(m-1)/2$
        non-isomorphic simple representations, all of dimension 2. The algebra $KQ/\cali_f$ is semisimple, and is
        isomorphic to the direct product of $(m-1)/2$ copies of the matrix algebra $M_{2 \times 2}(K)$.
    \item Suppose $m$ is even. Then the category $\mathrm{rep}(Q,\cali_f)$  is semisimple and has exactly $(m-2)/2
        + 2$ non-isomorphic simple representations; two of these representations have dimension 1 and the other
        representations have dimension 2. The algebra $KQ/\cali_f$ is semisimple, and is isomorphic to the direct
        product of two copies of $K$ and $(m-2)/2$ copies of $M_{2 \times 2}(K)$.   
\end{enumerate} 
\end{thm} 

\begin{proof}
Let $M=(M_a,M_b,M_\alpha, M_\beta)$ be a representation in $\repqrf$ where $\alpha$ and $\beta$ are the arrows
$a\ra b$ and $b\ra a$ in $Q$, respectively.  Set $h=\tilde f_{m-1}, U=M_a, V=M_b, A=M_\alpha$ and $B=M_\beta$. If
$m$ is odd, then the equations in \eqref{eq:odd} hold by Equations \eqref{eq:rf_alpha} and \eqref{eq:rf_beta}.
Using Lemma \ref{lemm:decompose}, we may then decompose $M$ into a direct sum where each summand is of the form
$N(\lambda):=(U_\lambda, V_\lambda, A\vert_{U_\lambda}, B\vert_{V_\lambda})$ where $\lambda$ is one of the
$(m-1)/2$ roots of $h=\tilde f_{m-1}$ and $B\vert_{V_\lambda}A\vert_\lambda=\lambda\cdot \id_{U_\lambda},
A\vert_{U_\lambda}B\vert_{V_\lambda}=\lambda\cdot \id_{V_\lambda}$. It is easy to verify that $N(\lambda)$ is
isomorphic to the representation $M(\lambda)$ from Lemma \ref{lemm:DihedralSimples}. The claims in Part (a) now
follow from the Artin--Wedderburn theorem and the equivalence between $\repqrf$ and mod-$KQ/\cali_f$. Similarly, if
$m$ is even, then the equations in \eqref{eq:even} hold and we may use Lemma \ref{lemm:decompose} to decompose $M$
into a direct sum of  the representations 
\[
    \ker\alpha:=(U_0,0,0,0),\quad \ker\beta:=(0,V_0,0,0)
\]
and simple representations isomorphic to $M(\lambda)$ where $\lambda$ is one of the $(m-2)/2$ roots of $\tilde
f_{m-1}$. The represenations $\ker\alpha$ and $\ker\beta$ further decompose into $\dim(U_0)$ and $\dim(V_0)$ copies
of the modules $S(a)$ and $S(b)$ from Lemma \ref{lemm:DihedralSimples}, respectively. Part (b) follows.
\end{proof}

We are ready to prove that in the case $\abs{S}=2$, the isomorphism type of the algebra $KQ/\cali_f$ does not
depend on the choice of the uniform  family of polynomials $\{f_n\}$. Note that we no longer assume that $m$ is
finite in the result below.

\begin{corollary} 
    \label{coro:DihedralMatch}
    Let $(W,S)$ be an irreducible Coxeter system where $S=\{a,b\}$ and $3\le m:=m(a,b)\le \infty$.  Let
    $\{f_n\},\{g_n\}$ be two uniform  families of polynomials over $K$. Then there is an algebra isomorphism $\Phi:
    KQ/\cali_f \ra KQ/\cali_g$ such that $\Phi(e_i)=e_{i}$ for $i\in \{a,b\}$.
\end{corollary}

\begin{proof} 
When $m=m(a,b)=\infty$, we have $\cali_f=\cali_g=0$, so we may take $\phi$ to be the identity map. Now assume $m$
is finite. We first treat the case where $m$ is even. Consider the direct product $ B=B_1\times B_2\times B_3\times
\dots\times B_r$ where $r=(m-2)/2+2$, $B_1=B_2=K$, and $B_i$ equals the matrix algebra $M_{2\times 2}(K)$ for all
$3\le i\le r$.  By Theorem \ref{thm:ArtinWedderburn}, there exist algebra isomorphisms $\phi: KQ/\cali_f\ra B$ and
$\psi:KQ/\cali_g\ra B$. 

Let $x=\phi(e_1), y=\phi(e_2)$ and write $x=(x_1,\dots,x_r), y=(y_1,\dots,y_r)$. Then we have:

\begin{enumerate}
    \item Since $e_1, e_2$ are idempotents, $x,y$ must be idempotents, therefore $x_i, y_i$ are idempotents in
        $B_i$ for all $1\le i\le r$.  This implies that $x_1,x_2,y_1,y_2\in\{0,1\}$ and that for all $3\le i\le r$,
        $x_i, y_i$ must be each conjugate to the zero matrix, the identity matrix, or the matrix $E_{11}:=\bm{1
        &0\\0&0}$.

    \item Since $1=e_1+e_2$ in $KQ/\cali_f$, we must have $x+y=1$ and hence $x_i+ y_i=1$ for all $1\le i\le r$.

    \item Since $\phi$ is an isomorphism, we have
            \[
                \dim (xBx)=\dim (\phi(e_1)B \phi(e_1))=\dim(e_1(KQ/\cali_f)e_1).
\]
    Here, we have $\dim (xBx)=\sum_{i=1}^r \dim (x_iB_ix_i)$ in the direct product $B$. We also have $\dim
    (e_1(KQ/\cali_f) e_1)=r-1$ because it is easy to see that the classes of the elements $e_1, \alpha\beta,
    (\alpha\beta)^2,\dots,(\alpha\beta)^{r-2}$ form a basis of $e_1(KQ/\cali_f)e_1$. It follows that $\sum_{i=1}^r
    \dim(x_iB_ix_i)=r-1$.  Similarly, we must have $\sum_{i=1}^r \dim (y_iB_iy_i)=r-1$. Notice that for all $1\le
    i\le r$, the dimensions of $x_iB_ix_i$ and $y_iB_iy_i$ depend only on the conjugacy classes of the idempotents
    $x_i, y_i$, respectively.
\end{enumerate}
    By straightforward dimension considerations, the above three facts force that $x_i, y_i$ are conjugate to
    $E_{11}$ for all $3\le i\le r$ and that we either have $x_1=0,x_2=1, y_1=1, y_2=0$ or have $x_1=1, x_2=0,
    y_1=0, y_2=1$. Similarly, the same conclusions apply to the coordinates $x'_i, y'_i$ of the elements
    $x'=(x'_1,x'_2,x'_3,\dots, x'_r)=\psi(e_1)$ and $y'=(y'_1,y'_2,y'_3,\dots, y'_r)=\psi(e_2)$. Thus, we either
    have $x\sim x', y\sim y'$ or have $x\sim y', y\sim x'$ where $\sim$ means two elements are conjugate. In both
    cases, it is easy to find an automorphism $\eta$ of $B$ such that the map $\Phi:=\psi\inverse \eta\phi :
    KQ/\cali_f\ra KQ/\cali_g$ is an isomorphism sending $e_i$ to $e_i$ for $i\in \{1,2\}$.  This proves the
    corollary in the case where $m$ is even.

    The proof for the case where $m$ is odd is similar but simpler: let $B=B_1\times \dots \times B_r$ where
    $r=(m-1)/2$ and $B_i=M_{2\times 2}(K)$ for all $1\le i\le r$, then consider the isomorphisms $\phi, \psi$
    guaranteed by Theorem \ref{thm:ArtinWedderburn} as before. This time, facts similar to (a), (b), (c) will force
    all coordinates of $\phi(e_1),\phi(e_2),\psi(e_1),\psi(e_2)$ to be conjugate to $E_{11}$, allowing us to form
    an isomorphism $\Phi$ with the desired properties as before.  
\end{proof}

\subsection{Proof of Theorem \ref{thm:IsoQuotients}: General Case}
\label{sec:isoquotients_general}
We now prove Theorem \ref{thm:IsoQuotients} for a general Coxeter system $(W,S)$. The rough idea is to notice that
each edge in the Coxeter diagram corresponds to a dihedral system, so we can take the ``local'' isomorphisms
provided by Corollary \ref{coro:DihedralMatch} and then assemble them to a ``global'' isomorphism between the
quotients of $KQ$.

\begin{proof}[Proof of Theorem \ref{thm:IsoQuotients}]
Let $E$ be the set of edges of the Coxeter diagram of $(W,S)$.  For each $e\in E$ of the form $a-b$, let $\alpha:
a\ra b$ and $\beta:b\ra a$ be the dual arrows arising from $e$ in $Q$ and consider the subquiver
$Q_e=(\{a,b\},\{\alpha,\beta\})$ of $Q$. Let $\cali_f(e),\cali_g(e)$ be the evaluation ideals of $\{f_n\}$ and
$\{g_n\}$ in $Q_e$, respectively. Fix an isomorphism $\Phi_e: KQ_e/\cali_f(e) \ra KQ_e/\cali_g(e)$ such that
$\Phi_e(e_a)=e_a, \Phi_e(e_b)=e_b$ for $e$; such an isomorphism exists by Corollary \ref{coro:DihedralMatch}. Note
that $KQ_e/\cali_g(e)$ naturally embeds into $KQ/\cali_g$, so we can naturally view an element of $KQ_e/\cali_g(e)$
as an element in $KQ/\cali_g$. We will do so without further comment.

Let  $Q^{\le 1}=\{e_a:a\in Q_0\}\cup Q_1$ be the set of stationary paths and arrows of $Q$.  Consider the function
$\phi: Q^{\le 1} \ra KQ/\cali_{g}$ such that for every edge $e=\{a,b\}$ in $G$ and the arrows $\alpha:a \ra  b,
\beta: b\ra a$ in $Q$, we have 
\begin{equation*}
    \label{eq:phi}
    \phi(e_a)=\Phi_e(e_a), \; \phi(e_b)=\Phi_e(e_b), \;
    \phi(\alpha)=\Phi_e(\alpha),\; \phi(\beta)=\Phi_e(\beta).
\end{equation*}
This function is well-defined because even if a vertex $a$ in $G$ is incident to two distinct edges $e, e'$ in $G$,
the maps $\Phi_e$ and $\Phi_{e'}$ both send $e_a$ to $e_a$, causing no ambiguity for the value of $\phi(e_a)$.
Next, recall again that the path algebra $KQ$ is generated by $Q^{\le 1}$ subject only to the relations that
$e_ue_v=\delta_{u,v}e_u$ for $u,v\in Q_0$ and the  relations $e_a\alpha =\alpha=\alpha e_b$ for each arrow $\alpha:
a\ra b$ in $Q_1$, and note that the map $\phi$ respects these relations: we have $
\phi(e_u)\phi(e_v)=e_ue_v=e_ue_v=\delta_{u,v}e_u=\delta_{u,v}\phi(e_u), $ $
\phi(e_a)\phi(\alpha)=\Phi_e(e_a)\Phi_e(\alpha)=\Phi_e(e_a\alpha)=\Phi(\alpha)=\phi(\alpha)$, and similarly
$\phi(\alpha)\phi(e_b)=\phi(\alpha)$. It follows that $\phi$ extends to a unique homomorphism $\Phi:K Q\ra
KQ/\cali_g$ with $\Phi(x)=\phi(x)$ for all $Q^{\le 1}$. Finally, for each edge $e:a-b$ in $Q$ and the corresponding
arrows $\alpha:a\ra b,\beta:b\ra a$, the restriction of $\Phi$ to $KQ_e$ agrees with $\Phi_e$, therefore $\Phi$
sends both $r_f(\alpha)$ and $r_f(\beta)$ to zero because $\Phi_e$ does. It follows that $\Phi$ factors through
$\cali_f$ to induce a homomorphism $\bar\Phi: KQ/\cali_f\ra KQ/\cali_g$.
Starting from the collection $\{\Psi_e:e\in E\}$ where $\Psi_e=\Phi_e\inverse$ for all $e\in E$, we may obtain in
the same way a homomorphism $\bar\Psi: KQ/\cali_g\ra KQ/\cali_f$, and it is clear that $\bar\Psi$ and $\bar\Phi$
are mutual inverses, therefore $KQ/\cali_f\cong KQ/\cali_g$.  
\end{proof}

\section{Quiver Contractions}
\label{sec:A-modTools}

Let $J_C$ be the subregular $J$-ring of an arbitrary Coxeter system $(W,S)$, let $K$ be an algebraically closed
field of characteristic zero, and let $A=A_K=K\otimes_{\Z} J_C$. Let mod-$A$ be the category of finite dimensional
right $A$-modules. The rest of the paper is dedicated to the study of mod-$A$.

In this section we introduce a procedure to modify a quiver $Q$ to a new quiver $\bar Q$ such that the algebra $K
\bar Q/\bar\cali_f$ is Morita equivalent to $K Q/\cali_f$ for any uniform  family of polynomials $\{f_n\}$ over
$K$, where $\bar\cali_f$ is the evaluation ideal of $\{f_n\}$ in $K\bar Q$. We call the procedure a \emph{quiver
contraction}.  In applications, we will often iterate contractions to obtain sequences of the form 
\begin{equation} \label{eq:repeatedContractions} 
    Q^{(0)}:=Q\ra Q^{(1)} \ra \dots\ra
Q^{(n)} 
\end{equation} 
where $Q$ is the double quiver of $(W,S)$. Denote the evaluation ideal of $\{f_n\}$ in $K Q^{(i)}$ by
$\cali_f^{(i)}$ for each $i$. Then  $A\cong K Q/\cali_f$ by Theorems \ref{thm:JcIso} and \ref{thm:IsoQuotients},
therefore mod-$A$ is equivalent to $\rep_K (Q^{(i)},\cali^{(i)}_f)$ for all $0\le i\le n$. For this reason,  we
shall develop tools for studying the last type of category in this section to prepare for the study of mod-$A$ in
Section \autoref{sec:mod-AResults}. 

\subsection{Definition of Quiver Contractions} 
\label{sec:contractionDefinitions}
Consider the following generalization of double quivers of Coxeter systems:
\begin{definition}
\label{def:generalizedDoubleQuivers}
A \emph{generalized double quiver} is a triple $(Q,d,m)$ consisting of a quiver
$Q=(Q_0,Q_1)$, a map
$d: Q_1\ra Q_1$, and a
map $m: Q_1\ra \Z_{\ge 1}\cup \{\infty\}$ such that  
\begin{enumerate}
    \item $d(Q_{ab})=Q_{ba}$ for all $a,b\in Q_0$, where
        $Q_{c,d}$ denotes the set of all arrows in $Q$ from $c$ to
        $d$ for all
        $c,d\in Q_0$.
    \item $d^2(\alpha)=\alpha$ for all $\alpha\in Q_1$;
    \item $m(\alpha)=m({d(\alpha)})$ for all $\alpha\in Q_1$.
\end{enumerate}
Given such a triple, we also call $Q$ a \emph{generalized double quiver}. We say that two arrows $\alpha,\beta\in
Q_1$ are \emph{dual} to each other if $\beta=d(\alpha)$, and call $m(\alpha)$ the \emph{weight} of $\alpha$ for all
$\alpha\in Q_1$.  
\end{definition}

\noindent Note that we may (and will) naturally view the double quiver $Q$ of a Coxeter system as a generalized
double quiver by setting $m(\alpha)=m_\alpha$ and $d(\alpha)=\bar\alpha$ for all $\alpha\in Q_1$.

We now define quiver contractions as operations on generalized double quivers $(Q,d,m)$. Roughly speaking, we will
define a contraction along a suitable pair of arrows $\alpha:a\ra b$ and $\beta:b\ra a$ where $a,b$ are distinct
vertices in $Q$. The contraction will identify $a,b$ by collapsing them into a new vertex $v_{ab}$, replace
$\alpha,\beta$ by a loop at $v_{ab}$, and reroute all other arrows incident to $a$ or $b$ in $Q$ to new arrows
incident to $v_{ab}$.  The assignments of duals and weights of arrows in the new quiver will be naturally inherited
from $d$ and $m$. 

\begin{definition}
    \label{def:quiverContraction}
    Let $(Q,d,m)$ be a generalized double quiver.
 A pair of arrows $\{\alpha,\beta\}$ is called \emph{contractible}
 if they are of the form $\alpha:a\ra b,\beta:b\ra a$ where $a,b$ are distinct
            vertices in $Q$, $\beta=d(\alpha)$, and
            $m(\alpha)=m(\beta)$ is an odd integer that is at least 3. 

        For a contractible pair of arrows $\{\alpha:a\ra b,\beta:b\ra
            a\}$, the \emph{contraction of
            $(Q,d,m)$ along $\{\alpha,\beta\}$} is the generalized
            double quiver $(\bar Q,\bar d, \bar m)$ where 
            \begin{enumerate}
                \item The vertex set of the quiver $\bar Q$ is
                    \[
                        \bar Q_0:=Q_0\setminus\{a,b\}\sqcup\{v_{ab}\},
                    \]
                    where $v_{ab}\notin Q_0$  is a newly introduced
                    vertex. The arrow set $\bar Q_1$ of $\bar Q$ is defined as
                    follows: write
                    \[
                        c':=
                        \begin{cases}
                            v_{ab}&\text{ if $c\in \{a,b\}$};\\
                            c &\text{ otherwise}
                        \end{cases}
                    \]
                    for all $c\in Q_0$ and define $\gamma'$ to be the arrow $u'\ra v'$ for each arrow $\gamma: u\ra
                    v$ in $Q_1$, then let
\[
    \bar Q_1:=\{\gamma': \gamma\in
    Q_1\setminus\{\alpha,\beta\}\}\sqcup\{\ve{ab}\},
\]
where $\ve{ab}\notin Q_1$ is a newly introduced loop at $v_{ab}$.  
\item $\bar d$ is defined by $\bar d(\ve{ab})=\ve{ab}$ and
    $\bar d(\gamma')=d(\gamma)'$ for all $\gamma\in
    Q_1\setminus\{\alpha,\beta\}$.
\item $\bar m$ is defined by $\bar m(\ve{ab})=m(\alpha)$ and
    $\bar m(\gamma')=m(\gamma)$ for all $\gamma\in
    Q_1\setminus\{\alpha,\beta\}$. 
    \end{enumerate}
\end{definition}

\noindent Note that quiver contractions introduce loops and may lead to multiple pairs of arrows between two
distinct vertices in the resulting quiver (see the quiver $Q^{(2)}$ in Example \ref{eg:LaurentPolynomials}),
features that cannot be present in double quivers of Coxeter diagrams. This is the reason why we do not forbid
these features in Definition \ref{def:generalizedEvaluationIdeals}. On the other hand, the generalization from
double quivers to generalized ones is mild enough that we can extend the definition of evaluation ideals easily, at
least for the cases we are interested in:

\begin{definition}
    \label{def:generalizedEvaluationIdeals}
    Let $\{f_n:n\ge 2\}$ be a uniform  family of
    polynomials over $K$, and let $(Q,d,m)$ be a generalized double quiver. We
    define the \emph{evaluation ideal} of $\{f_n\}$ in $KQ$ to be the two-sided
    ideal $\cali_f\se KQ$ given by 
    \[
        \cali_f:=\ip{r_f(\alpha):\alpha\in Q_1}
    \]
    where 
\[
    r_f(\alpha)=
    \begin{cases}
        0 & \text{if $m=\infty$};\\
        f_{m-1}(\alpha, d(\alpha)) & \text{if $m<\infty$ and $d(\alpha)\neq \alpha$};\\
        \tilde f_{m-1}(\alpha) & \text{if $m<\infty$ and $d(\alpha)=\alpha$},
    \end{cases}
\]
where $m=m(\alpha)$. Here, as in Equation \eqref{eq:TildeEval}, the evaluation of $\alpha$ through a constant term
$c$ in $\tilde f_{m-1}(\ve{})$ returns $ce_{a}$ for $a=\source(\alpha)$.  For example, if $\alpha$ is a self-dual
loop $\ve{}:a\ra a$ with $m=5$, then $r_f(\alpha)=\ve{}^2-e_a$ if $f_4=x^4-1$ and $r_f(\alpha)=2\ve{}^2-\ve{}-3e_a$
if $f_4=2x^4-x^2-3$.  
\end{definition}

\begin{remark}
    In this paper we are only interested in generalized double quivers $\bar Q$ obtained from the double quiver of
    a Coxeter diagram via iterated contractions. In this case, every self-dual arrow $\alpha$ in $\bar Q$ must be
    either a loop of the form $\varepsilon=\ve{ab}$ at a vertex $v=v_{ab}$ introduced during a contraction of a
    quiver $Q$ along a dual pair of arrows $\gamma:a\ra b,\delta:b\ra a$ or a reroute of such a loop. In
    particular, $m$ must be a finite, odd integer, so the third case in the definition of $r_f(\alpha)$ applies and
    gives $r_f(\alpha)=\tilde f_{m-1}(\ve{})$.  The relation $r_f(\alpha)=\tilde f_{m-1}(\ve{})$ in $K\bar Q$
    mirrors the relation $r_f(\gamma)=\tilde f_{m-1}(\gamma\delta)$ in the evaluation ideal $\cali_f$ of $KQ$ via
    the replacements $\gamma\delta\mapsto \ve{}$ and $a\mapsto v$.  For example, if $m=5$ and $f_4=x^4-1$, then the
    relation $r_f(\gamma)=\tilde f_{4}(\gamma\delta)=(\gamma\delta)^2-e_a\in\cali_f$ is mirrored by the relation
    $r_f(\ve{})=\tilde f_4(\ve{})=\ve{}^2-e_v$.  \label{rmk:sqrt} 
\end{remark}

Our main result on contractions is the following theorem.
\begin{thm} 
    \label{thm:MoritaEquivalence} 
    Let $(Q,d,m)$ be a generalized double quiver and let $(\bar Q, \bar d, \bar m)$ be a contraction of $(Q,d,m)$
    along a contractible pair of arrows $\{\alpha,\beta\}$. Let $\{f_n:n\ge 2\}$ be a uniform family of polynomials
    over $K$, and let $\cali_f$ and $\bar\cali_f$ be the evaluation ideal of $\{f_n\}$ in $K Q$ and $K \bar Q$,
    respectively. Then the algebras  $K Q/\cali_f$ and $K \bar Q/\bar\cali_f$ are Morita equivalent.
\end{thm}

We postpone the proof of the theorem to \autoref{sec:MoritaEquivalence}. Before the proof, we discuss several
detailed examples of quiver contractions and some consequences of the theorem in the next subsection.

\subsection{Examples of Quiver Contractions}
\label{sec:contractionExamples}
Throughout this subsection, $G$ denotes the Coxeter diagram of a Coxeter system $(W,S)$ and $Q$ stands for the
double quiver of $G$. When drawing generalized double quivers, we label each pair of dual arrows
$\{\alpha,d(\alpha)\}$ with their common weight $m(\alpha)$ except when $m(\alpha)=3$, including for the case where
$\alpha$ is a self-dual loop of the form $\ve{ab}$ introduced by a contraction.  For convenience, we consider only
the polynomials $\{f_n:n\ge 2\}$ where
\begin{equation}
    f_{n}=
    \begin{cases}
        x^n-1& \text{if $n$ is even};\\
        x^n-x& \text{if $n$ is odd}. 
    \end{cases}
    \label{eq:PowerPoly}
\end{equation}
for all $n\ge 2$. Note that $\{f_n\}$ is a uniform family over $K$ since $K$ is algebraically closed (see
Remark \ref{rmk:fieldAssumptions}).

\begin{example}
\label{eg:simpleContractions}
Suppose that $G$ and $Q$ are as shown at the top of \autoref{fig:simpleContraction3}. The arrows $\alpha$ and
$\beta$ are dual to each other in $Q$ and have weight $3$, therefore they form a contractible pair.  The
contraction along $\{\alpha,\beta\}$ results in the quiver $\bar Q$ shown in the bottom right corner of the figure,
where $v=v_{ab}$ and $\ve{}=\ve{ab}$. 
\begin{figure}[h!]
    \centering
    \begin{tikzpicture}        
        \node (a) {$G:$};
        \node[main node] (1) [above right = 0.8cm and 0.3cm of a] {$a$};
\node[main node] (2) [right = 2cm of 1] {$b$}; 
\node[main node] (3) [below = 2cm of 2] {$c$};
\node[main node] (4) [left = 2cm of 3] {$d$};

\path[draw] 
(1) edge node[above] {} (2) 
(2) edge node[right] {} (3)
(3) edge node[below] {\tiny{5}} (4)
(4) edge node[left] {\tiny{4}} (1);

\node (l) [below right = 0.8cm and 1cm of 2] {$\longrightarrow$}; 
\node (00) [right = 0.6cm of l] {$Q:$};
\node[main node] (5) [above right = 0.75cm and 0.6 cm of 00] {$a$};
\node[main node] (6) [right = 2cm of 5] {$b$};
\node[main node] (7) [below = 2cm of 6] {$c$};
\node[main node] (8) [left = 2cm of 7] {$d$}; 
\node (x) [right = 0.1cm of 00] {\tiny{4}};
\node (y) [below right = 0.1cm and 0.85cm of 8] {\tiny{5}};

\path[draw] 
(5) edge [bend left = 10,-stealth] node [fill=white, anchor=center, pos=0.5,
inner sep=0.5pt] {\tiny{$\alpha$}} (6) 
(6) edge [bend left = 10,-stealth] node [fill=white, anchor=center, pos=0.5,
inner sep=0.5pt] {\tiny{$\beta$}} (5)
(6) edge [bend left=10,-stealth] node [fill=white, anchor=center, pos=0.5,
inner sep=0.5pt] {\tiny{$\gamma$}} (7) 
(7) edge [bend left= 10,-stealth] node [fill=white, anchor=center, pos=0.5,
inner sep=0.5pt] {\tiny{$\delta$}} (6) 
(7) edge [bend left = 10,-stealth] node [fill=white, anchor=center, pos=0.5,
inner sep=0.5pt] {\tiny{$\zeta$}} (8) 
(8) edge [bend left = 10,-stealth] node [fill=white, anchor=center, pos=0.5,
inner sep=0.5pt] {\tiny{$\eta$}} (7) 
(8) edge [bend left = 10,-stealth] node [fill=white, anchor=center, pos=0.5,
inner sep=0.5pt] {\tiny{$\kappa$}} (5) 
(5) edge [bend left = 10,-stealth] node [fill=white, anchor=center, pos=0.5,
inner sep=0.5pt] {\tiny{$\lambda$}} (8);

\node (x) [below right = 0.6cm and 0.8cm of 4] {$\downarrow$};
\node (yy) [below=0.1cm of y] {$\downarrow$};

\node (xa) [below = 4.5cm of a] {$\bar G:$};
\node[main node] (b1) [above right = 0.8cm and 1.5cm of xa] {$v$};
\node[main node] (b3) [below right = 2cm and 1cm of b1] {$c$};
\node[main node] (b4) [below left = 2cm and 1cm of b1] {$d$};

\path[draw] 
(b1) edge node[above] {} (b3) 
(b3) edge node[below] {\tiny{5}} (b4)
(b4) edge node[left] {\tiny{4}} (b1);

\node (y) [below =4.5cm of l] {$\longrightarrow$}; 
\node (000) [right = 0.6cm of y] {$\bar Q:$};
\node[main node] (b5) [above right = 0.75cm and 1.8cm of 000] {$v$};
\node[main node] (b7) [below right = 2cm and 1cm of b5] {$c$};
\node[main node] (b8) [below left = 2cm and 1cm of b5] {$d$}; 

\node (x) [below left = 0.6cm and 0.7cm of b5] {\tiny{4}};
\node (y) [below right = 0.1cm and 0.95cm of b8] {\tiny{5}};

\path[draw] 
(b5) edge [loop above,looseness=15, out=120, in =60, distance=0.8cm, -stealth]
node {\tiny{$\varepsilon$}} (b5)
(b5) edge [bend left=10,-stealth] node [fill=white, anchor=center, pos=0.5,
inner sep=0.5pt] {\tiny{$\gamma'$}} (b7) 
(b7) edge [bend left= 10,-stealth] node [fill=white, anchor=center, pos=0.5,
inner sep=0.5pt] {\tiny{$\delta'$}} (b5) 
(b7) edge [bend left = 10,-stealth] node [fill=white, anchor=center, pos=0.5,
inner sep=0.5pt] {\tiny{$\zeta$}} (b8) 
(b8) edge [bend left = 10,-stealth] node [fill=white, anchor=center, pos=0.5,
inner sep=0.5pt] {\tiny{$\eta$}} (b7) 
(b8) edge [bend left = 10,-stealth] node [fill=white, anchor=center, pos=0.5,
inner sep=0.5pt] {\tiny{$\kappa'$}} (b5) 
(b5) edge [bend left = 10,-stealth] node [fill=white, anchor=center, pos=0.5,
inner sep=0.5pt] {\tiny{$\lambda'$}} (b8); 
    \end{tikzpicture}
    \caption{}
    \label{fig:simpleContraction3}
\end{figure}
\end{example}

Let us examine the effect of the contraction. The three pairs of arrows $\{\gamma,\delta\},
\{\eta,\zeta\}$ and $\{\kappa,\lambda\}$ in $Q$ have weights $3, 5,4$ and
give rise to
the elements
\begin{IEEEeqnarray}{rl}
    r_1=\gamma\delta-e_b, &\quad r_2=\delta\gamma-e_c\\
    r_3=(\zeta\eta)^2-e_c, &\quad r_4=(\eta\zeta)^2-e_d,\\
    r_{5}=\kappa\lambda\kappa-\kappa,&\quad
    r_6=\lambda\kappa\lambda-\lambda
\end{IEEEeqnarray}
of the evaluation ideal $\cali_f$. The contraction reroutes the
arrows $\gamma, \delta, \kappa, \lambda$ since they are incident to $a$
or $b$, but the rerouting preserves weights by definition, so the rerouted arrows
give rise to ``duplicates'' of the relations $r_1,r_2,r_5,r_6$ in the ideal
$\bar \cali_f$, namely,
the relations
\[
    r_1'=\gamma'\delta'-e_v,\quad  r_2'=\delta'\gamma'-e_c,\quad
    r_5'=\kappa'\lambda'\kappa'-\kappa',\quad
    r_6'=\lambda'\kappa'\lambda'-\lambda'. 
\]
The arrows $\zeta$ and $\eta$ and their weights remain unchanged in the contraction since they
are not incident to $a$ or $b$. Consequently, they contribute the same
relations $r_3$ and $r_4$ to $\bar\cali_f$ just as they do to $\cali_f$.
Finally, the arrows $\alpha, \beta$ 
are replaced by a single loop at $v$. Since $m:=m(\alpha)=m(\beta)$ is odd, $\alpha$ and $\beta$ contribute the relations
\[
    r_f(\alpha)=(\alpha\beta)^k-e_a,\quad r_f(\beta)=
(\beta\alpha)^k-e_b 
\]
to $\cali_f$, and their replacement $\ve{}$ contributes a single relation 
\[
    r_f(\ve{})=r_f(d(\ve{}))=\ve{}^k-e_v
\]
to $\bar\cali_f$; here, we have $m=3$ and $k=(m-1)/2=1$.

By Theorem \ref{thm:MoritaEquivalence}, the algebra $A$ is Morita equivalent to the quotient $K \bar Q/\bar\cali_f$
where $ \bar\cali_f=\ip{r_1',r_2',r_3'=r_3, r_4'=r_4, r_5',r_6',r_f(\ve{})} $. The relation  $r_f(\ve{})=\ve{}-e_v$
implies that $\ve{}=e_v$ in the quotient, therefore the quotient is isomorphic to the quotient $K\hat
Q/\hat\cali_f$ where $\hat Q$ is obtained from $\bar Q$ by removing the loop $\ve{}$ and $\hat \cali_f=\ip{r_i:1\le
i\le 6}\se K\hat Q$. More generally, we define a quiver contraction with respect to a pair of arrows
$\{\alpha,\beta\}$ to be \emph{simple} if $m(\alpha)=m(\beta)=3$. By the above discussion, Theorem
\ref{thm:MoritaEquivalence} remains true for a simple contraction if we omit the loop $\ve{}$ in the construction
of $\bar Q$. We shall do so from now on.

For a double quiver $Q$ of a Coxeter diagram $G$, a contraction of $Q$ along a pair of arrows $\{\alpha: a\ra
b,\beta:b\ra a\}$ is simple if and only if the corresponding edge $a-b$ in $G$ is simple.  When this is the case,
we may define a simple contraction of $G$ by ``contracting'' the edge $a-b$ until $a,b$ are identified as a new
vertex $v:=v_{ab}$, thus effectively rerouting all edges incident to $a$ or $b$ to $v$. More precisely, we may
define a weighted graph $\bar G$ whose vertex set is $S\setminus\{a,b\}\sqcup\{v\}$ where $v\notin S$ and whose
edge set is $\{e':e\text{ is an edge in $G$ other than $a-b$}\}$, where $e'=e$ if $e$ is not incident to $a$ or $b$
and $e'$ is the edge $v-c$ whenever $e$ is of the form $a-c$ or $b-c$ for a vertex $c\in S\setminus\{a,b\}$; the
weight $m(e')$ of $e'$ is defined to be the same as that of $e$.  We call $\bar G$ the \emph{simple contraction of
$G$ along $a-b$}. For the above example, the contraction of $G$ along $a-b$ results in the graph $\bar G$ shown in the
lower left corner of \autoref{fig:simpleContraction3}.  Note that if $a,b$ share no neighbor in $G$, which is the
case for our example and must be the case if $G$ has no cycles, then the graph $\bar G$ can again be viewed as the
Coxeter diagram of a Coxeter system $(\bar W, \bar S)$.  In this case, the double quiver of $\bar G$ makes sense,
and it is clear that the double quiver of the simple contraction $\bar G$ of $G$ coincides with the simple
contraction $\bar Q$ of the double quiver $Q$ of $G$ (once we ignore the loop $\ve{}=\ve{ab}$). This phenomenon is
manifest in our example: the diagram in \autoref{fig:simpleContraction3} commutes once we ignore the loop $\ve{}$.

\begin{remark}
Maintain the assumptions and notation of the previous paragraph. In particular, assume that $a-b$ is a simple edge
in $G$ where $a,b$ have no common neighbor. Then Theorem \ref{thm:MoritaEquivalence} implies that the algebra
$A=K\otimes_\Z J_C$ associated to the Coxeter system $(W,S)$ is Morita equivalent to the algebra $\bar A:=K
\otimes_\Z \bar J_C$ associated to the system $(\bar W,\bar S)$. Note that we may state the equivalence purely in
terms of contractions of Coxeter diagrams, with no reference to quivers. Also note that the equivalence implies
that when we study the category mod-$A$, we may assume $G$ has no simple edges whenever $G$ is a tree. The reason
is that, since no two vertices can share a neighbor in a tree, we may repeatedly remove all simple edges in $G$ by
simple contractions. 
    \label{rmk:ignoreSimpleEdges}
\end{remark}

The following example illustrates the reduction allowed by Remark \ref{rmk:ignoreSimpleEdges}. After the example,
we record an application of the reduction for future use.

\begin{example}
\label{eg:treeContraction} 
Suppose that  $G$  is the tree shown on the left in \autoref{fig:treeContraction}. By Remark
\ref{rmk:ignoreSimpleEdges}, the algebra $A$ associated to $(W,S)$ is Morita equivalent to the algebra $A':=K
\otimes J'_C$ where $J'_C$ is the subregular $J$-ring of the Coxeter system whose diagram is the graph $G'$
obtained by contracting all simple edges in $G$; the graph $G'$ is shown on the right on
\autoref{fig:treeContraction}. 
\begin{figure}[h!] 
    \centering 
    \begin{tikzpicture}
\node (0) {$G:$}; 
\node[main node] (1) [below right = 0.8cm and 0.3cm of 0] {$a$};
\node[main node] (2) [right = 1cm of 1] {$b$}; 
\node[main node] (3) [right = 1cm of 2] {$c$};
\node[main node] (4) [above = 0.8cm of 3] {$f$};
\node[main node] (5) [above = 0.8cm of 4] {$g$};
\node[main node] (6) [right = 1cm of 3] {$d$};
\node[main node] (7) [right = 1cm of 6] {$e$};

\node (l) [above right = 0.8cm and 0.3cm of 7] {$\rightarrow$}; 
\node (00) [right = 0.3cm of l] {$G':$};
\node[main node] (a1) [below right = 0.5cm and 0.3cm of 00] {$x$};
\node[main node] (a2) [right = 1cm of a1] {$y$}; 
\node[main node] (a3) [right = 1cm of a2] {$z$};
\node[main node] (a4) [above = 1cm of a2] {$w$};

\path[draw] 
(1) edge node[above] {\tiny{5}} (2) 
(2) edge node[right] {} (3)
(3) edge node[right] {} (3)
(3) edge node[left] {} (4)
(4) edge node[left] {\tiny{7}} (5)
(3) edge node[above] {\tiny{4}} (6)
(6) edge node[above] {} (7);

\path[draw] 
(a1) edge node[above] {\tiny{5}} (a2) 
(a2) edge node[above] {\tiny{4}} (a3)
(a2) edge node[left] {\tiny{7}} (a4);

\end{tikzpicture} 
\caption{}
\label{fig:treeContraction}
\end{figure}
\end{example}

\begin{prop}
    \label{prop:contractToDihedral}
    Let $(W,S)$ be a Coxeter system whose Coxeter graph $G$ is a tree, has no edge with infinite weight, and has at
    most one heavy edge. Then the category mod-$A$ associated to $(W,S)$ is semisimple.
\end{prop}
\begin{proof}
Let $e$ be an edge of maximal weight in $G$. In other words, let $e$ be the unique heavy edge if $G$ has one, and
let $e$ be any simple edge in $G$ otherwise. By contracting all edges different from $e$ in $G$ if necessary, we
may assume that $e$ is the only edge in $G$ and hence $\abs{S}=2$. Theorem \ref{thm:ArtinWedderburn} then implies
that mod-$A$ is semisimple.  
\end{proof}

The next two examples involve iterated quiver contractions. 

\begin{example}
\label{eg:LaurentPolynomials}
Suppose that $G=C_n(m)$ is a cycle with at most one heavy edge, where $n$ is the number of vertices in $G$ and $m$
is the maximal edge weight. In other words, suppose that $G$ has $n$ vertices $v_1,v_2,\dots, v_n$ for some $n\ge
3$ and has exactly $n$ edges $v_1-v_2,v_2-v_3,\dots,v_{n-1}-v_n,v_n-v_1$, then assume, without loss of generality,
that $G$ has edge weights $m(v_2,v_3)=m(v_3,v_4)=\dots=m(v_{n-1},v_n)=m(v_n,v_1)=3$ and $m(v_1,v_2)=m\ge 3$.  Note
that when $m=3$, the Coxeter group arising from the Coxeter diagram $C_n(m)=C_n(3)$ is exactly the affine Weyl
group of type $\tilde A_{n-1}$ for all $n\ge 3$.

By repeated simple contractions of Coxeter diagrams, we may reduce $G$ to a triangle whose double quiver is the
quiver $Q^{(1)}$ shown on the left of \autoref{fig:Laurent}.  Contracting $Q^{(1)}$  along the arrows
$\alpha_3,\beta_3$ results in the quiver  $Q^{(2)}$ in  the same figure, where $a=v_{v_1v_3}$ and the loop $\ve{v}$
is omitted as usual.  Furthermore, since $m(\alpha_2')=m(\alpha_2)=3$, we may further contract $Q^{(2)}$ along the
arrows $\alpha_2',\beta_2'$ to obtain the quiver $Q^{(3)}$, where $b=v_{av_2}$, the loop $\ve{b}$ is omitted, and
the arrows $\alpha_1'',\beta_1''$ are dual to each other and have weight $m$.  Note that as the last contraction
demonstrates, given a contractible pair of arrows $\alpha:a\ra b,\beta:b\ra a$ in a generalized double quiver $Q$,
every pair of dual arrows of the form $\gamma:a\ra b, \delta: b\ra a$ where $\gamma\neq \alpha$ is rerouted to a
pair of distinct loops $\gamma',\delta'$ at $v_{ab}$ which are dual to each other and have the same weight as
$\gamma$ and $\delta$.

\begin{figure}[h!]
    \centering
    \begin{tikzpicture}        
        \node (g) {$Q^{(1)}:$};
        \node (1) [below right = 0.6cm and 0.05cm of g] {$v_1$};
\node (2) [right = 2cm of 1] {$v_2$};
\node (3) [above right = 1.8cm and 0.65cm of 1] {$v_3$};
\node (l) [above right = 0.6cm and 0.05cm of 2] {$\rightarrow$};
\node (a0) [right=0.05cm of l] {$Q^{(2)}:$};
\node (a1) [right =0.1cm of a0] {$a$};
\node (a2) [right = 2cm of a1] {$v_2$};
\node (ll) [right = 0.05cm of a2] {$\rightarrow$};
\node (aa0) [right=0.05cm of ll] {$Q^{(3)}:$};
\node (aa1) [right =0.05cm of aa0] {$b$};

\node (m1) [below right = 0.1cm and 0.7cm of 1] {\tiny{$m$}};

\node (m2) [below right = 0.8cm and 0.75cm of a1] {\tiny{$m$}};

\node (m3) [above right = 0.3cm and 0.05cm of aa1] {\tiny{$m$}};
\node (m4) [below right = 0.3cm and 0.05cm of aa1] {\tiny{$m$}};

\path[draw] 
(1) edge [bend left = 15,-stealth] node [fill=white, anchor=center, pos=0.5,
inner sep=0.5pt] {\tiny{$\alpha_1$}} (2) 
(2) edge [bend left = 15,-stealth] node [fill=white, anchor=center, pos=0.5,
inner sep=0.5pt] {\tiny{$\beta_1$}} (1)
(2) edge [bend left = 15,-stealth] node [fill=white, anchor=center, pos=0.5,
inner sep=0.5pt] {\tiny{$\alpha_2$}} (3) 
(3) edge [bend left = 15,-stealth] node [fill=white, anchor=center, pos=0.5,
inner sep=0.5pt] {\tiny{$\beta_2$}} (2)
(3) edge [bend left = 15,-stealth] node [fill=white, anchor=center, pos=0.5,
inner sep=0.5pt] {\tiny{$\alpha_3$}} (1) 
(1) edge [bend left = 15,-stealth] node [fill=white, anchor=center, pos=0.5,
inner sep=0.5pt] {\tiny{$\beta_3$}} (3);

\path[draw] 
(a1) edge [bend left = 75,-stealth] node [fill=white, anchor=center, pos=0.5,
inner sep=0.5pt] {\tiny{$\beta_2'$}} (a2) 
(a2) edge [bend right = 30,-stealth] node [fill=white, anchor=center, pos=0.5,
inner sep=0.5pt] {\tiny{$\alpha_2'$}} (a1)
(a1) edge [bend right = 30,-stealth] node [fill=white, anchor=center, pos=0.5,
inner sep=0.5pt] {\tiny{$\alpha_1'$}} (a2) 
(a2) edge [bend left = 70,-stealth] node [fill=white, anchor=center, pos=0.5,
inner sep=0.5pt] {\tiny{$\beta_1'$}} (a1);

\path[draw] 
(aa1) edge [loop,looseness=15, out=105, in =75, distance=1.2cm,
-stealth] node [above] {\tiny{$\alpha_1''$}} (aa1)
(aa1) edge [loop,looseness=15, out=-75, in =-105, distance=1.2cm,
-stealth] node [below] {\tiny{$\beta_1''$}} (aa1);
    \end{tikzpicture}
    \caption{}
    \label{fig:Laurent}
\end{figure}

By Theorem \ref{thm:MoritaEquivalence}, the algebra $A$ associated to the Coxeter system whose Coxeter diagram is
the cycle $G$ is Morita equivalent to the algebra $\bar A:=K Q^{(3)}/\cali^{(3)}_f$ where $\cali^{(3)}_f$ is the
evaluation ideal of $\{f_n\}$ in $KQ^{(3)}$. The path algebra $K Q^{(3)}$ is isomorphic to the free unital
associative algebra $K\ip{x,y}$ on two variables via the identification  $e_b\mapsto 1, \alpha_1''\mapsto
x,\beta_1''\mapsto y$, and $\cali^{(3)}_f$ is given by 
\[
    \cali^{(3)}_f:=\ip{(\alpha_1''\beta_1'')^k-e_b,(\beta_1''\alpha_1'')^k-e_b}
\]
where $k=(m-1)/2$, so $\bar A$ is isomorphic to the algebra
\begin{equation}
    \label{eq:Tk}
    T_k:=K\ip{x,y}/\ip{(xy)^k=(yx)^k=1}.
\end{equation}
In particular, if $m=3$, then $k=1$ and $A$ is isomorphic to the Laurent polynomial algebra $K[t,t\inverse]$ via
the identification $e_b\mapsto 1, \alpha_1''\mapsto t, \beta_1''\mapsto t\inverse$. To summarize, we have just
proved the following result.

\begin{prop}
    \label{prop:Tk}
    Let $n\ge 3$. Let $m\ge 3$ be an odd integer, let $k=(m-1)/2$, and let $(W,S)$ be the Coxeter system with
    Coxeter diagram $G=C_n(m)$. Then the algebra $A$ associated to $(W,S)$ is Morita equivalent to the algebra
    $T_k$ defined by Equation \eqref{eq:Tk}. In particular, the Morita equivalence class of $A$ does not depend on
    the value of $n$, and if $m=3$, i.e., if $(W,S)$ is of type $\tilde A_{n-1}$, then $A$ is Morita equivalent to
    $K[t,t\inverse]$.
\end{prop}
\end{example}

\begin{example}
\label{eg:freeProduct}
Apart from the algebras of the form $T_k$ from Equation \eqref{eq:Tk}, 
group algebras of free products of finite cyclic groups can also be realized as the
algebras of the form $A$ associated to Coxeter systems up to Morita equivalence. To see
this, let $A_k$ be the $K$-algebra given by the presentation 
\[
A_k: =\ip{x: x^k=1}
\]
for each integer $k>1$, and let 
\[
A_{\mathbf{k}} = \ip{x_j: 1\le j\le n, x_j^{k_j}=1}
\]
for each tuple $\mathbf{k}=(k_1,\dots,k_n)$ where $n\ge 1$ and $k_i\in \Z_{\ge 1}$ for all $1\le i\le n$.  Then
$A_k$ is isomorphic to the group algebra of the cyclic group $C_k$ of order $k$, and $A_{\mathbf{k}}$ is isomorphic
to the group algebra of the free product $C_{\mathbf{k}}:=C_{k_1}*\dots*C_{k_n}$ of $C_{k_1},\dots, C_{k_n}$. For
each tuple $\mathbf{k}=(k_1,\dots,k_n)$, let $(W,S)$ be the Coxeter system where $S= \{0,1,2,\dots, n\}$,
$m_j:=m(0,j)=2k_j+1$ for each $1\le j\le n$, and $m(i,j)=2$ for all $1\le i<j\le n$. Let $Q$ be the double quiver
of $(W,S)$, and for each $1\le j\le n$ denote the arrows $0\ra j$ and  $j\ra 0$ in $Q$ by $\alpha_j$ and $\beta_j$,
respectively. It is easy to see that by starting from the quiver $Q$ and performing successive contractions, first
along $\{\alpha_1,\beta_1\}$, then along (the reroutes of) $\{\alpha_2,\beta_2\}$, then along (the reroutes of)
$\{\alpha_3,\beta_3\}$, and so on, we can transform $Q$ to a generalized double quiver $\bar Q$ with a single
vertex $v$ and $n$ self-dual loops $\ve{1},\dots, \ve{n}$ at $v$ of weight $m_1,m_2,\dots,m_n$, respectively.
\autoref{fig:starQuiver} demonstrates the construction of $\bar Q$ from the Coxeter diagram $G$ of $(W,S)$ when
$n=4$.

By Theorem \ref{thm:MoritaEquivalence}, the algebra $A$ is Morita equivalent to the quotient $K \bar Q/\bar\cali_f$
where $\bar\cali_f$ is the evaluation ideal of $\{f_n\}$ in $K \bar Q$. The following proposition is now immediate.

\begin{prop}
    Let $(W,S)$ be as described in the above paragraph. Then the algebra $A$ associated to $(W,S)$ is  Morita
    equivalent to the algebra $A_\mathbf{k}$.
\end{prop}

\begin{proof}
    Note that $K\bar Q$ is the free unital associative algebra generated by the loops $\ve{j}$ where $1\le j\le n$
and that  $\bar\cali_f=\ip{\ve{j}^{k_j}-e_v:1\le j\le n}$.  It follows that we may induce an algebra isomorphism
$\varphi: A_{\mathbf{k}}\ra K \bar Q/\bar\cali_f$ from the assignment $\varphi(x_j)=\ve{j}$ for all $1\le j\le n$.
\end{proof}

\begin{figure}[h!]
    \centering
    \begin{tikzpicture}        
\node (g) {$G:$};
\node (1) [above right = 0.8cm and 0.1cm of g] {$1$};
\node (0) [below right = 0.8cm and 0.8cm of 1] {$0$};
\node (2) [right = 2cm of 1] {$2$};
\node (3) [below = 2cm of 2] {$3$};
\node (4) [left = 2cm of 3] {$4$};
\node (l) [below right = 0.8cm and 0.1cm of 2] {$\rightarrow$};
\node (a00) [right=0.1cm of l] {$Q:$};
\node[main node] (a5) [above right = 0.8cm and 0.1 cm of a00] {$1$};
\node[main node] (a6) [right = 2cm of a5] {$2$};
\node[main node] (a7) [below = 2cm of a6] {$3$};
\node[main node] (a8) [left = 2cm of a7] {$4$}; 
\node[main node] (a0) [below right =0.9cm and 0.9cm of a5] {$0$};
\node (bl) [below right = 0.8cm and 0.1cm of a6] {$\rightarrow$};
\node (aa00) [right=0.1cm of bl] {$\bar Q:$};
\node (b0) [right = 1cm of aa00] {$v$};
\node (m1) [below=0.4cm of a5] {\tiny{$m_1$}};
\node (m2) [above=0.4cm of a8] {\tiny{$m_4$}};
\node (m3) [below=0.4cm of a6] {\tiny{$m_2$}};
\node (m4) [above=0.4cm of a7] {\tiny{$m_3$}};
\node (mm1) [above left = 0.005cm and 0.35cm of b0] {\tiny{$m_1$}};
\node (mm2) [above right = 0.005cm and 0.35cm of b0] {\tiny{$m_2$}};
\node (mm3) [below left = 0.005cm and 0.35cm of b0] {\tiny{$m_3$}};
\node (mm4) [below right = 0.005cm and 0.35cm of b0] {\tiny{$m_4$}};

\path[draw]
(1) edge node [left] {\tiny{$m_1$}} (0)
(2) edge node [right] {\tiny{$m_2$}} (0)
(3) edge node [right] {\tiny{$m_3$}} (0)
(4) edge node [left] {\tiny{$m_4$}} (0);

\path[draw] 
(a0) edge [bend left = 20,-stealth] node [fill=white, anchor=center, pos=0.5,
inner sep=0.5pt] {\tiny{$\alpha_1$}} (a5) 
(a5) edge [bend left = 20,-stealth] node [fill=white, anchor=center, pos=0.5,
inner sep=0.5pt] {\tiny{$\beta_1$}} (a0)
(a0) edge [bend left = 20,-stealth] node [fill=white, anchor=center, pos=0.5,
inner sep=0.5pt] {\tiny{$\alpha_2$}} (a6) 
(a6) edge [bend left = 20,-stealth] node [fill=white, anchor=center, pos=0.5,
inner sep=0.5pt] {\tiny{$\beta_2$}} (a0)
(a0) edge [bend left = 20,-stealth] node [fill=white, anchor=center, pos=0.5,
inner sep=0.5pt] {\tiny{$\alpha_3$}} (a7) 
(a7) edge [bend left = 20,-stealth] node [fill=white, anchor=center, pos=0.5,
inner sep=0.5pt] {\tiny{$\beta_3$}} (a0)
(a0) edge [bend left = 20,-stealth] node [fill=white, anchor=center, pos=0.5,
inner sep=0.5pt] {\tiny{$\alpha_4$}} (a8) 
(a8) edge [bend left = 20,-stealth] node [fill=white, anchor=center, pos=0.5,
inner sep=0.5pt] {\tiny{$\beta_4$}} (a0);

\path[draw] 
(b0) edge [loop,looseness=15, out=150, in =120, distance=1.2cm,
-stealth] node [above left] {\tiny{$\ve{1}$}} (b0)
(b0) edge [loop,looseness=15, out=60, in =30, distance=1.2cm,
-stealth] node [above right] {\tiny{$\ve{2}$}} (b0)
(b0) edge [loop,looseness=15, out=-30, in =-60, distance=1.2cm,
-stealth] node [below right] {\tiny{$\ve{3}$}} (b0)
(b0) edge [loop,looseness=15, out=-120, in =-150, distance=1.2cm,
-stealth] node [below left] {\tiny{$\ve{4}$}} (b0);
    \end{tikzpicture}
    \caption{}
    \label{fig:starQuiver}
\end{figure}
\end{example}

\begin{remark}
\label{rmk:operators}
A pleasant feature of iterated quiver contractions is that for every sequence of the form
\eqref{eq:repeatedContractions}, since we reduce the number of vertices in the quiver with every contraction,
representations in the category $\rep_K(Q^{(n)},\cali^{(n)}_f)$ are often relatively easy to describe.  For instance, in
Example \ref{eg:freeProduct} a representation in $\rep_K(Q^{(n)},\cali^{(n)}_f)$ is simply the data of a vector
space $M_v$ and endormophisms $\phi_1, \dots\phi_n$ of $M_v$ where $\phi_j^{k_j}=\id$ for all $1\le j\le n$.  In
Example \ref{eg:LaurentPolynomials}, to define a representation in $\rep_K(Q^{(3)},\cali^{(3)}_f)$ it suffices to
specify a space $M_{b}$ and two endormorphisms $M_{\alpha_1''}, M_{\beta_1''}$ of $M_{v_1}$ satisfying the
relations $f_{m-1}(\alpha_1'',\beta_1'')$ and $f_{m-1}(\beta_1'',\alpha_1'')$.  Finally, for the Coxeter system
$(W,S)$ from Example \ref{eg:treeContraction}, by repeated contractions along arrows corresponding to the edges of
weight 5 and 7 in $G'$, we may transform the double quiver of $G'$ to a quiver $\bar Q$ of the form shown in
\autoref{fig:57}, where the loop $\ve{1}$ is self-dual and has weight 7, the loop $\ve{2}$ is self-dual and has
weight 5, and $\alpha,\beta$ are dual to each other and have weight 4.  It follows that mod-$A$ is equivalent to
the category $\rep_K(\bar Q,\bar\cali_f)$ where
\[
    \bar\cali_f=\ip{\ve{1}^2-e_v, \ve{2}^3-e_v,
    \alpha\beta\alpha-\alpha,\beta\alpha\beta-\beta}.
\]
A representation in the latter category is then simply the data of two vector spaces $M_v, M_z$, two operators
$M_{\ve{1}}, M_{\ve{2}}$ on $M_v$ and two maps $M_\alpha:M_v\ra M_z, M_\beta:M_z\ra M_v$ which satisfy the four
relations in the above equation.  In all these examples, the representations are much easier to describe than those
in the category $\rep_K(Q,\cali_f)$ attached to the original double quiver $Q$.  Endomorphism like $\phi_j$,
$M_\alpha M_\beta, \pbpa$ and $M_{\ve{1}}, M_{\ve{2}}$ will be key tools in our study of mod-$A$; we will elaborate
on their use in \autoref{sec:repAnalysis}.

\begin{figure}[h!]
    \centering
    \begin{tikzpicture}        
\node (g) {$\bar Q:$};
\node (1) [below right = 0.25cm and 2cm of g] {$v$};
\node (2) [right = 1.5cm of 1] {$z$};
\node (7) [above right = 0.25cm and 0.001cm of 1] {\tiny{7}};
\node (7) [above left = 0.001cm and 0.4cm of 1] {\tiny{5}};
\node (7) [above right = 0.001cm and 0.55cm of 1] {\tiny{4}};
\path[draw] 
(1) edge [bend left = 15,-stealth] node [fill=white, anchor=center, pos=0.5,
inner sep=0.5pt] {\tiny{$\alpha$}} (2) 
(2) edge [bend left = 15,-stealth] node [fill=white, anchor=center, pos=0.5,
inner sep=0.5pt] {\tiny{$\beta$}} (1)
(1) edge [loop,looseness=15, out=195, in =165, distance=1.5cm,
-stealth] node [left] {\tiny{$\ve{1}$}} (1)
(1) edge [loop,looseness=15, out=105, in =75, distance=1.5cm,
-stealth] node [above] {\tiny{$\ve{2}$}} (1);
    \end{tikzpicture}
    \caption{}
    \label{fig:57}
\end{figure}
\end{remark}

\subsection{Proof of Theorem \ref{thm:MoritaEquivalence}}
\label{sec:MoritaEquivalence}

Let $(Q,d,m)$ be a generalized double quiver and let $(\bar Q,\bar d, \bar m)$ be its contraction along a
contractible pair of arrows $\{\alpha:a\ra b,\beta: b\ra a\}$. We prove Theorem \ref{thm:MoritaEquivalence} in this
subsection. To begin, we introduce an intermediate generalized double quiver $(Q',d',m')$ where
\begin{enumerate}
    \item $Q'_0=Q_0$, and $Q_1'$ is defined as follows: write                
        \[
            c'=
\begin{cases}
    a & \text{if } c=b;\\
    c & \text{otherwise}\\
\end{cases}
        \]
        for all $c\in Q_0$, 
        let 
        \[
\gamma'=
\begin{cases} 
    \gamma & \text{if } \gamma\in \{\alpha,\beta\};\\
    (c'\ra d')& \text{otherwise, if $\gamma$ is of the form $c\ra d$}\\
\end{cases}
        \]
        for each arrow $\gamma\in Q_1$, 
        then set 
        \[
            Q_1'=\{\gamma':\gamma\in Q_1\}.
        \]
    \item $d'$ is defined by $d'(\gamma')=[d(\gamma)]'$
        for all $\gamma\in Q'_1$;
    \item $m'$ is defined by $m'(\gamma')=m(\gamma)$ 
        for all $\gamma\in Q_1$.
\end{enumerate}
Intuitively, we consider $Q'$ an intermediate rerouted version of $Q$ similar to $\bar Q$: in $\bar Q$, we
identify $a,b$ with $v=v_{ab}$ and ``transfer'' all data relevant to $a$ or $b$ in $Q$ to $v$ by rerouting all
arrows in $Q$ incident to $a$ or $b$ to $v$; in $Q'$, however, we transfer almost all data relevant to $b$ to $a$
except for the arrows $\alpha$ and $\beta$. We may thus think of $a\in Q'_0$ as a partial copy of $a,b\in Q_0$ and
$v\in \bar Q_0$ as a complete copy of $a,b$. To obtain $\bar Q$ from $Q'$, it remains to rename $a$ as $v_{ab}$,
rename each arrow $\gamma\in Q'_1\setminus\{\alpha,\beta\}$ incident to $a$ as $\gamma'$, replace $\alpha',\beta'$
with $\ve{{ab}}$, and remove $b$. For the quiver $Q$ from Example \ref{eg:simpleContractions}, this procedure,
along with the construction of $Q'$ from $Q$, is illustrated in \autoref{fig:intermediate},  where $v=v_{ab}$ and
$\ve{}=\ve{ab}$.

\begin{figure}[h!]
    \centering
    \begin{tikzpicture}        
\node (00) {$Q:$};
\node[main node] (a5) [above right = 0.65cm and 0.5 cm of 00] {$a$};
\node[main node] (a8) [below = 1.8cm of a5] {$d$}; 
\node[main node] (a7) [right = 1.8cm of a8] {$c$};
\node[main node] (a6) [above = 1.8cm of a7] {$b$};
\node (xx) [below left = 0.75cm and 0.15cm of a5] {\tiny{4}};
\node (yy) [below right = 0.1cm and 0.75cm of a8] {\tiny{5}};

\path[draw] 
(a6) edge [bend left = 10,-stealth] node [fill=white, anchor=center, pos=0.5,
inner sep=0.5pt] {\tiny{$\beta$}} (a5) 
(a5) edge [bend left = 10,-stealth] node [fill=white, anchor=center, pos=0.5,
inner sep=0.5pt] {\tiny{$\alpha$}} (a6) 
(a6) edge [bend left=10,-stealth] node [fill=white, anchor=center, pos=0.5,
inner sep=0.5pt] {\tiny{$\gamma$}} (a7) 
(a7) edge [bend left= 10,-stealth] node [fill=white, anchor=center, pos=0.5,
inner sep=0.5pt] {\tiny{$\delta$}} (a6) 
(a7) edge [bend left = 10,-stealth] node [fill=white, anchor=center, pos=0.5,
inner sep=0.5pt] {\tiny{$\zeta$}} (a8) 
(a8) edge [bend left = 10,-stealth] node [fill=white, anchor=center, pos=0.5,
inner sep=0.5pt] {\tiny{$\eta$}} (a7) 
(a8) edge [bend left = 10,-stealth] node [fill=white, anchor=center, pos=0.5,
inner sep=0.5pt] {\tiny{$\kappa$}} (a5) 
(a5) edge [bend left = 10,-stealth] node [fill=white, anchor=center, pos=0.5,
inner sep=0.5pt] {\tiny{$\lambda$}} (a8);

\node (l) [below right = 0.75cm and 0.1cm of a6] {$\rightarrow$};
\node (a00) [right=0.1cm of l] {$Q':$};
\node[main node] (b5) [above right = 0.65cm and 0.5 cm of a00] {$a$};
\node[main node] (b8) [below = 1.8cm of b5] {$d$}; 
\node[main node] (b7) [right = 1.8cm of b8] {$c$};
\node[main node] (b6) [above = 1.8cm of b7] {$b$};
\node (x) [below left = 0.75cm and 0.15cm of b5] {\tiny{4}};
\node (y) [below right = 0.1cm and 0.75cm of b8] {\tiny{5}};

\path[draw] 
(b6) edge [bend left = 10,-stealth] node [fill=white, anchor=center, pos=0.5,
inner sep=0.5pt] {\tiny{$\beta$}} (b5) 
(b5) edge [bend left = 10,-stealth] node [fill=white, anchor=center, pos=0.5,
inner sep=0.5pt] {\tiny{$\alpha$}} (b6) 
(b5) edge [bend left=10,-stealth] node [fill=white, anchor=center, pos=0.5,
inner sep=0.5pt] {\tiny{$\gamma'$}} (b7) 
(b7) edge [bend left= 10,-stealth] node [fill=white, anchor=center, pos=0.5,
inner sep=0.5pt] {\tiny{$\delta'$}} (b5) 
(b7) edge [bend left = 10,-stealth] node [fill=white, anchor=center, pos=0.5,
inner sep=0.5pt] {\tiny{$\zeta$}} (b8) 
(b8) edge [bend left = 10,-stealth] node [fill=white, anchor=center, pos=0.5,
inner sep=0.5pt] {\tiny{$\eta$}} (b7) 
(b8) edge [bend left = 10,-stealth] node [fill=white, anchor=center, pos=0.5,
inner sep=0.5pt] {\tiny{$\kappa'$}} (b5) 
(b5) edge [bend left = 10,-stealth] node [fill=white, anchor=center, pos=0.5,
inner sep=0.5pt] {\tiny{$\lambda'$}} (b8); 

\node (bl) [below right = 0.75cm and 0.1cm of b6] {$\rightarrow$};
\node (a00) [right=0.1cm of bl] {$\bar Q:$};
\node[main node] (aa5) [above right = 0.5cm and 1.3 cm of a00] {$v$};
\node[main node] (aa7) [below right = 1.8cm and 0.9cm of aa5] {$c$};
\node[main node] (aa8) [below left = 1.8cm and 0.9cm of aa5] {$d$};

\node (ax) [below left = 0.6cm and 0.7cm of aa5] {\tiny{4}};
\node (ay) [below right = 0.1cm and 0.85cm of aa8] {\tiny{5}};

\path[draw] 
(aa5) edge [loop above,looseness=15, out=120, in =60, distance=0.8cm, -stealth]
node {\tiny{$\varepsilon$}} (aa5)
(aa5) edge [bend left=10,-stealth] node [fill=white, anchor=center, pos=0.5,
inner sep=0.5pt] {\tiny{$\gamma'$}} (aa7) 
(aa7) edge [bend left= 10,-stealth] node [fill=white, anchor=center, pos=0.5,
inner sep=0.5pt] {\tiny{$\delta'$}} (aa5) 
(aa7) edge [bend left = 10,-stealth] node [fill=white, anchor=center, pos=0.5,
inner sep=0.5pt] {\tiny{$\zeta$}} (aa8) 
(aa8) edge [bend left = 10,-stealth] node [fill=white, anchor=center, pos=0.5,
inner sep=0.5pt] {\tiny{$\eta$}} (aa7) 
(aa8) edge [bend left = 10,-stealth] node [fill=white, anchor=center, pos=0.5,
inner sep=0.5pt] {\tiny{$\kappa'$}} (aa5) 
(aa5) edge [bend left = 10,-stealth] node [fill=white, anchor=center, pos=0.5,
inner sep=0.5pt] {\tiny{$\lambda'$}} (aa8); 
    \end{tikzpicture}
    \caption{}
    \label{fig:intermediate}
\end{figure}

Maintain the notation of Theorem \ref{thm:MoritaEquivalence}, and let $\cali'_f$ be its evaluation ideal of $K Q'$
associated to the polynomials $\{f_n\}$. We show below that the algebra $K Q'/\cali'_f$ is both isomorphic to $K
Q/\cali_f$ and Morita equivalent to $K\bar Q/\bar\cali_f$; Theorem \ref{thm:MoritaEquivalence} immediately follows.
Note that by the last paragraph, we ``favored'' the vertex $a$ in the construction of $Q'$ by choosing it as the
partial copy of $a$ and $b$ before renaming it $v_{ab}$ in $\bar Q$, and we could equally have chosen to favor $b$
in a similar way.  This choice is insignificant in that if we favored $b$ when constructing $Q'$ then the promised
isomorphism and Morita equivalence still hold. In particular, we emphasize that in the direct construction of $\bar
Q$ described by Definition \ref{def:quiverContraction}, the vertices $a$ and $b$ clearly play equal roles, so a
quiver contraction is insensitive to the choice of the favored endpoint of the contractible pair of arrows. 

\begin{prop}
    \label{prop:rerouteIso}
    Maintain the setting of Theorem \ref{thm:MoritaEquivalence}. Then the algebra $K Q'/\cali'_f$ is isomorphic to
$K Q/\cali_f$.  
\end{prop}

\begin{proof}
We will construct mutually inverse homomorphisms $\Phi: K Q/\cali_f \ra K Q'/\cali'_f$ and $\Psi: K Q'/\cali'_f\ra
K Q/\cali_f$ to prove the proposition. To do so, we use presentations of the algebras as usual: since $K Q$ is the
algebra generated by the set $Q^{\le 1} =\{e_u: u\in Q_0\}\cup Q_1$ subject only to the relations that
$e_ue_v=\delta_{u,v}e_u$ for all $u,v\in Q_0$ and the relations $e_u\gamma =\gamma =\gamma e_v$ for each arrow
$\gamma: u \ra v$ in $Q_1$, the algebra $K Q/\cali_f$ is generated by the same set $Q^{\le 1}$ subject to the above
relations and the relations $r_f(\alpha)$ for all $\alpha\in Q_1$.  We may therefore construct $\Phi$ by inducing
it from a function $\varphi: Q^{\le 1} \ra K Q'/\cali_f'$ which respects all the necessary relations.  Similarly,
we may construct the homomorphism $\Psi$ from a function $\psi: Q'^{\le 1}=\{e_s:s\in Q'_0\}\cup Q'_1\ra K
Q/\cali_f$ which respects the necessary conditions.

Let $m=m(a,b)$.  Then $m\ge 3$ and $m$ is odd by Definition \ref{def:quiverContraction}. By scaling if necessary,
we may assume the polynomial $f_{m-1}$ has constant term $-1$, in which case the relations
$r_f(\alpha),r_f(\beta)\in \cali_f$ must be of the form 
\[
    f_{m-1}(\alpha,\beta)= g(\alpha\beta)\alpha \beta-e_a, \quad
    f_{m-1}(\beta,\alpha)=g(\beta\alpha)\beta\alpha-e_b=\beta g(\alpha\beta)\alpha -e_b
\]
for some polynomial $g\in K[x]$, respectively. Let 
\[
    \sigma_1=g(\alpha\beta)\alpha,\quad \sigma_2=\beta.
\]
Then $\sigma_1,\sigma_2$ make sense in both $K Q$ and $K Q'$, and 
we have $r_f(\alpha)=\sigma_1\sigma_2-e_a, r_f(\beta)=\sigma_2\sigma_1-e_b$, so that
\begin{equation}
     \sigma_1\sigma_2 =e_a, \quad\sigma_2\sigma_1=e_b
    \label{eq:sigmaRelations}
\end{equation}
in both $K Q/\cali_f$ and $K Q'/\cali_f'$. 
To define the functions $\varphi: Q^{\le 1} \ra K Q'/\cali'_f$ and $\psi: Q'^{\le 1}\ra KQ/\cali_f$, first let 
 \[
            X_+=\{\gamma\in Q_1: \source(\gamma)=b\}\setminus\{\beta\},
        \]
        and
        \[
            X_-=\{\gamma\in Q_1: \target(\gamma)=b\}\setminus\{\alpha\}.
        \]
        Note that the set $X_+\cap X_{-}$ consists of all loops at $b$ in
        $Q_1$, and each loop in it is rerouted to a loop at $a$ in $Q'$. 
        Next, let $\varphi(e_u)=e_u$ for all $u\in Q_0=Q'_0$. Finally, recall
        that  $Q'_1=\{\gamma':\gamma\in Q_1\}$ and define  
        $\varphi$ and $\psi$ on $Q_1$ and $Q_1'$ by letting 
\[
    \varphi(\gamma) = 
    \begin{cases}
    \sigma_2\gamma'\sigma_1 & \text{if $\gamma\in X_+\cap X_-$};\\
        \sigma_2 \gamma' & \text{if $\gamma\in X_+\setminus X_-$};\\
        \gamma'\sigma_1 & \text{if $\gamma\in X_-\setminus X_+$};\\
        \gamma & \text{otherwise},\\
    \end{cases}\quad
    \psi(\gamma') = 
    \begin{cases}
    \sigma_1\gamma\sigma_2 & \text{if $\gamma\in X_+\cap X_-$};\\
        \sigma_1 \gamma & \text{if $\gamma\in X_+\setminus X_-$};\\
        \gamma\sigma_2 & \text{if $\gamma\in X_-\setminus X_+$};\\
        \gamma & \text{otherwise}\\
    \end{cases}
\]
for each $\gamma\in Q_1$. Using the relations in \eqref{eq:sigmaRelations}, it is straightforward to verify that
$\varphi$ and $\psi$ respect all necessary relations mentioned in the previous paragraph and induce mutually
inverse algebra homomorphisms $\Phi: K Q/\cali_f\ra  K Q'/\cali'_f$ and $\Psi:K Q'/\cali'_f\ra K Q/\cali_f$, as
desired.  
\end{proof}

\begin{prop} 
  \label{prop:MoritaEquivalence} 
  Maintain the setting of Theorem \ref{thm:MoritaEquivalence}. Then the algebra $K \bar Q/\cali(\bar\calr)$ is
  Morita equivalent to $K Q'/\cali(\calr')$.  
\end{prop}

\begin{proof}
Let $\Lambda= K Q'/\cali'_f$ and let $\sigma_1, \sigma_2$ be as in the proof of Proposition \ref{prop:rerouteIso}.
Since $\sigma_1\sigma_2 = e_a$ and $\sigma_2\sigma_1 = e_b$ in $\Lambda$, the maps $\phi_1: e_a\Lambda \ra
e_b\Lambda, x\mapsto \sigma_2 x$ and $\phi_2:e_b\Lambda \ra e_a\Lambda, y\mapsto \sigma_1 y$ give mutually inverse
isomorphisms between the projective modules $\Lambda$-modules $e_a \Lambda$ and $e_b\Lambda$.  Set
$V_b=Q'_0\setminus\{b\}$ and let  
\[
e=1-e_b=\sum_{u\in V_b}e_u\in \Lambda.  
\] 
Since $e_a\Lambda\cong e_b\Lambda$, the submodule $\Lambda':=e\Lambda = \oplus_{u\in V_b} (e_u\Lambda)$ of the
regular module $\Lambda$ is a progenerator in the category of $\Lambda$-modules, therefore $\Lambda$ is Morita
equivalent to the endomorphism algebra $\mathrm{End}_\Lambda(\Lambda')$.  We have
$\mathrm{End}_\Lambda(\Lambda')\cong e\Lambda e$ since $e$ is an idempotent, so to prove the proposition it
suffices to show that $K \bar Q/\cali(\bar \calr)$ is isomorphic to $e\Lambda e$. We will do so by inducing a
homomorphism $\Phi: K \bar Q/\cali(\bar\calr) \ra e\Lambda e$ from a function $\phi: \bar Q^{\le 1} :=\{e_u: u\in
\bar Q_0\}\cup \bar Q_1 \ra e\Lambda e$ and then showing that $\Phi$ is bijective. 

Let $v=v_{ab}$ and $\ve{}=\ve{ab}$.  To define the function $\phi: \{e_u: u\in \bar Q_0\}\cup \bar Q_1 \ra e\Lambda
e$, let $\varphi(e_v)=e_a$ and let $\phi(e_u)=e_u$ for all $u\in \bar Q_0\setminus\{v\}$, then let $\phi(\ve{}) =
\alpha\beta$ and let $\phi(\gamma)= \gamma$ for all $\gamma\in \bar Q_1\setminus\{\ve{}\}$.  Viewing $K\bar
Q/\bar\cali_f$ as the algebra generated by $\bar Q^{\le 1}$ subject to the suitable relations as usual, we can
again check that $\phi$ respects all these relations: indeed, by the definitions of $\bar Q$ and $Q'$, it suffices
to check only the relations involving the loop $\ve{}\in \bar Q_1$, i.e., the relation $e_v \ve{} = \ve{} = \ve{}
e_v$ and the relation $r_f(\ve{})=\tilde{f}_{m(\alpha)-1}(\ve{})$. These relations are respected by $\phi$ since
$\phi(e_v)=e_a, \phi(\ve{})=\alpha\beta$, $e_a\alpha\beta=\alpha\beta=\alpha\beta e_a$, and $\tilde
f_{m(\alpha)-1}(\alpha\beta)=r_f(\alpha)\in \cali_f$ (see Remark \ref{rmk:sqrt}). It follows that $\Phi$ induces a
unique algebra homomorphism $\Phi: K \bar Q/\bar\cali_f \ra e\Lambda e$.  

To prove that $\Phi$ is bijective, we keep the notation from Definition \ref{def:quiverContraction} and from the
definition of $Q'$. Let $X=Q_1\setminus\{\alpha,\beta\}$. Then 
\[
\bar Q_1=\{\gamma':\gamma\in X\}\sqcup\{\ve{}\},\quad
Q'_1=\{\gamma':\gamma\in X\}\sqcup\{\alpha,\beta\},
\]
and $\Phi(\gamma')=\gamma'$ for all $\gamma\in X$ (where the $\gamma'$s stand for their respective images in $K
\bar Q/\bar\cali_I$ and $e\Lambda e$).  Let $\calp_b$ be the set of all paths on $Q'$ which both start and end at a
vertex in $V_b$. Then $e\Lambda e$ is spanned by the classes of paths in $\calp_b$. Now, since $\alpha,\beta$ are
the only arrows in $Q'$ with $b$ as its target and source, respectively, if a path $p\in \calp_b$ passes $b$ at any
point then it must have traveled to $b$ from $a$ via $\alpha$ and then immediately traveled back to $a$ via
$\beta$.  Consequently, $p$ is a product of $\alpha\beta=\Phi(\ve{})$ and arrows from the set $X=\Phi(X)\se \image
\Phi$. It follows that $p\in \image \Phi$, so $\Phi$ is surjective.

It remains to prove that $\Phi: K \bar Q/\bar\cali_f\ra e\Lambda e $ is injective. Since $\Lambda = K Q'/\cali'_f$,
it suffices to show that $e\cali'_fe \se \Phi(\bar\cali_f)$. Since $Q'_1=\{\alpha,\beta\}\sqcup\{\gamma':\gamma\in
X\}$, the set $e\cali_f'e$ is spanned by nonzero elements of the form 
\begin{equation}
    y_1 = p_1 [r_f(\alpha)] q_1=p_1[\Phi(f(\ve{},\ve{}))]q_1, 
    \label{eq:y1}
\end{equation}
\begin{equation}
        y_2= p_2[r_f(\beta)]q_2, 
    \label{eq:y2}
\end{equation}
and
\begin{equation}
    y_3= p_3 [r_f(\gamma')] q_3 =p_3[\Phi(r_f(\gamma))]q_3
    \label{eq:y3}
\end{equation}
where $\gamma\in X$ and $p_i,q_i$ are paths in $K Q'$ for all $i\in \{1,2,3\}$.  We need to show that
$y_1,y_2,y_3\in \Phi(\bar\cali_f)$.  Note that the following holds for all $i\in \{1,2,3\}$:
\begin{enumerate}
    \item Since 
        $e=\sum_{u\in V_b}e_u$ and $y_i\neq 0$, we have $\source(p_i)$, $
        \target(q_i)\in V_b$. 
\item Let $r_i$ be the
bracketed relation in $y_i$ in Equations \eqref{eq:y1}-\eqref{eq:y3}. Then
$\target(p_i)=\source(r_i), \target(r_i)=\source(q_i)$ since $y_i\neq 0$. In particular, we have $\target(p_1)=\source(q_1)=a\in V_b$  and
$\target(p_3)$, $\source(q_3)\in V_b$ since the rerouted arrow $\gamma'$ cannot be
incident to $b$.
\item By (a) and (b), $p_3,q_3\in \calp_b\se \image\Phi$ where the last
    containment holds by the last paragraph. Furthermore, we have
    $r_1=r_f(\alpha)=\Phi(r_f(\ve{}))$ and
    $r_3=r_f(\gamma')=\Phi(r_f(\gamma))$ by the definition of 
    $\Phi$. It follows that $y_1,y_3\in\Phi(\bar\cali_f)$, as desired.
\item By (b) we have $\target(p_2)=b=\source(q_2)$, but since $\alpha, \beta$ are the only arrows in
$Q'$ with $b$ as its target and source, respectively, we must have $p_2 = p'_2
\alpha$ and $q_2=\beta q''_2$ for some paths $p'_2, q''_2\in \calp_b$. 
 It follows that 
\begin{equation}
    y_2 = p''_2\alpha[r_f(\beta)]\beta q''_2= p'_2
    [r_f(\alpha)]\alpha\beta q''_2 = p'_2 [\Phi(r_f(\vea))] q'_2
    \label{eq:yy2}
\end{equation}
where $q'_2=\alpha\beta q''_2$. Note that $p_2',q_2'\in \calp_b\se \image \Phi$, therefore
$y_2\in \Phi(\bar\cali_f)$. 
\end{enumerate}
The proof is now complete. 
\end{proof}

\begin{remark}
    \label{rmk:fieldAssumptions}
    We have assumed that the field $K$ is algebraically closed in Theorem \ref{thm:IsoQuotients}, Theorem
    \ref{thm:ArtinWedderburn}, Theorem \ref{thm:MoritaEquivalence}, Proposition \ref{prop:rerouteIso} and
    Proposition \ref{prop:MoritaEquivalence}.  However, it is worth noting these results also hold, by the exact
    same proofs, if we assume instead that $K$ is an arbitrary field of characteristic zero and that $\{f_n\}$ is a
    family of polynomials which all split over $K$ and satisfy Conditions (a) and (b) of Definition
    \ref{def:DistPoly}. The purpose of the assumption that $K$ is algebraically closed is to guarantee that the
    polynomials $\{f_n\}$ defined by Equation \ref{eq:PowerPoly} split and hence form a uniform family over $K$;
    the simple forms of these polynomials will greatly simplify the study of representations in categories of the
    form $\rep_K(\bar Q,\bar \cali_f)$ appearing in \autoref{sec:contractionExamples} and the remaining parts of
    the paper.
\end{remark}

\subsection{Representations of Contracted Quivers}
\label{sec:repAnalysis}

Let $(\bar Q, \bar d, \bar m)$ be a generalized double quiver obtained from the double quiver $Q$ of a Coxeter
system $(W,S)$ via a sequence of contractions. Let $\bar\cali_f$ be the evaluation ideal of a uniform  family
$\{f_n\}$ of polynomials over $K$. Then the category mod-$A$ is equivalent to the category $\mathrm{rep}_K(\bar
Q,\bar\cali_f)$. We develop tools for constructing and analyzing representations in $\mathrm{rep}_K(\bar Q,\bar
\cali_f)$ in this subsection.

Let $M=(M_a,M_\alpha)_{a\in \bar Q_0,\alpha\in \bar Q_1}$ be a representation in $\mathrm{rep}_K \bar Q$. The
definition of $\bar\cali_f$ implies that $M$ is a representation in
$\brepqrf$ if and only if for every arrow of the form $\alpha:a\ra b$ in $\bar Q$, the set of assignments
\[
    M_{\{\alpha,\beta\}}:=\{M_a, M_b, M_\alpha,
M_\beta\}
\]  
where $\beta=\bar d(\alpha)$ satisfies the relations $r_f(\alpha)$ and $r_f(\beta)$ in the sense that
the maps
\begin{equation}
    \label{eq:generalizedLocalRelations1}
 f_{m-1}(M_\alpha,M_\beta):=
    \begin{cases}
    \tilde f_{m-1}(M_\alpha M_\beta) & \text{if $m$ is odd};\\
        \tilde f_{m-1}(M_\alpha M_\beta)M_\alpha & \text{if $m$ is even}
    \end{cases}
\end{equation}
and
\begin{equation}
    \label{eq:generalizedLocalRelations2}
f_{m-1}(M_\beta,M_\alpha):=
    \begin{cases}
        \tilde f_{m-1}(M_\beta M_\alpha) & \text{if $m$ is odd};\\
        \tilde f_{m-1}(M_\alpha M_\beta)M_\alpha & \text{if $m$ is even}
    \end{cases}
\end{equation}
where $m=m(\alpha)$ both equal 0. Call a set of the form $M_{\{\alpha,\beta\}}$ a \emph{local representation for
$\{\alpha,\beta\}$} if it satisfies the equations \eqref{eq:generalizedLocalRelations1} and
\eqref{eq:generalizedLocalRelations2}. Then to construct a representation $M\in \brepqrf$ it suffices to
\emph{assemble} a collection of local representations 
\[
    \mathcal{M}:=\{M_{\{\alpha,\beta\}}\,\vert\,\alpha\in
Q_1,\beta=\bar\alpha\}
\]
that is  \emph{consistent} in the sense that for every vertex $a\in \bar Q_0$, there is a common vector space $V_a$
such that $M_a=V_a$ for every local representation $M_{\alpha,\beta\}}\in \mathcal{M}$ where $\alpha$ is incident
to $a$.  Here, we \emph{assemble} a consistent collection $\mathcal{M}$ into a representation $M\in \brepqrf$ as
follows: first, for each vertex $a\in \bar Q_0$,  pick any arrow $\alpha\in \bar Q_1$ incident to $a$, denote
$M_{\{\alpha,\bar\alpha\}}$ by $M'$, then let $M_a=M'_a$; second, for each arrow $\gamma\in \bar Q_1$ and
$\beta=\bar d(\alpha)$, denote $M_{\{\alpha,\beta\}}$ by $M''$, then let $M_\alpha=M''_\alpha$ and
$M_\beta=M''_\beta$. Note that the first step can be done unambiguously, independent of the choice of the arrow
$\alpha$, if and only if $\mathcal{M}$ is consistent. Also note that in the above discussion we allow the
possibility that $\alpha$ is self-dual, in which case $\beta=\alpha$ and $a=b$.

Subsequently, we will often construct a representation in $\brepqrf$ by assembling a consistent collection of local
representations. The local representations can in turn be studied, by comparison of the equations
\eqref{eq:rf_alpha}, \eqref{eq:rf_beta} and \eqref{eq:generalizedLocalRelations1},
\eqref{eq:generalizedLocalRelations2}, similarly to how we studied $\rep_K(Q,\cali_f)$ in the dihedral case in
\autoref{sec:IsoQuotients}:  

\begin{prop}
\label{prop:edgeLoopRep}
Let $M\in \rep_K \bar Q$. Let $\alpha:a\ra b$ be an arrow in $\bar Q_1$, let $\beta=\bar d(\alpha)$, and let
$m=\bar m(\alpha)$. Let $M_{\{\alpha,\beta\}}=\{M_a,M_b,M_\alpha,M_\beta\}$.  Then the following results hold.  
\begin{enumerate}
    \item If $m=\infty$, then $M_{\{\alpha,\beta\}}$ is automatically a local representation for
        $\{\alpha,\beta\}$. If $m<\infty$, then $M_{\{\alpha,\beta\}}$ is a local representation whenever $\papb$
        and $\pbpa$ are  diagonalizable maps whose eigenvalues are roots of $\tilde f_{m-1}$. 
    \item If $M_{\{\alpha,\beta\}}$ is a local representation for $\{\alpha,\beta\}$ and $m<\infty$, then $\papb$
        and $\pbpa$ are diagonalizable and their eigenvalues are either roots of $\tilde f_{m-1}$ or zero. 
\end{enumerate}
\end{prop}
\begin{proof}
    Part (a) follows from Equations \eqref{eq:generalizedLocalRelations1} and
\eqref{eq:generalizedLocalRelations2}. Part (b) follows from the same equations and Lemma \ref{lemm:decompose}.(a).
\end{proof}

The following specializations of Proposition \ref{prop:edgeLoopRep}.(a) will be very useful:

\begin{corollary}
\label{lemm:trivialLocalReps}
Let $M, \alpha,\beta, a,b,m$ and $M_{\left\{\alpha,\beta \right\}}$
be as in Proposition \ref{prop:edgeLoopRep}. 
Let $\{f_n\}$ be the polynomial family defined by \eqref{eq:PowerPoly}.  Then the following holds for any positive
integer $n$.
\begin{enumerate}
    \item If $M_a=M_b=K^n$ and $M_\alpha=M_\beta=\id$, then
        $M_{\{\alpha,\beta\}}$ defines a local
    representation for $\{\alpha,\beta\}$.
\item  Suppose that  $m\ge 5$ and $\beta\neq \alpha$. If $M_a= M_b=K^n$ and $(M_\alpha
    M_\beta)^2=(M_\beta M_\alpha)^2=\id$, then 
 $M_{\{\alpha,\beta\}}$ defines a local
    representation for $\{\alpha,\beta\}$.
\end{enumerate}
\end{corollary}

\begin{proof}
 The results follow from Proposition \ref{prop:edgeLoopRep}.(a), the fact that
$1$ is a root for $\tilde f_{m-1}$, and the fact that $x^2-1$
divides $\tilde f_{m-1}$ whenever $m\ge 5$.
\end{proof}

To prove results on mod-$A$ in Section \autoref{sec:mod-AResults}, we will often need to not only construct a
suitable representation $M$ in $\rep_K(\bar Q, \bar \calr)$ but also prove that $M$ is simple or not semisimple.
The proofs typically proceed in the following way. Consider a specific vertex $a\in \bar Q_0$ and a number of paths
$p_1, p_2,\dots, p_k$ on $\bar Q$ that both start and end at $a$. Recall from \autoref{sec:QuiverReps} that these
paths give rise to endormophisms $\phi_1:=M_{p_1},\phi_2:=M_{p_2}, \dots$ and $\phi_k:=M_{p_k}$ of $M_a$, and that
a subrepresentation $N$ of $M$ must  assign to $a$ a vector space $N_a\se M_a$ that satisfies the invariance
condition $\phi_i(N_a)\se N_a$ for all $1\le i\le k$. Together, these invariance conditions force $N_a$ to take
certain forms, which in turn force $M$ to satisfy certain properties such as being simple. We refer to the analysis
of what form $N_a$ can take as \emph{subspace analysis at $a$}.  

We now explain how we will construct representations to facilitate successful subspace analysis via examples. All
the examples will be used in the proofs of Section \ref{sec:mod-AResults}, and $\{f_n\}$ stands for the uniform
family of polynomials defined by Equation \eqref{eq:PowerPoly} throughout the examples.  Our first method starts
with an irreducible representation $\rho: G\ra \mathrm{GL}(V)$ of a group $G$:

\begin{example}
\label{eg:Miller}
We construct a simple representation $M$ in $\brepqrf$ for the generalized double quiver $\bar Q$ in
\autoref{fig:57} from Remark \ref{rmk:operators}. To start, consider the symmetric group $G=S_q$ where $q\ge 8$ and
any irreducible representation $\rho: G\ra \mathrm{GL}(V)$ of $G$. By \cite{Miller23}, the group $S_q$ can be
generated by two elements $\sigma,\tau$ of orders 2 and 3, respectively.  It follows that
$\rho(\sigma)^2=\rho(\sigma^2)=\rho(e)=\id_V$ and similarly $\rho(\tau)^3=\id_V$. 

To define $M$, first let $M_v= V, M_{\ve{1}}=\rho(\sigma)$ and $M_{\ve{2}}=\rho(\tau)$. This defines local
representations for the sets $\{\ve{1}\}$ and $\{\ve{2}\}$ because $M_{\ve{1}}, M_{\ve{2}}$ satisfy the relations
$r_f(\ve{1})=\ve{1}^2-e_v, r_f(\ve{2})=\ve{2}^3-e_v$, respectively. To finish the definition of $M$, it remains to
assign a local representation for the dual arrows $\{\alpha,\beta\}$ that is consistent with these two local
representations.  By Corollary \ref{lemm:trivialLocalReps}, it suffices to define $M_z=V$ and $M_\alpha=M_\beta
=\id_V$. 

Let $N$ be a subrepresentation of $M$. Since the representation $V$ is irreducible and $\sigma,\tau$ generate $G$,
the only subspaces of $M_v$ that are invariant under both $M_{\ve{1}}=\rho(\sigma)$ and $M_{\ve{2}}=\rho(\tau)$ are
$0$ and $M_v$ itself, therefore we have $N_v=0$ or $N_v=M_v$.  Since $M_\alpha,M_\beta$ are isomorphisms, in these
two cases we must have $N=0$ or $N=M$, respectively, therefore $M$ is simple.  
\end{example}

The endomorphisms $\phi_1,\dots,\phi_k$ of $M_a$ mentioned above are often all diagonalizable in our examples. This
makes the following well-known fact from linear algebra very useful for subspace analysis: let $V$ be a vector
space and let $\phi$ be a diagonalizable endomorphism of $V$. Suppose $\phi$ has $d$ distinct eigenvalues and
$V=\oplus_{i=1}^d E_{i}$ is the corresponding eigenspace decomposition. Then a subspace $W$ of $V$ is invariant
under $\phi$ if and only if $W$ is \emph{compatible with the eigenspace decomposition} in the sense that
$W=\oplus_{i=1}^d (W\cap E_i)$.  We use this characterization in the following two examples. 

\begin{example}
\label{eg:twoHeavyEdges}
Let $(W,S)$ be the Coxeter system whose Coxeter diagram $G$ is shown on the left of \autoref{fig:twoHeavyEdges},
where $m_1, m_2\in  \Z_{\ge 4}\cup\{\infty\}$. The double quiver $Q$ of $G$ is shown on the right of the same
figure. We construct a representation $M$ in $\rep_K(Q,\cali_f)$ and apply subspace analysis at $b$ to show that
$\rep_K(Q,\calr)$, hence mod-$A$, is not semisimple.
 
\begin{figure}[h!] 
\centering
\begin{tikzpicture} 
\node (0) {$G:$};
\node[main node] [right =0.5cm of 0] (1) {$a$}; 
\node[main node] [right=1cm of 1] (2) {$b$}; 
\node[main node] [right=1cm of 2] (3) {$c$};
\path[draw] 
(1) edge node [above] {$m_1$} (2) 
(2) edge node [above] {$m_2$} (3); 

\node [right=0.5cm of 3]  (r) {$\ra$};
\node [right = 0.5cm of r] (q) {$Q:$};
\node[main node] [right=0.5cm of q]  (4) {$a$};
\node[main node] [right=1.5cm of 4] (5) {$b$}; 
\node[main node] [right=1.5cm of 5] (6) {$c$};

\path[draw] 
(4) edge [bend left = 15,-stealth] node [fill=white, anchor=center, pos=0.5,
inner sep=0.5pt] {\tiny{$\alpha$}} (5) 
(5) edge [bend left = 15,-stealth] node [fill=white, anchor=center, pos=0.5,
inner sep=0.5pt] {\tiny{$\beta$}} (4)
(5) edge [bend left = 15,-stealth] node [fill=white, anchor=center, pos=0.5,
inner sep=0.5pt] {\tiny{$\gamma$}} (6) 
(6) edge [bend left = 15,-stealth] node [fill=white, anchor=center, pos=0.5,
inner sep=0.5pt] {\tiny{$\delta$}} (5);
\end{tikzpicture} 
\caption{} 
\label{fig:twoHeavyEdges}
\end{figure}

We begin by constructing the local representations $M_{\{ \alpha,\beta\}}$ and $M_{\{\gamma,\delta\}}$ based on the
values of $m_1$ and $m_2$, respectively. For $\{\alpha,\beta\}$, first let $M_a=M_b=K^2$. Let $\lambda_2=1$. Take
$\lambda_1=-1$ if $m_1=\infty$, $\lambda=0$ if $m_1=4$, and $\lambda=z$, a root of $\tilde f_{m-1}$ different from
$\lambda_2$, if $4<m_1<\infty$. Next, set $M_\beta$ to be the map given by the matrix $\bm{0 & 0\\ 0 &1}$ if
$m_1=4$ and to be the identity map otherwise, then set $M_\alpha=\bm{\lambda_1 & 1\\ 0 &\lambda_2}$. It is
straightforward to check that $M_{\{\alpha,\beta\}}:=\{M_a, M_b, M_\alpha, M_\beta\}$ forms a local representation:
if $m_1=\infty$, then the relations $r_f(\alpha), r_f(\beta)$ are zero and there is nothing to check; if
$4<m_1<\infty$, then the relations $r_f(\alpha), r_f(\beta)$ are satisfied by Proposition \ref{prop:edgeLoopRep}
because $M_\alpha M_\beta = M_\beta M_\alpha = M_\alpha$, a diagonalizable map whose
eigenvalues are roots of $\tilde f_{m-1}$;  finally, if $m_1=4$, then the relations
$r_f(\alpha)=\alpha\beta\alpha-\alpha, r_f(\beta)=\beta\alpha\beta-\beta$ are satisfied because $M_\alpha M_\beta
M_\alpha=M_\alpha, M_\beta M_\alpha M_\beta=M_\beta$ by direct computation. We may similarly define a local
representation for $\{\gamma,\delta\}$ by letting $M_b=M_c=K^2$, defining numbers $\mu_2,\mu_1$ and the map
$M_\gamma$ based on $m_2$ in the same way we defined $\lambda_2,\lambda_1$ and $M_\beta$ based on $m_1$, and
defining the map $M_\delta$ to be given by the matrix $\bm{\mu_1 & 0\\ 0 & \mu_2}$.  The two local representations
are consistent because they assign the same vector space $K^2$ to the vertex $b$, therefore they can be assembled
to a representation $M\in \rep_K(Q,\cali_f)$. 

Consider the operators $\phi_1:=M_\alpha\circ M_\beta$ and $\phi_2: =M_\delta\circ M_\gamma$ on $M_b$. Then
\[ 
\phi_1= \begin{bmatrix} \lambda_1 & 1 \\ 0 & \lambda_2
\end{bmatrix},\quad \phi_1= \begin{bmatrix} \mu_1 & 0\\ 0 & \mu_2
\end{bmatrix}, 
\]
so the eigenspace decompositions of $M_b$ with respect to $\phi_1,\phi_2$ are given by 
\begin{equation}
\label{eq:MbDecomp}
M_b= \ip{e_1}\oplus
\ip{e_1+(\lambda_2-\lambda_1)e_2}=\ip{e_1}\oplus \ip{e_2}  
\end{equation}
where $\ip{v}$ stands for the span of $v$ for each vector $v$ and $e_1, e_2$ denote the standard basis vectors
$\bm{1\\0}, \bm{0\\1}$ of $K^2$, respectively. Let $N$ be a nonzero subrepresentation of $M$. Then the vector space
$N_b$ must be invariant under both $\phi_1$ and $\phi_2$. Consequently, $N_b$ must be compatible with the
decompositions in Equation \eqref{eq:MbDecomp} in the sense that 
\[
    N_b=(N_b\cap \ip{e_1})\oplus (N_b\cap \ip{e_2})=(N_b\cap\ip{e_1})\oplus
(N_b\cap\ip{e_1 + (\lambda_2 - \lambda_1)e_2}).  
\] 
But each intersection in the above equation is either trivial or of dimension 1, so $N_b$ must be $\ip{e_1}$. This
implies that $N$ cannot have a complement in $M$, so $\mathrm{rep}_K (Q,\cali_f)$ is not semisimple.  
\end{example}

\begin{example}
\label{eg:twoHeavyEdgesSimple}
Let $Q$  and $M$ be as in the previous example.  We modify $M$ to produce an infinite family of simple
representations in $\repqrf$ to be used in Section \autoref{sec:mod-AResults}. To begin, define a representation
$M^x$ for each scalar $x\in K\setminus\{\lambda_1,\lambda_2\}$ as follows: let $M^{x}_a= M^{x}_b=M^{x}_c=K^2,
M^x_\gamma=M_\gamma$ and $M^x_\delta =M_\delta$, then let $M^{x}_\beta = \id$ and let $M^x_{\alpha}$ be the map
given by the matrix 
\[
B_x= \bm{x & x(\lambda_1+\lambda_2-x)-\lambda_1\lambda_2\\ 1 &
\lambda_1+\lambda_2-x}.
\]
As before, this defines local representations $M_{\{\gamma,\delta\}}:=\{M^x_b,M^x_c,M^x_\gamma,M^x_\delta\}$ and
$M_{\{\alpha,\beta\}}:=\{M^x_a,M^x_b,M^x_\alpha, M^x_\beta\}$ for $\{\gamma,\delta\}$ and $\{\alpha,\beta\}$,
respectively.  Moreover, the two local representations are consistent and assemble to a representation $M^x$.  Let
$\phi^x_1:=M_\alpha^x \circ M_\beta^x$ and $\phi^x_2:=M^x_\delta\circ M^x_\gamma$.  The matrix $B_x$ guarantees
that the map $\phi_1^x$ has an eigenvector $v_1:=(x-\lambda_1)e_1+e_2$ with eigenvalue $\lambda_1$ as well as an
eigenvector $v_2:=(x-\lambda_2)e_1+e_2$ with eigenvalue $\lambda_2$, so the eigenspace decompositions of $M^x_b$
with respect to $\phi_1^x$ and $\phi_2^x$ are given by
\[
    M^x_b= \ip{v_1}\oplus \ip{v_2}=\ip{e_1}\oplus \ip{e_2}. 
\]
It follows that in any subrepresentation $N$ of $M^x$ we must have 
\[
    N_b=N_b\cap \ip{v_1}\oplus N_b\cap\ip{v_2}=N_b\cap \ip{e_1}\oplus N_b\cap
    \ip{e_2}. 
\] 
The vectors $v_1,v_2, e_1,e_2$ are pairwise distinct since $\lambda_1\neq \lambda_2, \mu_1\neq \mu_2$ and
$x\not\in\{\lambda_1,\lambda_2\}$, therefore the above equation holds only if $N_b=0$ or $N_b=K^2$.  If $m_2\neq
4$, then $\mu_1\neq 0$ and $M^x_\gamma$ is an isomorphism, therefore we must  have $N=0$ or $N=M^x$, which in turn
implies that $M^x$ is simple. If $m_2= 4$, it is easy to check that $M^x$ contains a simple module $N^x$ such that
$N^x_a=N^x_b=K^2$ and $N^x_c=\ip{e_2}$. Set $N^x=M^x$ when $m_2\neq 4$. Then $N^x$ is a simple representation in
$\repqrf$ for all $x\in K\setminus\{\lambda_1,\lambda_2\}$ regardless of the value of $m_2$. 

Observe that if $x\neq y$ then $N^x\not\cong N^y$, for there cannot exist linear isomorphisms $\phi_a: N^x_a\ra
N^y_a, \phi_b: N^x_b\ra N^y_b$ such that $\phi_b N^x_\alpha = N^y_\alpha \phi_a$ and $\phi_a N^x_\beta= N^y_\beta
\phi_b$ simultaneously: the latter equation holds only if $\phi_a=\phi_b$ as maps from $K^2$ to $K^2$ because
$N^x_\beta=N^y_\beta=\id$, but then the first equation implies that the matrices $B_x, B_y$ are equal, which cannot
happen if $x\neq y$.  
\end{example}

In the next example, we combine the ideas of the last three examples to construct simple representations in a more
flexible way. 

\begin{example}
\label{eg:46simple}
Let $(W,S), G$,  and $Q$ be as in Example \ref{eg:twoHeavyEdges}.  We construct a representation $M\in \repkqr$ in
the case that $m_1\ge 4, m_2\ge 6$ and prove that $M$ contains a simple subrepresentation $N$. 

Consider the irreducible representation $\rho: S_q \ra \mathrm{GL}(V)$ and the elements $\sigma,\tau\in S_q$ from
Example \ref{eg:Miller}. Using minimal polynomials as we did in the proof of Lemma \ref{lemm:decompose}, we may
deduce from the fact $\rho(\sigma)^2=\id_V$ that $\rho(\sigma)$ is a diagonalizable map whose eigenvalues are from
the set $\{-1,1\}$, so with respect to $ \rho(\sigma)$ we have an eigenspace decomposition $V=E_{1}\oplus E_{2}$
where $E_1, E_{2}$ are the eigenspaces for the eigenvalues $1$ and $-1$, respectively.  Similarly, since
$\rho(\tau)^3=\id_V$, the map $\rho(\tau)$ is diagonalizable and its eigenvalues lie in the set
$\{\omega_1,\omega_2,\omega_3\}$ containing the three third roots of unity, therefore we have an eigenspace
decomposition $V=F_{1}\oplus F_2\oplus F_3$ of $V$ where each $F_i$ is the eigenspace for the eigenvalue
$\omega_i$. Note that 0 and $V$ are the only subspaces of $V$ compatible with both these decompositions, because
they are the only subspaces invariant under both $\rho(\sigma)$ and $\rho(\tau)$ by Example \ref{eg:Miller}. 

To define $M$, first set $M_a=M_b=M_c=V$. Next, assign the linear maps $M_\alpha, M_\beta$ based on the value of
$m_1$: if $m_1>4$, then $\tilde f_{m_1-1}$ contains at least two nonzero roots $\lambda_1,\lambda_2$, and we set
\[
  M_\alpha=\id_{E_1}\oplus \id_{E_2}, \quad M_\beta=(\lambda_1\cdot
    \id_{E_2})\oplus (\lambda_2\cdot \id_{E_2})
\]
where the notation means, for example,  that $M_\beta$ restricts to $\lambda_1\cdot \id$ on $E_1$ and to
$\lambda_2\cdot \id$ on $E_2$; if $m_1=4$, we set
\[
    M_\alpha=\id_{E_1}\oplus 0_{E_2}, \quad M_\beta=(\lambda_1\cdot
    \id_{E_2})\oplus 0_{E_2}
\]
where $\lambda_1$ is the unique nonzero root of $\tilde f_{3}$.  Similarly, if $m_2>6$ then $\tilde f_{m_2-1}$ has
at least three nonzero roots $\mu_1,\mu_2,\mu_3$ and we set
\[
    M_\gamma=\id_{F_1} \oplus  \id_{F_2}\oplus
    \id_{F_3}, \quad M_\delta=(\mu_1\cdot \id_{F_1}) \oplus (\mu_2\cdot
    \id_{F_2})\oplus (\mu_3\cdot \id_{F_3}),
\]
while if $m_2=6$ then we set
\[
    M_\gamma=\id_{F_1} \oplus  \id_{F_2}\oplus
    0_{F_3}, \quad M_\delta=(\mu_1\cdot \id_{F_1}) \oplus (\mu_2\cdot
    \id_{F_2})\oplus 0_{F_3}.
\]
where $\mu_1,\mu_2$ are the two nonzero roots of $\tilde f_{m_2-1}$.
By Parts (a) and (b) of Proposition \ref{prop:edgeLoopRep}, the assignments
define a representation in $\repqrf$. Moreover, the eigenspace decompositions
of $M_b$ with respect to the maps $\phi_1=M_\alpha M_\beta$ and
$\phi_2=M_\delta M_\gamma$ coincide with the eigenspace decompositions of $V$
with respect to the maps $\rho(\sigma)$ and $\rho(\tau)$, respectively,
therefore a subrepresentation of $M$ must assign the space $0$ or $V$ to the
vertex $b$. It follows that $M$ contains a simple representation $N$ with
$N_b=V$ and 
\[
N_a=
\begin{cases}
    M_a=E_1\oplus E_2 & \text{if $m_1>4$};\\
    E_{1}& \text{if $m_1=4$},\\
\end{cases}\;
N_c=
\begin{cases}
    M_c=F_1\oplus F_2\oplus F_3 & \text{if $m_1>6$};\\
    F_{1}\oplus F_2 & \text{if $m_1=6$}.\\
\end{cases}\quad
\]

\end{example}

\begin{remark} 
Given a vector space $V$ and operators $\phi_1,\cdots, \phi_k$ on $V$, the study of subspaces of $V$ that are
simultaneously compatible with the eigenspace decompositions of all the operators is closely related to enumerative
geometry and Schubert calculus. For example, if $K=\C$, $V= K^4$, $n=2$ and the operators $\phi_1 , \phi_2$ yield
eigenspace decompositions $M_b = E_1 \oplus E_2$ and $M_b = F_1\oplus F_2$  where $E_1, E_2, F_1, F_2$ all have
dimension $2$, then ``generically'' there exists a subspace $W\se V$ with dimension 2 that is simultaneously
compatible with both $\phi_1$ and $\phi_2$, because a classical result in Schubert calculus asserts that
generically, given four lines in the projective 3-space $\C\mathbb{P}^3$, there are two lines that intersect all
these four lines.  On the other hand, when the dimensions of the eigenspaces of $V$ with respect to $\phi_1,\cdots,
\phi_k$ are known, we can often show that no proper, nontrivial subspace of $V$ can be simultaneously compatible
with the corresponding eigenspace decompositions by certain codimension computations involving products of Schubert
classes. For instance, to construct a simple representation in $\rep_K(\bar Q,\bar \cali_f)$ in Example
\ref{eg:Miller}, it is possible to specify for every positive integer $n$ a representation $M\in \rep_K(\bar Q,\bar
\cali_f)$ in such a way that $\dim (M_b)=6n$, the maps $M_\alpha$ and $M_\beta$ are isomorphisms,
$\phi_1:=M_{\ve{1}}$ is diagonalizable with two eigenspaces $E_1, E_2$ of dimension $3n$, and $\phi_2:=M_{\ve{2}}$
is diagonalizable with three eigenspaces $F_1, F_2,F_3$ of dimension $2n$. A codimension computation using Schubert
calculus guarantees that generically $M_b$ has no subspace compatible with both $M_{\ve{1}}$ and $M_{\ve{2}}$ other
than 0 and $M_b$ itself, so $M$ is simple (generically).  This yields an alternative proof of the existence of
certain simple representations in $\rep_K(\bar Q,\bar\calr)$.  In the interest of space, however, we omit details
of the codimension computation and the necessary background on Schubert calculus. In particular, we will not make
precise what the word ``generically'' means in this paragraph.  
\end{remark} 

We end this subsection by proving a proposition to be used in \autoref{sec:repProofs} under the following setting:
Let $Q$ be the double quiver of a Coxeter diagram $G$ and let $\{f_n\}$ be as defined in Equation
\ref{eq:PowerPoly}.  Suppose that $G$ contains a vertex $v$ which is adjacent to a unique vertex in $G$. Let $u$ be
that unique vertex, let $m=m(u,v)$, and let $\hat G$ be the graph obtained from $G$ by removing the vertex $v$ and
the edge $v-u$.  Let $\hat Q$ be the double quiver of $\hat G$, then let $\hat \cali_f$ be the evaluation ideal of
$\{f_n\}$ in $K \hat Q$. The proposition allows us to ``enlarge'' certain representations in $\repqrf$ to a simple
representation in $\rep_K(\hat Q,\hat\cali_f)$: 

\begin{prop}
    \label{prop:mixingSimples}
    Let $Q,\calr, \hat Q,\hat \calr$ and $u, v$ be as described above.  Suppose that  $\{S(1), S(2), \dots, S(k)\}$
    is a nonempty set of pairwise non-isomorphic simple representations in $\rep_K(\hat Q,\hat \calr)$ and let
    $S=\oplus_{i=1}S(i)$. If $m>3$, then there is a simple representation $M$ in $\repqrf$ such that $M_{a}=S_a$
    for all $a\in \hat Q_0$.  \end{prop}

We prove the proposition via subspace analysis at the vertex $u$, by using the following lemma:

\begin{lemma} \label{lemm:DiagonalSimples} 
    Let $\Lambda$ be an arbitrary ring. Let $\{S(1), \ldots, S(k)\}$ be a set of pairwise non-isomorphic simple
right $\Lambda$-modules, let $S=\oplus_{i=1}^k S(i)$, and let $x = \sum_{1 \le i \le k}x_i\in S$ where $x_i\in
S(i)$ for each $i$.  If $x_i\neq 0$ for all $1\le i\le k$, then the submodule generated by $x$ equals $S$.
\end{lemma}

\begin{proof}
    We use induction on $k$. The claim clearly holds when $k=1$. For $k>1$, let $I_i=\mathrm{Ann}(x_i):=\{r\in
    \Lambda: x_ir=0\}$ for each $i$. Then $I_1,\dots, I_k$ are distinct maximal right ideals of $\Lambda$ since
    $S(1), \dots, S(k)$ are pairwise non-isomorphic simple right $\Lambda$-modules. Let $r\in I_1\setminus I_2$ and
    let $y_i=x_ir$ for all $i$. Then $y_1=0$ and $y_2\neq 0$. Let $J=\{1\le j\le k:y_j\neq 0\}$ and let
    $y=\sum_{j\in J}y_j$.  Applying the inductive hypothesis on the module $S':=\oplus_{j\in J}S(j)$, we conclude
    that $S'\se y\Lambda$. Furthermore, since $y=\sum_{j\in J}y_j=\sum_{1\le j\le k}y_j=\sum_{1\le j\le k}x_jr=xr$,
    we have $y\Lambda \se x\Lambda$. It follows that $S'\se x\Lambda$. In particular, we have $\sum_{j\in J}x_j\in
    x\Lambda$ and hence $\sum_{j\notin J}x_j\in x\Lambda$.  By the inductive hypothesis, the element $\sum_{j\notin
    J}x_j$ generates the module $S'':=\oplus_{j\notin J}S(j)$, so $S''\se x\Lambda$. It follows that $S=S'\oplus
    S''\se x\Lambda$.   
\end{proof}

\begin{proof}[Proof of Proposition \ref{prop:mixingSimples}]
Denote the arrows $u\ra v$ and $v\ra u$ in $Q$ by $\alpha$ and $\beta$, respectively.  By
\autoref{sec:repAnalysis}, to construct a representation $M\in \repqrf$ it suffices to extend $S$ by a local
representation $M_{\{\alpha,\beta\}}=\{M_u, M_v, M_\alpha, M_\beta\}$ for the set $\{\alpha,\beta\}$ such that
$M_u=S_u$.  We do so by setting $M_u=S_u, M_v=K$ and setting $M_\alpha, M_\beta$ as follows: 

\begin{enumerate}
\item 
    If $m=4$, let $d=\dim(S_u)$ and let $B_S=B_1\sqcup B_2\sqcup\cdots B_k$ be a basis of $S_u$ where $B_i$ is a
    basis of $S(i)_u$ for all $1\le i\le k$.  Consider the $1\times d$ and $d\times 1$ matrices
\[
    X= \bm{1& 1& \dots& 1}, \quad Y=\bm{d& -1 & \dots & -1}^T.
\]
    Define $M_\alpha: S_u\ra K$ and $M_\beta: K\ra S_u$ to be the maps whose matrices with respect to $B_S$ and
    $\{1\}$ (considered a basis of $K$) are given by $X$ and  $Y$, respectively. Since $\papb=\id$, the assignments
    for $M_\alpha$ and $M_\beta$ satisfy the relations $r_f(\alpha)=\alpha\beta\alpha-\alpha$ and
    $r_f(\beta)=\beta\alpha\beta-\beta$, so $M$ is indeed a representation in $\repqrf$.  
\item 
    If $m\ge 5$, then let $d$ and $X, Y$ be as before, set $M_\alpha=\id$, and define $M_\beta$ to be the whose
    matrix with respect to $B_S$ is $I_d-2YX$. These assignments ensure that $(M_\alpha
    M_\beta)^2=(\pbpa)^2=M_\beta^2=\id$, so they define a representation in $\repqrf$ by Lemma
    \ref{lemm:trivialLocalReps}.(b).  
\end{enumerate} 
The representation $M$ satisfies the condition that $M_a=S_a$ for all $a\in \hat Q_0$ by definition, so it remains
to show that $M$ is simple.  To this end, let $N$ be a subrepresentation of $M$ in $\repcqr$, and let
$x=(x_i)_{1\le i\le k}\in $ be any nonzero vector in $S_u=\oplus_{i=1}^k S(i)_u$. Then the set $J:=\{1\le j\le n:
x_j\neq 0\}$ is nonempty. Invoke the equivalence between the categories $\rep_K(\hat Q,\hat\cali_f)$ and mod-$K\hat
Q/\hat\cali_f$ to identify $S(1), \cdots, S(k)$ and $S$ as modules of the algebra $K \hat Q/\hat\cali_f$. Then
Lemma \ref{lemm:DiagonalSimples} implies that the submodule of generated by $x$ in $S$ must contain the direct sum
$\oplus_{j\in J} S(j)$. Invoking the same equivalence again, we conclude that $N_u$ contains a basis vector $e$
from the basis $B_j$ of $S(j)$ for some $j\in J$. By direct computation, the element $x'=\pbpa(e)\in M_u$ must have
nonzero entries at all coordinates, therefore $x'$ generates all of $M_u$ in $N$, i.e., we have $N_u=M_u$, by Lemma
\ref{lemm:DiagonalSimples}. Since $S(1),\cdots, S(k)$ are simple and $M_\alpha$ is surjective, it follows that $M$
is simple.  
\end{proof}

\section{\texorpdfstring{Results on mod-$A$}{}}
\label{sec:mod-AResults}
We maintain the setting of Section \ref{sec:A-modTools} and study the category mod-$A$ in this section. Recall that
$A=A_K=K\otimes_\Z J_C$ where $K$ is an algebraically closed field with characteristic zero and $J_C$ is an
irreducible Coxeter system $(W,S)$ with Coxeter diagram $G$ and subregular cell $C$. 

\subsection{Results}
\label{sec:repResults}
Our first main result characterizes in terms of the Coxeter diagram $G$ when mod-$A$ is semisimple, as well as when
mod-$A$ \emph{has finitely many simples}, i.e., when it contains finitely many simple modules up to isomorphism.

\begin{thm} \label{thm:repThm1} 
The following conditions are equivalent: 
\begin{enumerate} 
\item The graph $G$ is a tree, has no edge with infinite weight, and has at most one heavy edge.
\item The category mod-$A$ is semisimple.
\item The category mod-$A$ has finitely many simples. 
\end{enumerate} 
\end{thm}

Our second main result gives a similar characterization of when mod-$A$ \emph{has bounded simples} in the sense
that there exists an upper bound on the dimensions of the simple modules of mod-$A$. Since the simple modules of
mod-$A$ are certainly bounded if there are only finitely many of them, we start with the assumption that the
conditions of Theorem \ref{thm:repThm1} do not hold:

\begin{thm} \label{thm:repThm2} 
Suppose that $G$ contains a cycle or has an edge with infinite weight or has at
least two heavy edges.  Then the dimensions of the simple modules of mod-$A$
are bounded above if and only if one of the following mutually exclusive
conditions holds:
\begin{enumerate} 
    \item $G$ contains a unique cycle, and all edges in $G$ are simple.
    \item $G$ is a tree and contains exactly two heavy edges; moreover, each
    of those two edges has weight 4 or weight 5.
\end{enumerate} 
\end{thm} 
\noindent Here and henceforth, a \emph{cycle} in a graph means a tuple $C=(v_1,v_2,\dots, v_{n})$ of $n\ge 3$
vertices in the graph such that $v_1-v_2, v_2-v_3, \dots,v_{n-1}-v_n, v_n-v_1$ are all edges in $G$.

Note that Theorems \ref{thm:repThm1} and \ref{thm:repThm2} has the following consequence:

\begin{remark}
\label{rmk:freeProductRep}
Recall from Example \ref{eg:freeProduct} that for every tuple $\mathbf{k}=(k_1,\dots,k_n)$ of positive integers,
the algebra $A_{\mathbf{k}}=\ip{x_j:1\le j\le n, x_j^{k_j}=1}$ is isomorphic to the group algebra of the free
product $C_{\mathbf{k}}=C_{k_1}*\dots*C_{k_n}$ where each $C_{k_i}$ is the cyclic group of order $k_i$. Note that
if $k_j=1$ for some $j$ then $C_{k_j}$ is the trivial group and makes trivial contribution to the free product in
the sense that $C_\mathbf{k}\cong C_{k_1}*\dots*C_{k_{j-1}}*C_{k_{j+1}}*\dots*C_{k_n}$, so we assume from now on
that $k_j>1$ for all $1\le j\le n$. In the example, we showed that $A_\mathbf{k}$ is Morita equivalent to the
algebra $A$ associated with a Coxeter system whose Coxeter diagram is a tree and has a heavy edge of weight
$m_j:=2k_j+1$ for each $1\le j\le n$; in particular, under the assumption that $k_j>1$ for all $j$, the weight
$m_j$ is an odd number greater than 3, and we have $m_j=5$ if and only if $k_j=2$. Theorems \ref{thm:repThm1} and
\ref{thm:repThm2} now imply the following result:
    
\begin{prop}
\label{prop:freeProductRep}
Suppose $\mathbf{k}=(k_1,\dots, k_n)$ where $k_i\in \Z_{> 1}$ for all $1\le i\le n$, and
let mod-$A_{\mathbf{k}}$ be the category of finite dimensional right modules of
$A_\mathbf{k}$.         

\begin{enumerate}
\item The category mod-$A_\mathbf{k}$ is semisimple if and only it contains
    finitely many isomorphism classes of simple modules. Moreover, these two
    conditions are satisfied if and only if $n=1$, i.e., if and only if
    $C_{\mathbf{k}}$ has a single factor and is a finite cyclic group.
\item Suppose the category mod-$A$ has infinitely many pairwise non-isomorphic
    simple modules. Then the simple modules of mod-$A_{\mathbf{k}}$ have
    bounded dimensions if and only if $\mathbf{k}=(2,2)$, i.e., if and only if
    $C_\mathbf{k}$ is isomorphic to the free product $C_{2}* C_2$.
\end{enumerate}
\end{prop}
\end{remark}

Let us explain our strategy for proving Theorems \ref{thm:repThm1} and \ref{thm:repThm2}.  For Theorem
\ref{thm:repThm1}, first recall that (a) implies (b) by Proposition \ref{prop:contractToDihedral}, thanks to simple
graph contractions. It is well-known that Condition (a) is equivalent to the condition that the cell $C$ is finite
(see \cite{LusztigSubregular}), therefore (a) also implies (c), since $\dim(A)=\abs{C}$ and a finite dimensional
semisimple algebra has finitely many simple modules. To prove the theorem, it remains to prove that (b) implies (a)
and that (c) implies (a). We will prove the contrapositives of these two implications:

\begin{prop} \label{prop:bImpliesA} 
    If $G$ contains a cycle or has an edge with infinite weight or has at least
    two heavy edges, then mod-$A$ contains a module which is not semisimple.
\end{prop}

\begin{prop} \label{prop:cImpliesA} 
    If $G$ contains a cycle or has an edge with infinite weight or has at least
    two heavy edges, then mod-$A$ contains an infinite set of pairwise non-isomorphic simple modules.
\end{prop}

To prove Theorem \ref{thm:repThm2}, we will first prove the ``if'' implication:

\begin{prop}
    \label{prop:Thm2if}
    If $G$ satisfies either Condition \textup{(a)} or Condition \textup{(b)} in Theorem \ref{thm:repThm2}, then
    mod-$A$ has bounded simples.
\end{prop}
\noindent
To prove the ``only if'' implication of Theorem \ref{thm:repThm2}, we again prove its contrapositive. Doing so
requires describing the situations where Conditions (a) and (b) in the theorem fail under the assumption that $G$
is not a tree, has an edge with infinite weight, or has at least two heavy edges. A moment's thought reveals that
we may formulate the contrapositive as follows:

\begin{prop}
    \label{prop:Contrapositive}
    The dimensions of the simple modules in mod-$A$ have no upper bound if $G$ satisfies one of the following conditions:
\begin{enumerate}
    \item $G$ contains a unique cycle as well as a heavy edge;
    \item $G$ contains at least two cycles;
    \item $G$ is a tree and has exactly two heavy edges; moreover, one of these heavy edges has weight at least 6;
    \item $G$ is a tree and has at least three heavy edges.
\end{enumerate}
\end{prop}

We have reduced the proofs of Theorems \ref{thm:repThm1} and \ref{thm:repThm2} to the proofs of Propositions
\ref{prop:bImpliesA}, \ref{prop:cImpliesA}, \ref{prop:Thm2if}, and \ref{prop:Contrapositive}, which we will give in
\autoref{sec:repProofs}.  Note that all these theorems and propositions are stated without any reference to any
quivers. On the other hand, the proofs in \autoref{sec:repProofs} will all use quiver representations and rely
heavily on the techniques and examples of Section \ref{sec:A-modTools}. It is also worth noting that part of the
proof of Proposition \ref{prop:Contrapositive} will use Proposition \ref{prop:cImpliesA}: to obtain desired simple
representations for the former proposition, we will sometimes form direct sums of simple representations promised
by latter proposition and then ``enlarge'' the direct sums using Proposition \ref{prop:mixingSimples}.

\subsection{Proofs}
\label{sec:repProofs}

Let $Q$ be the double quiver of $G$, let $\{f_n\}$ be the uniform  family of polynomials defined by Equation
\eqref{eq:PowerPoly}, and let $\cali_f$ be the evaluation ideal of $\{f_n\}$ in $K Q$. We prove Propositions
\ref{prop:bImpliesA}-\ref{prop:Contrapositive} by proving the same conclusions for the equivalent category
$\rep_K (Q,\cali_f)$ in this subsection. Of the four propositions, we first prove Proposition \ref{prop:Thm2if}.
The other three propositions all state that mod-$A$ contains modules with certain properties, so we will prove them
by explicit construction of suitable representations in $\repqrf$. Since the properties of mod-$A$ that we are
interested in, namely, being semisimple, having finitely many simples, and having bounded simples, are all
preserved under Morita equivalences, when dealing with $\repqrf$ we may assume that certain contractions have been
performed on $Q$ and thus effectively deal with a category of the form $\rep_K(\bar Q,\bar \cali_f)$ from
\autoref{sec:repAnalysis}. For instance, by Example \ref{eg:treeContraction} and Remark \ref{rmk:operators}, if the
Coxeter diagram $G$ is the tree from \autoref{fig:treeContraction} then we may study mod-$A$ via the category
$\brepqrf$ for the generalized double quiver $\bar Q$ from \autoref{fig:57}. 

As final preparation for our proofs, we fix some notation and terminology for cycles in the Coxeter diagram $G$.
Given a cycle $C=(v_1,v_2,\dots,v_n)$ with $n\ge 3$ vertices in $G$, we say $C$ has \emph{length} $n$, set
$v_{n+1}:=v_1$, define $V_C=\{v_1,\dots, v_n\}$, and define
\[
    E_C=\{v_i-v_{i+1}:1\le i\le n\}.
\]
We call the edges in $E_C$ the \emph{sides} of $C$ and define a \emph{diagonal in $C$} to be a edge in $G$ that
connects two vertices in $C$ but does not lie in $E_C$.  We say $C$ is a \emph{minimal} cycle in $G$ if $C$ has no
diagonals (a diagonal in $C$ would break $C$ into two shorter cycles). For each $1\le i\le n$, we let
$m_{C,i}=m(v_i,v_{i+1})$ and denote the arrows $v_i\ra v_{i+1}$ and $v_{i+1}\ra v_i$ in the double $Q$ of $G$ by
$\alpha({C,i})$ and $\beta(C,i)$, respectively. For a representation $M\in \rep_K Q$, we let 
\[
    M_C:=M_{\alpha(C,n)}\circ \dots\circ M_{\alpha(C,2)}\circ M_{\alpha(C,1)},
\]
\[
    \bar
    M_C:= M_{\beta(C,1)}\circ  M_{\beta(C,2)}\circ \dots\circ 
    M_{\beta(C,n)}.
\]
If it is clear what $C$ is from context, then we omit $C$ and write $m_i, \alpha_i, \beta_i, M_i$ and $\bar M_i$
for $m_{C,i}, \alpha(C,i), \beta(C,i), M_{\alpha(C,i)}$ and $M_{\beta(C,i)}$, respectively.

\begin{proof}[Proof of Proposition \ref{prop:Thm2if}]
It suffices to show that mod-$A$ or $\rep_K(Q,\cali_f)$ has bounded simples if $G$ satisfies one of the conditions
in Theorem \ref{thm:repThm2}. To do so, first assume that $G$ satisfies Condition (a), i.e., that $G$ contains a
unique cycle and has only simple edges. Let $C$ be the unique cycle. Then $C$ is necessarily minimal.  By applying
simple graph contractions if necessary, we may assume that $G$ is exactly $C$ in the sense that $V_C$ contains all
the vertices of $G$ and $E_C$ contains all the edge of $G$. But then the algebra $A$ is Morita equivalent to the
Laurent polynomial ring $\cala= K[t,t\inverse]$ by Example \ref{eg:LaurentPolynomials}. As $\cala$ is commutative,
every simple module of $\cala$ has dimension 1, so mod-$A$ has bounded simples.
 
Next, suppose that $G$ satisfies Condition (b), i.e., that $G$ is a tree with exactly two heavy edges and the edges
have weights $m_1,m_2\in \{4,5\}$.  Applying simple graph contractions on $G$ if necessary, we may assume that $G$
and $Q$ are as pictured in \autoref{fig:twoHeavyEdges}. Let $M\in \rep_K(Q,\cali_f)$. Let
$\alpha,\beta,\gamma,\delta$ be as in \autoref{fig:twoHeavyEdges} and let $\phi_1=M_\alpha M_\beta, \phi_2=M_\delta
M_\gamma$. Let $i\in \{1,2\}$. Then $\phi_i$ are diagonalizable by Proposition \ref{prop:edgeLoopRep}.(c). Since
$f_3=x^3-x$ and $f_4=x^4-1$, it follows that if $m_i =4$, then $\phi_i^2 =\phi_i$ and the eigenvalues of $\phi_i$
lie in the set $\{0,1\}$; if $m_i =5$, then $\phi_i ^2 =\id$ and the eigenvalues of $\phi_i$ lie in the set
$\{-1,1\}$. Now let $n=\dim (M_b)$ and suppose the eigenspace decomposition of $M_b$ relative to $\phi_1$ and
$\phi_2$ are $M_b=E_1\oplus E_{2}$ and $M_b=F_1\oplus F_{2}$, respectively, with $E_1, F_1$ being the eigenspaces
for the eigenvalue 1 and $E_2,F_2$ being the eigenspaces for the eigenvalue 0 or -1. We claim that $M$ contains a
submodule of dimension at most 6. The claim would imply that $\rep_K(Q,\cali_f)$ has bounded simples, as desired.

To prove the claim, first note that if $E_i\cap F_j\neq 0$ for some $i,j\in \{1,2\}$, then any nonzero vector $v\in
E_i\cap F_j\se M_b$ must generate a submodule $N$ of $M$ where $N_a, N_b, N_c$ are the spans of $M_\beta(v), v,
M_\gamma(v)$, respectively. The module $N$ has dimension at most 3, proving the claim.  Otherwise, we must have
$n=2k$ for some positive integer $k$ and $\dim(E_1)=\dim(E_2)=\dim(F_1)=\dim(F_2)=k$. In this case, we may choose a
suitable basis $B$ for $M_b$ so that the matrices of $\phi_1$ relative to $B$ is the block diagonal matrix 
\[ 
    [\phi_1]_B= \begin{bmatrix} I_k & 0 \\ 0 & D\\
\end{bmatrix} 
\] 
where $I_k$ is the $k\times k$ identity matrix, $D=0$ if $m_1=4$, and $D=-I_k$ if $m_1 = 5$.  Further, by choosing
a basis $B'$ of $M_b$ for which the change-of-basis matrix $P$ from $B'$ to $B$ is of the block diagonal form 
\[ 
    P= \begin{bmatrix} P_1 & 0 \\ 0 & P_2 \end{bmatrix} 
\] 
with suitable $k\times k$ matrices $P_1, P_2$, we can ensure that 
\[ 
    [\phi_1]_{B'}=P\inverse [\phi_1]_B P = [\phi_1]_B, \quad [\phi_2]_{B'} = P\inverse [\phi_2]_B P= \begin{bmatrix} A_{11} & A_{12}\\ A_{21} & A_{22} \end{bmatrix} 
\] 
where each $A_{ij}$ is $k\times k$ and $A_{11}, A_{44}$ are in Jordan canonical form. Suppose $B'=\{v_1,\cdots,
v_n\}$. Then $\phi_1(v_1)=v_1$ and $\phi_2(v_1)=\lambda v_1 +v$ where $\lambda$ is the top left entry in $A_{11}$
and $v$ in the span $\ip{v_{k+1}, v_{k+2},\cdots, v_{n}}$. It follows that $\ip{v_1,\phi_2(v_1)}=\ip{v_1, v}$.
Moreover, since either $\phi_2^2=\phi_2$ or $\phi_2^2 =1$, we have 
\[ 
    \phi_2(v)=\phi_2(\phi_2(v_1)-\lambda v_1)=\phi_2^2(v_1)-\lambda \phi_2(v_1)\in \ip{v_1,\phi_2(v_1)}= \ip{v_1,v}.
\] 
It follows that the space $V:=\ip{v_1,v}$ is invariant under both $\phi_1$ and $\phi_2$, so it generates a
subrepresentation $N$ of $M$ such that $N_b=V$ and $\dim(N)\le 6$. This completes the proof.  
\end{proof}

\begin{proof}[Proof of Proposition \ref{prop:bImpliesA}]
We need to construct a non-semisimple representation $M\in \rep_K (Q,\cali_f)$ when $G$ contains a cycle, an edge
with infinite weight, or at least two heavy edges. We first deal with the case that $G$ contains an edge with
infinite weight or a cycle.  Let $\{a,b\}$ be an edge with infinite weight in $G$ if such an edge exists;
otherwise, let $a,b$ be the vertices $v_1,v_2$ from a cycle $C=(v_1,v_2,\dots,v_n)$ in $G$, respectively. Denote
the arrow $a\ra b$ by $\alpha$ and the arrow $b\ra a$ by $\beta$.  To construct $M$, first let $M_s=K^2$ for all
$s\in Q_0$. Let $m=m(a,b)$, let $\lambda_m$ be  a root of the polynomial $\tilde f_{m-1}$ if $m<\infty$, let $x$ be
an arbitrary nonzero scalar in $K$, and let 
\[
J_x=\bm{x & 1\\ 0 &
    x}, \quad 
L = 
\begin{cases}
    I_2 & \text{if $m=\infty$};\\
    \lambda_m \cdot J\inverse & \text{if $m<\infty$}.
\end{cases}
\]
Let $M_\alpha, M_\beta$ be the maps given by $J_x$ and $L$, respectively, then let $M_\gamma=\id$ for all arrows
$\gamma\in Q_1\setminus\{\alpha,\beta\}$.  The assignment  $M:=(M_s,M_\gamma)_{s\in Q_0, \gamma\in Q_1}$ defines a
representation in $\repqrf$ by Proposition \ref{prop:edgeLoopRep} and Corollary \ref{lemm:trivialLocalReps}. It is
clear that $M^x$ has a subrepresentation $N$ with $N_s=\ip{e_1}$, the span of the first standard basis vector, for
all $s\in Q_0$. On the other hand, if we set $\phi=\pbpa$ in the case $m=\infty$ and set $\phi=M_C$ otherwise, then
$\phi$ is an endomorphism of $M_a$ given by the matrix $J_x$ which is in Jordan canonical form and has a single
$2\times 2$ Jordan block, therefore the subspace $N_a$ of $M_a$ cannot have a complement in $M_a$ that is invariant
under $\phi$. It follows that $N$ has no complement in $M$ as a subrepresentation, so $M$ is not semisimple.

It remains to consider the case where $G$ has no cycles or edges of infinite weight but has at least two heavy
edges. By applying graph contractions, we may ensure that $G$ contains a subgraph of the form shown in
\autoref{fig:twoHeavyEdges} and $Q$ contains a subquiver of the form shown in the same Figure. We may define a
representation $M\in \repqrf$ by setting $M_a, M_b, M_c, M_\alpha, M_\beta, M_\gamma, M_\delta$ as in Example
\ref{eg:twoHeavyEdges} and setting $M_s=K^2$ and $M_\zeta=\id$ for all vertices $s\in Q_0\setminus\{a,b,c\}$ and
all arrows $\zeta\in Q_1\setminus\{\alpha,\beta,\gamma,\delta\}$, because doing so amounts to assembling a
consistent collection of local representations in the sense of \autoref{sec:repAnalysis}. Moreover, it is clear
that $M$ has a subrepresentation $L$ with $L_s=\ip{e_1}$ for all $s\in Q_0$.  By the subspace analysis in Example
\ref{eg:twoHeavyEdges}, any subrepresentation $N$ of $M$ must have $N_b=\ip{e_1}=L_b$, therefore the
subrepresentation $N$ has no complement in $M$. It follows that $M$ is not semisimple, and we are done.
\end{proof} 

\begin{proof}[Proof of Proposition \ref{prop:cImpliesA}]
We need to find infinitely many simple representations in $\repqrf$ when $G$ contains a cycle, an edge with
infinite weight, or at least two heavy edges. We keep the notation from the previous proof and start with the case
that $G$ contains an edge with infinite weight or a cycle. In this case, let $M$ and $N$ be as in the previous
proof. Denote $N$ by $N^x$ to reflect the fact that $N$ depends on the value of the scalar $x$ because the matrix
$J_x$ does. Then $N^x$ is clearly simple, so to show that $\repqrf$ has infinitely many simples it suffices to
verify that $N^x\not\cong N^y$ whenever $x\neq y$.  If $m=\infty$, then we can do so by using basic linear algebra
to show that there do not exist linear isomorphisms such that $\phi_b N^x_\alpha = N^y_\alpha \phi_a$ and $\phi_a
N^x_\beta= N^y_\beta \phi_b$ simultaneously.  If $m\neq \infty$, then by definition we have $a=v_1,b=v_2$ for
vertices $v_1,v_2$ in a cycle $C=(v_1,v_2,\dots,v_n)$, and we can show that $N^x\not\cong N^y$ if $x\neq y$ by
showing that no linear map $\phi_1: M_a\ra M_a$ can satisfy $N^x_C\phi_1=\phi_1 N^y_C$ when $x\neq y$. We omit the
details.

It remains to deal with the case where $G$ has no cycle or edges of infinite weight but has at least two heavy
edges. As in the previous proof, we may assume that $G$ contains the Coxeter diagram in \autoref{fig:twoHeavyEdges}
as a subgraph and $Q$ contains the quiver there as a subgraph. To obtain an infinite family of simple
representations in $\repqrf$, we extend the simple representations of the form $N^x$ from Example
\ref{eg:twoHeavyEdgesSimple} as follows:
\begin{enumerate}
\item 
    If $N^x_c=K^2$, then we extend $N^x$ by setting $N^x_s=K^2$ and
    $N^x_\zeta=\id$ for all vertices $s\in Q_0\setminus\{a,b,c\}$ and all
    arrows $\zeta\in Q_1\setminus\{\alpha,\beta,\gamma,\delta\}$. This
    specifies a representation in $\repqrf$, and the extended representation
    $N^x$ is still simple since all the maps $N^x_\zeta$ are isomorphisms. 
\item 
    If $N^x_c=K$, then note that since $G$ contains no cycle, removing the
    edge $b-c$ from $G$ must result in a graph with two connected
    components, one containing $b$ and the other containing $c$. Let $V_b, V_c$
    be the sets of vertices in the first and second component, respectively.
    Then we may extend $N^x$ by setting $N^x_s=K^2$ for all $s\in
    V_b\setminus\{a,b\}$, setting $N^x_s=K$ for all $s\in V_c\setminus\{c\}$,
    and setting $N^x_\zeta=\id$ for all  $\zeta\in
    Q_1\setminus\{\alpha,\beta,\gamma,\delta\}$. It is easy to see that the
    extension gives a simple representation in $\repqrf$ as in Case (a).
\end{enumerate}
To finish the proof, it suffices to show that $N^x\not\cong N^y$ whenever $x\neq y$. This holds by the same
argument used at the end of Example \ref{eg:twoHeavyEdgesSimple}.  
\end{proof}

\begin{proof}[Proof of Proposition \ref{prop:Contrapositive}]
Let $n$ be an arbitrary positive integer larger than 7. We prove the proposition by constructing a simple
representation $M\in\repqrf$ with $\dim (M)>n$ when any of the Conditions (a)-(d) holds: 

(a) Suppose $G$ contains a unique cycle $C=(v_1,v_2,.., v_k)$ and a heavy edge.  We first consider the case where
the set $E_C$ contains a heavy edge.  As in the proof of Proposition \ref{prop:Thm2if}, up to simple graph
contractions we may assume $G$ is exactly $C$.  Without loss of generality, suppose that the edge $\{v_1,v_2\}$ is
heavy and let $m=m(v_1,v_2)$.  Depending on whether $m>4$ or $m=4$, we construct $M$ in one of two ways:

\begin{enumerate}
\item[(i)] If $m>4$, then let $q=n+1$, consider the symmetric group $G=S_{q}$, and consider the partition $(n,1)$
    of $q$. By the theory of Specht modules, the partition gives rise to an irreducible representation $\rho: G\ra
    \mathrm{GL}(V)$ of $G$ over $K$ where $\dim (V)=n$.  Recall from Example \ref{eg:Miller} that since $n>7$ there
    exist elements $\sigma,\tau\in G$ which generate $G$ with orders 2 and 3, respectively, and that  consequently
    we have $\rho(\sigma)^2=\rho(\tau)^3=\id_V$. To define $M$, let $M_v=V$ for all $v\in Q_0$, let 
    \[
    M_{\alpha_1}=\rho(\sigma),\quad M_{\alpha_2}=\rho(\tau)\rho(\sigma)^{-1},\quad M_{\beta_2}=M_{\alpha_2}^{-1},
    \]
    and let $M_{\beta}=\id$ for all $\gamma\in Q_1\setminus\{\alpha_1,\alpha_2,\beta_2\}$. This defines a
    representation by Lemma \ref{lemm:trivialLocalReps}, and clearly we have $\dim(M)>n$. To see that $M$ is
    simple, note that the operator $M_{\beta_1}M_{\alpha_1}:M_{1}\ra M_{1}$ equals $\rho(\sigma)$ and the operator
    $M_{C}: M_1\ra M_1$ equals $\rho(\tau)$. By subspace analysis at $v_1$ similar to the analysis at $y$ in
    Example \ref{eg:Miller}, any subrepresentation $N$ of $M$ must have either $N_{v_1}=0$ or $N_{v_1}=M_{v_1}$.
    Since $M_\gamma$ is an isomorphism for all $\gamma\in Q_1$, it follows that $M$ is simple. 

\item[(ii)] Now suppose $m=4$.  By Example \ref{eg:LaurentPolynomials}, up to contractions we may assume that $Q$
    is the generalized double quiver denoted $Q^{(3)}$ in \autoref{fig:Laurent}. In other words, after relabelling
    vertices and arrows we may assume that $Q$ is the quiver with a unique vertex $v$ along with two loops
    $\alpha,\beta:v\ra v$ and that
\[
    \cali_f=\{f_{m-1}(\alpha,\beta)=\alpha\beta\alpha-\alpha,
    f_{m-1}(\beta,\alpha)=\beta\alpha\beta-\beta\}.
\]
To construct $M$, let $M_v=K^n$, let $B=\{e_1,e_2,\dots, e_n\}$ be the standard basis of $K^n$, and let $M_\alpha,
M_\beta$ be the unique linear maps defined by 
    \[
        M_\alpha(e_i)=
        \begin{cases}
            e_{i+1} & \text{if $1\le i<n$};\\
            0 & \text{if $i=n$},
        \end{cases}
        \quad
        M_\beta(e_i)=
        \begin{cases}
            e_{i-1} & \text{if $1< i\le n$};\\
            0 & \text{if $i=1$}.
        \end{cases}
    \]
Intuitively, we may think of $M_\alpha$ as a ``raising'' operator on $K^n$ and $M_\beta$ as a ``lowering'' operator
in light of their effects on the standard basis elements.  It is easy to check that the above assignments define a
representation in $\repqrf$. Moreover, given any nonzero vector $v\in M_{v}$ we may use the maps $M_\alpha,
M_\beta$ to obtain any basis vector in $B$ up to a scalar, therefore $M$ must be simple.
\end{enumerate}

It remains to deal with the case that all edges in $C$ are simple but some edge in $G$ not in $V_C$ is heavy.
Applying simple graph contractions if necessary, we may assume that some vertex $u\in V_C$ is incident to a heavy
edge $\{u,v\}$ of weight $m=m(u,v)\ge 4$ for some $v\notin V_C$. Without loss of generality, we may also assume
that $u=v_1$.  View $C$ as a graph $G''$ with vertex set $V_C$ and edge set $E_C$, define $G'$ to be the subgraph
of $G$ obtained by adding the vertex $v$ and the edge $\{u,v\}$ to $G''$, and let $Q''$ and $Q'$ be the double
quivers of $G''$ and $G'$, respectively. Let $\cali''_f$ and $\cali'_f$ be the  evaluation ideals of $\{f_n\}$ in
$KQ''$ and $KQ'$.  Then by the proof of Proposition \ref{prop:cImpliesA}, the category $\rep_K(Q'', \cali''_f)$
contains $(n+1)$ pairwise non-isomorphic simple representations $S(1), S(2), \dots, S(n), S(n+1)$. Let
$S=\oplus_{i=1}^{n+1} S(i)$. Then Proposition \ref{prop:mixingSimples} implies that the category
$\rep_K(Q',\cali'_f)$ contains a simple representation $M$ with $\dim(M_u)=\dim (S_u)$.  Finally, we may extend $M$
to a representation $M\in \repcqr$, by using the same idea as in ``Case (b)'' in the proof of Proposition
\ref{prop:cImpliesA}: since $C$ is the only cycle in $G$, removing the edges in $E_C$ from $G$ results in a graph
with $k$ connected components with vertex sets $V_1, V_2,\dots, V_k$ such that $v_i\in V_i$ for all $1\le i\le k$;
we can then extend $M$ by setting $M_a=M_{v}$ for all $a\in V_1\setminus\{v_1,v\}$, setting $M_a=M_{v_i}$ for all
$a\in V_i\setminus\{v_i\}$ for each $2\le i\le k$, and setting $M_\gamma=\id$ for all $\gamma\in Q_1\setminus
Q'_1$.  It is clear that the extended representation is still simple and has dimension larger than $n$.

(b) Suppose $G$ contains two cycles. In light of Part (a), to construct the desired representation $M$ we may
assume that all edges in $G$ are simple. By considering minimal cycles and applying simple graph contractions if
necessary, we may assume that $G$ contains two minimal cycles which share a vertex, i.e., two minimal cycles of the
form $C=(v_1,v_2,\dots, v_k)$ and $C'=(w_1,w_2,\dots, w_l)$ where $v_1=w_1$.  Furthermore, while $C$ and $C'$ may
share an edge, we may write the tuples $C, C'$ in such a way that $v_2\neq w_2$. Denote the arrows $v_1\ra v_2,
v_2\ra v_1, w_1\ra w_2$ and $w_2\ra w_1$ by $\alpha_1,\beta_1,\alpha'_1$ and $\beta'_1$, respectively.  To
construct the desired representation $M\in \repqrf$, consider the representation $\rho: S_{q}\ra \mathrm{GL}(V)$
from Part (a).(i), let $\sigma,\tau$ be the same elements in $S_q$ as before, let $M_s=V$ for all $s\in Q_0$, let 
\[
        M_{\alpha_1}=\rho(\sigma),\quad
        M_{\beta_1}=\rho(\sigma)^{-1},\quad
        M_{\alpha'_1}=\rho(\tau),\quad M_{\beta'_1}=\rho(\tau)^{-1}, 
\]
and let $M_\gamma=\id$ for all arrows $\gamma\in Q_1\setminus\{\alpha_1,\beta_1,\alpha'_1,\beta'_1\}$. This defines
a representation $M\in \repcqr$ by Lemma \ref{lemm:trivialLocalReps}, and it is obvious that $\dim (M)>n$. The
endomorphisms $M_{\beta_1}M_{\alpha_1}$ and $M_{\beta'_1}M_{\alpha'_1}$ of $M_{v_1}$ equal $\rho(\sigma)$ and
$\rho(\tau)$, respectively, therefore $M$ is simple by the same arguments as before.

(c) Suppose $G$ is a tree and has exactly two heavy edges, one of which has weight at least 6. Using simple graph
contractions if necessary, we may assume that $G$ is of the form shown in \autoref{fig:twoHeavyEdges}, with $m_1\ge
4$ and $m_2\ge 6$. Let $q, G$ and $V$ be as in Part (a).(i). Then by Example \ref{eg:46simple}, there exists a
simple representation $M\in \repqrf$ such that $\dim(M)> \dim (V)=n$, as desired.

(d) Suppose that $G$ is a tree and has at least three heavy edges. Using simple graph contractions if necessary, we
may assume that $G$ contains a subgraph of one of the forms shown in \autoref{fig:threeHeavyEdges}.
\begin{figure}[h!] 
    \centering 
    \begin{tikzpicture}
        \node (0) {$G_1$:};
        \node[main node] (1) [right=0.5cm of 0] {$x$};
\node[main node] (2) [right = 1cm of 1] {$y$}; 
\node[main node] (3) [right = 1cm of 2] {$u$};
\node[main node] (4) [right = 1cm of 3] {$v$};

\node (00) [right =1cm of 4] {$G_2$:};
\node[main node] (a1) [below right = 0.4cm and 0.3cm of 00] {$x$};
\node[main node] (a2) [right = 1cm of a1] {$u$}; 
\node[main node] (a3) [right = 1cm of a2] {$y$};
\node[main node] (a4) [above = 1.2cm of a2] {$v$};

\path[draw] 
(1) edge node[above] {\small{$m_1$}} (2) 
(2) edge node[above] {\small{$m_2$}} (3)
(3) edge node[above] {\small{$m_3$}} (4);

\path[draw] 
(a1) edge node[above] {\small{$m_1$}} (a2) 
(a2) edge node[above] {\small{$m_2$}} (a3)
(a2) edge node[left] {\small{$m_3$}} (a4);

\end{tikzpicture} 
\caption{}
\label{fig:threeHeavyEdges}
\end{figure}
In both cases, let $G''$ be the subgraph of $G$ induced by the vertices $x,y,u$, let $G'$ be the subgraph of $G$
obtained by adding the vertex $v$ and the edge $\{u,v\}$ to $G''$, then define $Q'', \cali''_f, Q',\calr'$ as we
did in Part (a). We may produce a representation $M\in \repqrf$ in the same fashion as in Part (a): first, use the
proof of Proposition \ref{prop:cImpliesA} to find $(n+1)$ pairwise non-isomorphic simple representations in
$\rep_K(Q'',\cali''_f)$; second, use Proposition \ref{prop:mixingSimples} to extend the direct sum of these $(n+1)$
simple representations to a simple representation $M$ in $\rep_K(Q',\cali'_f)$; finally, further extend $M$ to a
representation in $\repqrf$ where $M_\gamma=\id$ for all $\gamma\in Q_1\setminus Q'_1$. As before, the extended
representation $M\in \repcqr$ must be simple and satisfy $\dim (M)>n$. This completes the proof.
\end{proof}

\bibliographystyle{alpha}\nocite{*}
\bibliography{J_ref.bib}

\begin{thebibliography}{DDPW08}

\bibitem[Alv08]{AlvisH4}
D.~Alvis.
\newblock Subrings of the asymptotic {H}ecke algebra of type {$H_4$}.
\newblock {\em Experiment. Math.}, 17(3):375--383, 2008.

\bibitem[BB81]{BeilinsonBernstein}
A.~Be\u{\i}linson and J.~Bernstein.
\newblock Localisation de {$g$}-modules.
\newblock {\em C. R. Acad. Sci. Paris S\'{e}r. I Math.}, 292(1):15--18, 1981.

\bibitem[BFO09]{BFOIII}
R.~Bezrukavnikov, M.~Finkelberg, and V.~Ostrik.
\newblock On tensor categories attached to cells in affine {W}eyl groups.
  {III}.
\newblock {\em Israel J. Math.}, 170:207--234, 2009.

\bibitem[BK81]{BrylinskiKashiwara}
J.-L. Brylinski and M.~Kashiwara.
\newblock Kazhdan-{L}usztig conjecture and holonomic systems.
\newblock {\em Invent. Math.}, 64(3):387--410, 1981.

\bibitem[BK18]{BravermanKazhdan}
A.~Braverman and D.~Kazhdan.
\newblock Remarks on the asymptotic {H}ecke algebra.
\newblock In {\em Lie groups, geometry, and representation theory}, volume 326
  of {\em Progr. Math.}, pages 91--108. Birkh\"{a}user/Springer, Cham, 2018.

\bibitem[Bon17]{Bonnafe}
C.~Bonnaf\'{e}.
\newblock {\em Kazhdan-{L}usztig cells with unequal parameters}, volume~24 of
  {\em Algebra and Applications}.
\newblock Springer, Cham, 2017.

\bibitem[DDPW08]{DDPW}
B.~Deng, J.~Du, B.~Parshall, and J.~Wang.
\newblock {\em Finite {D}imensional {A}lgebras and {Q}uantum {G}roups}, volume
  150 of {\em Mathematical Surveys and Monographs}.
\newblock American Mathematical Society, Providence, RI, 2008.

\bibitem[DL15]{DiazLopez}
A.~Diaz-Lopez.
\newblock Representations of {H}ecke algebras on quotients of path algebras.
\newblock \url{https://arxiv.org/abs/1509.02403}, 2015.

\bibitem[EW14]{EliasWilliamson}
B.~Elias and G.~Williamson.
\newblock The {H}odge theory of {S}oergel bimodules.
\newblock {\em Ann. of Math. (2)}, 180(3):1089--1136, 2014.

\bibitem[EW16]{SoergelCalculus}
B.~Elias and G.~Williamson.
\newblock Soergel calculus.
\newblock {\em Represent. Theory}, 20:295--374, 2016.

\bibitem[Gec98]{GeckDecomposition}
M.~Geck.
\newblock Kazhdan-{L}usztig cells and decomposition numbers.
\newblock {\em Represent. Theory}, 2:264--277, 1998.

\bibitem[Gec07]{GeckCellular}
M.~Geck.
\newblock Hecke algebras of finite type are cellular.
\newblock {\em Invent. Math.}, 169(3):501--517, 2007.

\bibitem[KL79]{KL}
D.~Kazhdan and G.~Lusztig.
\newblock Representations of {C}oxeter groups and {H}ecke algebras.
\newblock {\em Invent. Math.}, 53(2):165--184, 1979.

\bibitem[KL80]{KL2}
D.~Kazhdan and G.~Lusztig.
\newblock Schubert varieties and {P}oincar\'{e} duality.
\newblock In {\em Geometry of the {L}aplace operator ({P}roc. {S}ympos. {P}ure
  {M}ath., {U}niv. {H}awaii, {H}onolulu, {H}awaii, 1979)}, Proc. Sympos. Pure
  Math., XXXVI, pages 185--203. Amer. Math. Soc., Providence, R.I., 1980.

\bibitem[KP19]{MatrixBall}
D.~Kim and P.~Pylyavskyy.
\newblock Asymptotic {H}ecke algebras and {L}usztig-{V}ogan bijection via
  affine matrix-ball construction.
\newblock \url{https://arxiv.org/abs/1902.06668}, 2019.

\bibitem[Lus83]{LusztigSubregular}
G.~Lusztig.
\newblock Some examples of square integrable representations of semisimple
  {$p$}-adic groups.
\newblock {\em Trans. Amer. Math. Soc.}, 277(2):623--653, 1983.

\bibitem[Lus87]{LusztigCell2}
G.~Lusztig.
\newblock Cells in affine {W}eyl groups. {II}.
\newblock {\em J. Algebra}, 109(2):536--548, 1987.

\bibitem[Lus89]{LusztigCellIV}
G.~Lusztig.
\newblock Cells in affine {W}eyl groups. {IV}.
\newblock {\em J. Fac. Sci. Univ. Tokyo Sect. IA Math.}, 36(2):297--328, 1989.

\bibitem[Lus95]{LusztigLimit}
G.~Lusztig.
\newblock Quantum groups at {$v=\infty$}.
\newblock In {\em Functional analysis on the eve of the 21st century, {V}ol. 1
  ({N}ew {B}runswick, {NJ}, 1993)}, volume 131 of {\em Progr. Math.}, pages
  199--221. Birkh\"{a}user Boston, Boston, MA, 1995.

\bibitem[Lus14a]{LusztigJInvolutions}
G.~Lusztig.
\newblock Asymptotic {H}ecke algebras and involutions.
\newblock In {\em Perspectives in representation theory}, volume 610 of {\em
  Contemp. Math.}, pages 267--278. Amer. Math. Soc., Providence, RI, 2014.

\bibitem[Lus14b]{LusztigHecke}
G.~Lusztig.
\newblock Hecke algebras with unequal parameters.
\newblock \url{https://arxiv.org/abs/math/0208154v2}, 2014.

\bibitem[Lus18]{LusztigSpecialRep}
G.~Lusztig.
\newblock Special representations of {W}eyl groups: a positivity property.
\newblock {\em Adv. Math.}, 327:161--172, 2018.

\bibitem[Mil01]{Miller23}
G.~A. Miller.
\newblock On the groups generated by two operators.
\newblock {\em Bull. Amer. Math. Soc.}, 7(10):424--426, 1901.

\bibitem[Pie10]{Pietraho}
T.~Pietraho.
\newblock Module structure of cells in unequal-parameter {H}ecke algebras.
\newblock {\em Nagoya Math. J.}, 198:23--45, 2010.

\bibitem[Sav05]{SavageQuiver}
A.~Savage.
\newblock Finite-dimensional algebras and quivers.
\newblock \url{https://arxiv.org/abs/math/0505082}, 2005.

\bibitem[Sch14]{Schiffler}
R.~Schiffler.
\newblock {\em Quiver representations}.
\newblock CMS Books in Mathematics/Ouvrages de Math\'{e}matiques de la SMC.
  Springer, Cham, 2014.

\bibitem[Soe90]{Soergel1}
W.~Soergel.
\newblock Kategorie {O}, perverse {G}arben und {M}oduln \"{u}ber den
  {K}oinvarianten zur {W}eylgruppe.
\newblock {\em J. Amer. Math. Soc.}, 3(2):421--445, 1990.

\bibitem[Soe92]{Soergel2}
W.~Soergel.
\newblock The combinatorics of {H}arish-{C}handra bimodules.
\newblock {\em J. Reine Angew. Math.}, 429:49--74, 1992.

\bibitem[Soe07]{Soergel3}
W.~Soergel.
\newblock Kazhdan-{L}usztig-{P}olynome und unzerlegbare {B}imoduln \"{u}ber
  {P}olynomringen.
\newblock {\em J. Inst. Math. Jussieu}, 6(3):501--525, 2007.

\bibitem[Wil18]{WilliamsonHodge}
G.~Williamson.
\newblock The {H}odge theory of the {H}ecke category.
\newblock In {\em European {C}ongress of {M}athematics}, pages 663--683. Eur.
  Math. Soc., Z\"{u}rich, 2018.

\bibitem[Xi02]{Xi}
N.~Xi.
\newblock The based ring of two-sided cells of affine {W}eyl groups of type
  {$\widetilde A_{n-1}$}.
\newblock {\em Mem. Amer. Math. Soc.}, 157(749):xiv+95, 2002.

\bibitem[Xu19]{Xu}
T.~Xu.
\newblock On the {s}ubregular {J}-{r}ings of {C}oxeter {s}ystems.
\newblock {\em Algebr. Represent. Theory}, 22(6):1479--1512, 2019.

\end{thebibliography}
\end{document}